\RequirePackage{fix-cm}
\documentclass[sn-mathphys]{sn-jnl-02}

\usepackage[utf8]{inputenc} 		
\usepackage[T1]{fontenc}    		
\usepackage{tabularx}
\usepackage{longtable}
\usepackage{subfigure}
\usepackage{amssymb}
\usepackage{comment}
\usepackage{enumitem}
\usepackage{footmisc}
\usepackage{ulem}

\colorlet{color1}{blue}
\colorlet{color2}{red!50!black}

\definecolor{ivory}{RGB}{218,215,203}

\definecolor{cuhkp}{RGB}{98,56,105} 	
\definecolor{cuhkpl}{RGB}{152,24,147} 	
\definecolor{cuhkb}{RGB}{219,160,1} 	
\definecolor{cuhkbd}{RGB}{178,129,0} 	
\definecolor{cuhkr}{RGB}{88,35,155}  	

\definecolor{blackp}{RGB}{0,0,0} 
\definecolor{redp}{RGB}{255,0,0}
\definecolor{orangep}{RGB}{255,128,0}
\definecolor{brownp}{RGB}{128,77,0}
\definecolor{yellowp}{RGB}{255,230,0}
\definecolor{greenp}{RGB}{128,230,0}
\definecolor{bluep}{RGB}{0,128,255}
\definecolor{purplep}{RGB}{152,24,147}
\definecolor{pinkp}{RGB}{230,0,128}   
\definecolor{lavender}{rgb}{0.9, 0.9, 0.98}

\DeclareMathOperator*{\argmin}{argmin}

\newcommand{\rmn}[1]{\textup{\textrm{#1}}}

\renewcommand{\R}{\mathbb{R}}
\newcommand{\N}{\mathbb{N}}
\newcommand{\Rn}{\mathbb{R}^n}
\newcommand{\Rex}{(-\infty,\infty]}

\newcommand{\vp}{\varphi}

\newcommand{\dist}{\mathrm{dist}}
\newcommand{\crit}{\mathrm{crit}}

\newcommand{\iprod}[2]{\langle #1, #2 \rangle}

\newcommand{\be}{\begin{equation}}
\newcommand{\ee}{\end{equation}}

\newcommand\prox[1]{\mathrm{prox}_{\lambda\varphi}(#1)}

\newcommand\proxs{\mathrm{prox}_{\lambda\varphi}}

\newcommand\Fnor[1]{F^{\lambda}_{\mathrm{nor}}(#1)}
\newcommand\Fnors{F^{\lambda}_{\mathrm{nor}}}
\newcommand\oFnor{F^{\lambda}_{\mathrm{nor}}}
\newcommand\Fnat[1]{F^{\lambda}_{\mathrm{nat}}(#1)}
\newcommand\oFnat{F^{\lambda}_{\mathrm{nat}}}
\newcommand\dom[1]{\mathrm{dom\\}(#1)}

\newcommand\oprox{\mathrm{prox}_{\lambda\varphi}}

\newcommand\diag{\mathrm{diag}}

\newcolumntype{C}[1]{>{\centering\arraybackslash}p{#1}}

\jyear{2024}%

\theoremstyle{thmstyleone}%

\newtheorem{thm}{Theorem}[section]
\newtheorem{lemma}[thm]{Lemma}

\newtheorem{remark}[thm]{Remark}
\newtheorem{assumption}[thm]{Assumption}
\newtheorem{defn}[thm]{Definition}

\makeatletter
\newcommand\footnoteref[1]{\protected@xdef\@thefnmark{\ref{#1}}\@footnotemark}
\makeatother

\begin{document}

\title[A linesearch-type normal map-based semismooth Newton method]{A linesearch-type normal map-based semismooth Newton method for nonsmooth nonconvex composite optimization}


\author[1]{\fnm{Hanfeng} \sur{Zeng}}\email{hanfengzeng@link.cuhk.edu.cn}

\author[2]{\fnm{Wenqing} \sur{Ouyang}}\email{wo2205@columbia.edu}

\author*[1]{\fnm{Andre} \sur{Milzarek}}\email{andremilzarek@cuhk.edu.cn}
 
\affil[1]{\orgdiv{School of Data Science (SDS)},  \orgname{The Chinese University of Hong Kong, Shenzhen}, \city{Shenzhen}, \state{Guangdong}, \country{China}} 
\affil[2]{\orgdiv{IEOR},  \orgname{Columbia University}, \city{New York}, \country{USA}}


\abstract{We propose a novel linesearch variant of the trust region normal map-based semismooth Newton method developed in [Ouyang and Milzarek, Math. Program. \textbf{212}(1-2), 389--435 (2025)] for solving a class of nonsmooth, nonconvex composite-type optimization problems. Our approach uses adaptive parameter estimation techniques, which allow us to avoid explicit and potentially expensive Lipschitz constant computations. We provide extensive convergence results including global convergence, convergence of the iterates under the Kurdyka-{\L}ojasiewicz inequality, and transition to fast local q-superlinear convergence. Compared to the original trust region framework, the linesearch-based algorithm is simpler and the overall convergence analysis can be conducted under weaker assumptions---in particular, without requiring explicit boundedness conditions on the Hessian approximations and iterates. Numerical experiments on sparse logistic regression, image compression, and nonlinear least squares with group penalty terms demonstrate the efficiency of the proposed approach.} \vspace{-1ex}





\maketitle


\section{Introduction}
In this work, we consider the composite optimization problem
\begin{align} \label{eq1-1}
\min\limits_{x\in\mathbb{R}^n}~\psi(x):=f(x)+\varphi(x),
\end{align}
where $f:\mathbb{R}^n\rightarrow\mathbb{R}$ is a continuously differentiable (not necessarily convex) function and $\varphi:\mathbb{R}^n\rightarrow(-\infty,\infty]$ is a convex, lower semicontinuous (lsc), and proper (not necessarily smooth) mapping. Problems of this form have widespread applications, including sparse $\ell_1$-regularized problems \cite{tibshirani1996regression,shevade2003simple,donoho2006compressed}, group sparse problems \cite{cotter2005sparse,yuan2006model,meier2008group}, structured dictionary learning \cite{mairal2009online,bach2011optimization}, matrix completion \cite{candes2009exact,cai2010singular}, and machine learning tasks \cite{Bis06,shalev2014understanding,BotCurNoc18} among many others.

Many modern algorithms for solving \eqref{eq1-1} are based on forward-backward splitting techniques and proximal steps \cite{FukMin81, chen1997convergence, ComWaj05, parikh2014proximal}. Let $\oprox: \Rn \to \Rn$, $\prox{x}:= \argmin_{y \in \Rn} \varphi(y) + \frac{1}{2\lambda} \|x-y\|^2$, $\lambda > 0$, denote the proximity operator of $\varphi$ \cite{Mor65}. The classical proximal gradient method, \cite{FukMin81}, then applies a gradient descent step for the smooth function $f$ followed by a proximal ``backward'' step for the nonsmooth convex mapping $\vp$:
\be \label{eq:fbs-intro} x_{k+1} = \prox{x_k - \lambda\nabla f(x_k)}. \ee
The forward-backward scheme in \eqref{eq:fbs-intro} can also be interpreted as a fixed-point procedure using the \textit{natural residual} 
%
\begin{equation*} \oFnat : \Rn \to \Rn, \quad \Fnat{x} := x - \prox{x - \lambda\nabla f(x)}. \label{eq:natural-residual} \end{equation*}
Furthermore, the fixed-point iteration \eqref{eq:fbs-intro} and the natural residual $\oFnat$ can be connected to Robinson's \textit{normal map}:
%
\[ \oFnor : \Rn \to \Rn, \quad \Fnor{z}:=\nabla f(\prox{z})+{\lambda^{-1}}(z-\prox{z}). \label{eq:normal-map} \]
The natural residual $\oFnat$ and normal map $\oFnor$ offer different representations of the underlying first-order necessary optimality conditions of problem \eqref{eq1-1} (cf. Section~\ref{ssec:pre}). The normal map was initially introduced by Robinson \cite{robinson1992normal} and has been primarily applied in the context of classical variational inequalities and generalized equations, particularly when $\oprox$ reduces to the projection $\mathcal P_C$ onto a closed convex set $C$. We refer to \cite{facchinei2007finite} for additional background.

The main goal of this work is to study and develop a semismooth Newton method \cite{QiSun93,qi1993convergence} with linesearch globalization for solving the nonsmooth equation
\be \label{eq:intro-fnor} \oFnor(z) = 0. \ee
Specifically, we perform inexact semismooth Newton steps for \eqref{eq:intro-fnor}, generated via the conjugate gradient method, and use a linesearch-based globalization on a suitable merit function for \eqref{eq1-1}  to ensure sufficient decrease of the Newton steps. The resulting second-order algorithm can then be shown to converge globally and locally at a q-superlinear rate.

Semismooth Newton methods, \cite{QiSun93,qi1993convergence,QiSun99,facchinei2007finite}, represent an important and successful class of algorithms for solving nonlinear, nonsmooth equations. While the local convergence properties of semismooth Newton methods are well understood, global convergence results and globalization strategies are more intricate and specific mechanisms often depend on the underlying structure of the problem. Existing globalization approaches for general nonsmooth equations, $F(x) = 0$, include linesearch-based globalization schemes \cite{HanPanRan92,qi1993convergence,MarQi95} (on suitable merit functions, such as, e.g., $\frac12\|F(x)\|^2$), specialized techniques for complementarity problems and KKT systems \cite{DeLFacKan96,KanQi99,MunFacFerFisKan01}, projection methods for monotone equations \cite{SolSva01,xiao2018regularized}, and trust region-type globalizations \cite{ulbrich2001nonmonotone}. 

There is also a large amount of works on semismooth Newton methods specialized to \eqref{eq1-1} and to the natural residual $\oFnat$. For sparse problems with $\varphi \equiv \mu \|\cdot\|_1$, $\mu > 0$, various semismooth Newton-type methodologies have been analyzed in \cite{GriLor08,milzarek2014semismooth,HanRaa15,ByrChiNocOzt16}. In \cite{PatBem13,SteThePat17}, the authors propose the forward-backward envelope (FBE) as a smooth merit function for the composite problem \eqref{eq1-1}. Based on this envelope, different semismooth Newton methods with linesearch-type globalization are developed and investigated. In \cite{gfrerer2025globally}, a recent FBE-based semismooth Newton method is proposed with global convergence guarantees under local Lipschitz continuity of $\nabla f$. In \cite{milzarek2016numerical}, a multidimensional filter globalization scheme for the semismooth Newton method is studied, which extends the analysis in \cite{milzarek2014semismooth} to the general composite-type setting \eqref{eq1-1}. 

Proximal Newton methods constitute another popular class of second-order approaches for \eqref{eq1-1}. Lee et al. \cite{lee2014proximal} propose the first generic version of the exact proximal Newton method when $f$ is convex. Several variants have been developed for both convex and nonconvex cases, see, e.g., \cite{byrd2016inexact,yue2019family,lee2019inexact, kanzow2021globalized,mordukhovich2023globally}. Recently, Liu et al. \cite{liu2024inexact} propose an inexact proximal Newton method with Hessian regularization using a linesearch framework, where $f$ can be nonconvex. Inspired by \cite{liu2024inexact}, vom Dahl and Kanzow \cite{vom2024inexact} develop a variant that can avoid linesearch by dynamically adjusting the regularization parameter. Notably, this allows the authors to achieve convergence without requiring global Lipschitz continuity of the gradient mapping $\nabla f$.

Although proximal updates \eqref{eq:fbs-intro} and the natural residual $\oFnat$ have been widely adopted in the design of solution methodologies for \eqref{eq1-1}, the normal map $\oFnor$ has generally received less attention. 
To the best of our knowledge, Pieper's PhD thesis \cite{pieper2015finite} is one of the first to consider the normal map $\oFnor$ within a trust region framework and in an optimal control setting; see also \cite{kunisch2016time,boulanger2017sparse,RunAigKunSto18a} for related applications. Mannel and Rund \cite{ManRun20,mannelhybrid} analyze the local properties of a quasi-Newton variant of Pieper's algorithm with Broyden-like updates. However, the trust region globalizations utilized in the previous works \cite{pieper2015finite,kunisch2016time,boulanger2017sparse,RunAigKunSto18a,ManRun20,mannelhybrid} generally do not ensure that accumulation points of the generated iterates correspond to stationary points of \eqref{eq1-1}. To resolve this issue, Ouyang and Milzarek \cite{ouyang2025trust} propose a novel merit function and an adjusted trust region acceptance mechanism. This modification allows the authors to establish comprehensive convergence results, including global convergence (i.e., $\|\Fnor{z_k}\| \to 0$ as $k \to \infty$), convergence of the iterates under the celebrated Kurdyka-{\L}ojasiewicz (KL) property and transition to locally fast q-superlinear convergence. More recently, stochastic normal map approaches have also been developed; see \cite{milzarek2023convergence,qiu2023new}. For more literature on the use of the normal map in variational inequalities and generalized equations, we refer to \cite[Section 1.2]{ouyang2025trust}.  

Normal map-based approaches can offer certain advantages over natural residual-based methodologies. First, since the range of the proximity operator $\oprox$ is a subset of $\dom{{\varphi}}$, the normal map remains well-defined if $\nabla f$ is only defined on the effective domain $\dom{\vp}$. This is a distinctive feature of the normal map and a key motivation for developing normal map-based algorithms. 
In addition, the normal map-based semismooth Newton equation for \eqref{eq:intro-fnor} can be naturally expressed as a symmetric linear system. This linear system can then be solved using standard linear system solvers; we refer to Section~\ref{ssec:semistep} for further details. For natural residual-based approaches, such symmetric transformations might not always be possible and heavily depend on the intrinsic structure of the generalized derivative.

\subsection{Motivation and Contribution}
Motivated by the success and general advantages of the normal map, this work aims to address the following two key questions and observations. First, the existing normal map-based approaches for solving \eqref{eq1-1} all rely on trust region globalization techniques. To the best of our knowledge, a \textit{normal map-based semismooth Newton method with linesearch globalization} does not seem to exist so far. Second, the trust region algorithm in \cite{ouyang2025trust} requires explicit access to the Lipschitz constant of $\nabla f$. This can be costly and intractable for large-scale problems. At the moment, it seems unclear \textit{how to avoid the computation of the Lipschitz constant} (to further enhance the efficiency of the algorithm). We now summarize our core contributions:
\begin{itemize}
  \item We propose a novel linesearch-type normal map-based semismooth Newton method for solving composite problems of the form \eqref{eq1-1}. Inspired by the trust region approach in \cite{ouyang2025trust}, we carefully design a linesearch strategy that allows us to maintain all the theoretical properties of the trust region method \cite{ouyang2025trust} (including comprehensive global and local convergence guarantees) while providing a more concise framework. 
  \item We develop and incorporate an adaptive scheme to estimate the Lipschitz constant of the gradient mapping $\nabla f$ in our linesearch globalization.
  Such adaptive strategies are not uncommon and have been the recent focus in the development of \textit{universal methods} \cite{nesterov2015universal,marumo2024parameter,marumo2024universal}.
  \item Finally, we conduct numerical experiments on sparse logistic regression, nonconvex image compression, and a nonlinear least squares problem with group Lasso penalty, demonstrating the favorable performance of the linesearch-type normal map-based semismooth Newton method. 
\end{itemize}

\subsection{Notation}
By $\iprod{\cdot}{\cdot}$ and $\|\cdot\|$, we denote the standard Euclidean inner product and norm, respectively. For matrices, the norm $\|\cdot \|$ is the spectral norm. The set of symmetric $n \times n$ matrices is denoted by $\mathbb S^n$.
The effective domain of a function $\theta : \Rn \to \Rex$ is defined as $\dom{\theta}=\{x \in \Rn : \theta(x)<\infty\}$. The set $\partial \theta$ denotes Clarke's subdifferential for extended-valued functions or for locally Lipschitz continuous mappings $\theta: \Rn \to \R^m$, see \cite[Section 8.J]{rockafellar2009variational} or \cite[Section 2.4]{clarke1990optimization}. Throughout this work, we assume that $f$ is continuously differentiable and $\varphi:\mathbb{R}^n\rightarrow(-\infty,\infty]$ is a convex, lsc, and proper mapping. In Table~\ref{table:notation}, we summarize some frequently used symbols and notation.

\begin{table}[t]
\centering
\begin{tabular}{p{1.6cm}p{3cm}p{6cm}}  
\cmidrule[1pt](){1-3} 
Notation & \multicolumn{1}{c}{Description} & \multicolumn{1}{c}{Formula} \\[0.5ex]  
\cmidrule(){1-3} \\[-2.5ex]
$\oFnor$ & normal map & $\oFnor(z) := \nabla f(\prox{z}) + \frac{1}{\lambda}(z-\prox{z})$  \\[0.5ex]
$H$ & merit function & $H(\tau, z) := \psi(\oprox(z)) + \frac{\tau\lambda}{2}\|\oFnor(z)\|^2$ \\[0.5ex]
$\chi$ & criticality measure & $\chi(z) := \|\oFnor(z)\|$ \\[0.5ex]
$M$ & gen. derivative of $\oFnor$ & $M:=BD+(I-D)/\lambda$  \\[0.5ex]
$d_k$; $e_k$ & descent directions & $d_k := -\Fnor{z_k}$; $e_k:=q_k/\lambda - M_k q_k$ \\[0.5ex]
\cmidrule[1pt](){1-3} \\[-1.5ex]
\end{tabular}
\vspace{-2ex}

\begin{tabular}{p{1.6cm}p{9.43cm}}  
\cmidrule[1pt](){1-2} 
Notation & \multicolumn{1}{c}{Description} \\[0.5ex]  
\cmidrule(){1-2} \\[-2.5ex]
$B$; $D$ & (Approximation of) $\nabla^2 f$; gen. derivative of $\oprox$ \\[0.5ex]
$q_k$ & approximate solution of \eqref{eq3-2}---returned by the CG-method \\[0.5ex]
$\alpha$; $s_k(\alpha)$ & stepsize; linesearch direction at the $k$-th iteration \\[0.5ex]
$L_k$ & estimate of the Lipschitz constant (of $\nabla f$) at the $k$-th iteration \\[0.5ex]
$\eta_k$ & parameter in the gradient-related test \eqref{eq:cur} \\[0.5ex]
$\tau_k$; $\sigma$; $\rho$; $\nu_k$ & parameters in the linesearch condition \eqref{eq:armijo} \\[0.5ex]
\cmidrule[1pt](){1-2} \\[-1.5ex]
\end{tabular}
\caption{List of variables, parameters, and functions.}
\label{table:notation}
\end{table}

\subsection{Organization}
The rest of this paper is organized as follows.  In Section~\ref{sec:framework}, we first introduce some required preliminaries and then present our algorithmic framework. In Section~\ref{sec:convergence}, we analyze key convergence properties. In Section~\ref{sec:numerical}, we illustrate and discuss the numerical performance of our algorithm. Section~\ref{sec:conclusion} concludes this paper and we mention possible future directions.


\section{Algorithm Framework}\label{sec:framework}
In this section, we first introduce the first-order optimality conditions of the problem \eqref{eq1-1}. We then present the two main components of our algorithm, i.e., the semismooth Newton step and the linesearch globalization in Section~\ref{ssec:semistep} and \ref{ssec:linesearch}, respectively. The full algorithm framework is described in Algorithm~\ref{algo:main}.

\subsection{First-order Optimality and Preliminaries}\label{ssec:pre}
By $\mathrm{crit}(\psi) := \{x \in \Rn: 0 \in \partial \psi(x) = \nabla f(x) + \partial \vp(x)\}$, we denote the set of all stationary points of $\psi$. 
Here, $\partial\vp$ is the standard subdifferential of the convex function $\vp$. 
The condition $\bar x \in \mathrm{crit}(\psi)$ can be characterized as follows: 
\[ \bar x \in \mathrm{crit}(\psi) \quad \iff \quad \Fnat{\bar x } = \bar x  - \prox{\bar x  - \lambda \nabla f(\bar x )} = 0, \quad \lambda > 0. \]
In addition, if $\bar x$ is a stationary point of $\psi$, then $\bar z := \bar x -\lambda\nabla f(\bar x)$ is a zero of the normal map $\Fnor{\bar z} = 0$. Conversely, if $\bar z$ is a zero of the normal map, then we have $\bar x := \prox{\bar z} \in \crit(\psi)$, i.e., $\bar x$ is a stationary point.
We refer to \cite[Lemma 2.1]{ouyang2025trust} for further details. 

Following \cite{facchinei2007finite,ulbrich2011semismooth}, a mapping $F : \Rn \to \Rn$ is  semismooth at $x$ with respect to $\mathcal M : \Rn \rightrightarrows \R^{n\times n}$ if $F$ is Lipschitz continuous in a neighborhood of $x$, directionally differentiable at $x$, and 
\begin{equation*}
{\sup}_{M\in \mathcal M(x+h)}~\| F(x+h)-F(x)-Mh  \| = o(\|h\|) \quad \text{as} \quad h\rightarrow 0\footnote{We note that, in \cite{facchinei2007finite}, such a function $F$ is said to have a Newton approximation (with $\mathcal M$ being such an approximation).}.
\end{equation*}
The set-valued mapping $\mathcal M$ can be different from Clarke's subdifferential $\partial F$. If $f$ is twice continuously differentiable around $\prox{z}$ and $\proxs$ is semismooth at $z$, then $\oFnor$ is semismooth at $z$ with respect to $\mathcal M^\lambda(z) := \{M = \nabla^2f(\prox{z})D + (I-D)/\lambda: D \in \partial\prox{z}\}$, see \cite[Lemma 2.2]{ouyang2025trust}. Finally, we note that the matrices $D$ and $I-D$, $D \in \partial\prox{z}$, are symmetric and positive semidefinite for all $z \in \Rn$, cf. \cite[Lemma 3.3.5]{milzarek2016numerical}. 

\subsection{Semismooth Newton Step}\label{ssec:semistep}
To solve the nonsmooth equation $\oFnor(z)=0$, we consider semismooth Newton steps of the form:
\begin{align}\label{eq3-1}
    M_ks_k=-\Fnor{z_k}, \quad z_{k+1}=z_k+s_k, \quad M_k : = B_k D_k + (I-D_k)/\lambda,
    \end{align}
where $D_k \in \partial \prox{z_k}$ and $B_k$ is a symmetric matrix approximating the Hessian $ \nabla^2 f(\prox{z_k})$. 
Leveraging the symmetry of $D_k$ and $I-D_k$, we can modify the (typically non-symmetric) linear equation \eqref{eq3-1} by multiplying it with $D_k$
from the left. This leads to the related symmetric linear system:
\begin{align}\label{eq3-2}
D_k M_k \tilde q=-D_k\Fnor{z_k}.
\end{align}
Such modification has two advantages: First, since $D_kM_k$ is symmetric, we can use standard approaches to solve \eqref{eq3-2} inexactly, such as the CG method \cite{dembo1983truncated}. (See also Appendix~\ref{appen:prooflemma} for an explicit reference to the CG method used in our approach). Furthermore, it is often possible to exploit the structure of the generalized derivative $ D_k$ to reduce the dimension of the linear system \eqref{eq3-2}. 
%

%
%

In the next lemma, we describe the mentioned advantages in more detail. We refer to Appendix~\ref{appen:prooflemma} for a proof of Lemma~\ref{lemma3-8}.
\begin{lemma}
    \label{lemma3-8} Let $M = BD+(I-D)/\lambda$ be given and let $B, D \in \R^{n \times n}$ be symmetric. Suppose that the CG method \cite{dembo1983truncated} is run with tolerance parameter $\epsilon \geq 0$ to solve $D M \tilde q= -D\Fnor{z}$ and define $m = \mathrm{dim}~\mathrm{range}(D)$. Then, it holds that:
    \begin{itemize} 
    \item[\rmn{(i)}] The CG method stops after at most $m\leq n$ iterations.
    \item[\rmn{(ii)}] Assume that $D M$ is positive semidefinite and $M$ is invertible. The CG method then returns $\tilde q$ with $\|D (M \tilde q + \Fnor{z})\| \leq \epsilon$.
    \item[\rmn{(iii)}] Assume $\tilde q$ satisfies the condition $\|D(M \tilde q + \oFnor(z))\| \leq \epsilon$. Then, setting $s = \tilde q - \lambda (M \tilde q +  \oFnor(z))$, it follows $\|Ms +  \oFnor(z)\| \leq \|I-\lambda B\| \epsilon$. 
    \end{itemize}
\end{lemma}

Considering the linear system $DM\tilde q = -D\Fnor{z}$ and as discussed in \cite[Section 2]{dembo1983truncated}, the CG method either stops if non-positive curvature, $\iprod{p_i}{DMp_i} \leq 0$, is detected or if the tolerance condition $\|D(M q_i + \oFnor(z))\| \leq \epsilon$ is satisfied at the CG iterates $p_i, q_i$, cf.\ Algorithm~\ref{algo:CG} in Appendix~\ref{appen:prooflemma}. Lemma~\ref{lemma3-8}~(i) states that such a termination will occur within at most $m = \mathrm{dim}~\mathrm{range}(D)$ iterations. Moreover, Lemma~\ref{lemma3-8}~(ii) emphasizes one situation ensuring that the linear system is consistent and the CG method stops with the tolerance condition being satisfied. Lemma~\ref{lemma3-8} (iii) implies that solutions to the original linear system \eqref{eq3-1} can be computed by solving the (potentially reduced) symmetric system \eqref{eq3-2}.

We note that the natural residual-based equation $\Fnat{x} = 0$ can be solved in a similar fashion yielding the Newton system $\widetilde M_k s_k = -\Fnat{x_k}$, where $ \widetilde M_k=I+D_k(I-\lambda B_k)$ and $D_k \in \partial \prox{x_k-\lambda \nabla f(x_k)}$. In contrast to the normal map-based perspective adopted in this work, this system cannot be directly converted into a symmetric linear system. In particular, an analogue of Lemma~\ref{lemma3-8} (iii) does not seem to be available in this case.

\subsection{Linesearch Globalization}\label{ssec:linesearch}
The key novelty of our linesearch framework lies in a gradient-related test and a special choice of the search directions. 
Specifically, at the $k$-th iteration, let $q_k$ be an approximate solution of the subproblem \eqref{eq3-2} returned by the CG method and set
$$d_k=-F_{\text{nor}}^{\lambda}(z_k),\quad  e_k= q_k/\lambda -M_k q_k.$$
Invoking Lemma~\ref{lemma3-8} (iii), the direction $\lambda (d_k + e_k)$ is an approximate solution of the original linear system \eqref{eq3-1} that  contains second-order information (from $M_k$). We can decompose $\lambda (d_k + e_k)$ into the first-order direction $\lambda d_k$ and the remainder term $\lambda e_k$. We then conduct the following gradient-related test:
\begin{equation}\label{eq:cur}
    \|e_k\| \le \|\Fnor{z_k}\|/\eta_k,
\end{equation}
where $\eta_k = \min\{ b_k\|\Fnor{z_k}\|^q, \eta\}$ with $b_k = (k \ln^2(k+1))^{q}$ and $q,\eta>0$. We set $\texttt{flag}= \texttt{SO}$ if it is passed and $\texttt{flag}= \texttt{FO}$ otherwise.
The special choice of $\eta_k$ helps to ensure the transition from global to local convergence. Based on \texttt{flag}, we then select the linesearch direction as follows
\begin{equation*} 
    s_k(\alpha) = \begin{cases}
        \alpha \lambda d_k & \text{if $\texttt{flag}= \texttt{FO}$},\\  \alpha \lambda (d_k + \alpha  e_k) & \text{if $\texttt{flag}= \texttt{SO}$},\\
    \end{cases}
\end{equation*}
where $\alpha$ is the stepsize that will be determined in the linesearch procedure. Notably, we scale $e_k$ with $\alpha^2$ when $\texttt{flag}= \texttt{SO}$, which can be interpreted as a \textit{second-order damped backtracking}. The gradient-related test and the special choice of $s_k(\alpha)$ allow us to recover all of the theoretical convergence results in \cite{ouyang2025trust} with a simpler analysis and under fewer assumptions. Finally, let us note that tests akin to \eqref{eq:cur} have already appeared (as assumptions) in other nonsmooth second-order approaches; cf. \cite[Sections 3.2 and 4.1]{SteThePat17}.

We now introduce the backtracking linesearch. Let us set $x_k =\prox{z_k}$ and $p_\alpha=\prox{z_k + s_k(\alpha)}$. At every inner $t$-th iteration ($t$ is initialized at $t=0$), we 
compute $\alpha=\rho^t$,
\begin{equation}\label{eq:Lk}
    \begin{aligned}
        L_k^{(t)} &= \begin{cases}
               \max\{\frac{2U_\alpha}{V_\alpha^2} , \frac{W_\alpha}{V_\alpha} \} &  \text{if } V_\alpha \neq 0,\\
               \bar L & \text{otherwise},
        \end{cases} \quad \tau_k^{(t)} = \min\Big\{  \frac{2\gamma(1-\nu_k)}{(L_k^{(t)})^2 \lambda^2 +2}, \tau_{k-1}\Big\}.
    \end{aligned}
\end{equation}
Here, we set $\bar L>0$ (to ensure the well-definedness of $L_k^{(t)}$ when $V_\alpha =0$), $0<\nu_k \le \nu$,  $\gamma, \rho, \nu \in (0,1)$, and 
\begin{equation*}
    \begin{aligned}
       U_\alpha &= f(p_\alpha)-f(x_k)-\iprod{\nabla f(x_k)}{p_\alpha-x_k},\\
       V_\alpha &= \|p_\alpha -x_k\|,\quad W_\alpha = \|\nabla f(p_\alpha) - \nabla f(x_k)\|.
    \end{aligned}
\end{equation*}
This estimation strategy is inspired by the heuristic adaptive scheme developed in \cite[Section 7.3]{ouyang2025trust}. Note that no theoretical results are provided in \cite{ouyang2025trust} for this mechanism. In this work, we fully incorporate \eqref{eq:Lk} in our algorithmic scheme and establish comprehensive theoretical guarantees for the proposed strategy.
We then check if the stepsize $\alpha$ satisfies the Armijo-type condition 
\begin{equation}\label{eq:armijo}
    H(\tau_k^{(t)},z_k + s_k(\alpha)) - H(\tau_k^{(t)}, z_k) \le -\frac{\sigma \lambda \tau_k^{(t)} \alpha}{2}\|\Fnor{z_k}\|^2- \frac{\nu_k}{\lambda \alpha} \|p_\alpha - x_k\|^2,
\end{equation}
where $\sigma \in (0,1)$ and
\begin{equation*} 
    H(\tau, z):=\psi(\prox{z})+\frac{\tau\lambda}{2}\|\Fnor{z}\|^2.
\end{equation*}
We return $\alpha_k=\alpha$ and $\tau_k=\tau_k^{(t)}$ if \eqref{eq:armijo} is passed and move to the next inner iteration if \eqref{eq:armijo} fails to hold.
The merit function $H(\tau,z)$ is adapted from \cite[Definition 3.2]{ouyang2025trust}. In \cite{ouyang2025trust}, the parameter $\tau$ is a predefined constant that is related to the global Lipschitz constant $L$ of $\nabla f$. By contrast, we estimate $L_k^{(t)}$ and $\tau_k^{(t)}$ in an adaptive way. In Section~\ref{ssec:local}, we show that both the gradient-related test \eqref{eq:cur} and the Armijo-type condition \eqref{eq:armijo} will be passed eventually with stepsize $\alpha_k = 1$ under suitable local assumptions. As a result, the step $s_k(\alpha_k) = s_k(1)$ will coincide with the second-order direction $\lambda(d_k+e_k)$ and the proposed algorithm locally reduces to a truncated semismooth Newton method. 

The following lemma states a descent property of the merit function $H(\tau,z)$, which will be frequently used in our analysis.

\begin{lemma}\label{lemma:des}
    Let $z, e \in \Rn$, $\alpha \in (0,1]$ and $\gamma, \nu \in (0,1)$ be given and set $d:=-\Fnor{z}$ and $x := \prox{z}$. Suppose there exists $L = L(\alpha) >0$ such that
    \begin{align*}
        f(p_\alpha) \le f(x) + \iprod{\nabla f(x)}{p_\alpha-x} + \frac{L}{2}\|p_\alpha-x\|^2,\quad  \|\nabla f(p_\alpha) -\nabla f(x) \| \le L\|p_\alpha-x\|,
    \end{align*}
    where $p_\alpha = \prox{z+\alpha\lambda(d+\alpha e)}$. If $\tau \leq\frac{2\gamma (1-\nu)}{{L}^2 {\lambda}^2 +2}$ and $\alpha \leq \min\{1, \frac{2(1-\nu)(1-\gamma)}{(1+2\tau)L\lambda+\tau}\}$, it then holds that
        \begin{align*} 
            H(\tau, z+\alpha\lambda (d + \alpha e))-H(\tau, z) & \leq -\frac{\tau\lambda\alpha}{2} \|\Fnor{z}\|^2 - \frac{\nu}{\lambda\alpha} \|p_\alpha - x\|^2 + \frac{\tau\lambda\alpha}{2} \|\alpha e\|^2 \\ & \hspace{4ex} +\left[L\tau\lambda  + 1 - \tau \right] \lambda\alpha \|\alpha e\| \|d+\alpha e\|. 
            \end{align*} 
\end{lemma}
\begin{proof}
The assumptions on $\tau$ and $\alpha$ ensure that the quantity $C(\alpha)$ in \cite[Lemma 4.3]{ouyang2025trust} is non-positive. The assertation now follows from \cite[Lemma 4.3]{ouyang2025trust} by replacing $e$ with $\alpha e$.
\end{proof}

We end this section by comparing our mechanism to generate $s_k(\alpha)$ and the well-known dogleg procedure for the trust region method, cf.\ \cite[Chapter 4.1]{NocWri06}. The dogleg method chooses a direction $\tilde s$ by minimizing a (quadratic) model of the objective function along the path $\tau \mapsto \tilde s(\tau)$, 
\begin{equation*}
	   \tilde{s}(\tau) := \begin{cases}\tau p^{U} & 0 \leq \tau \leq 1, \\ p^{U}+(\tau-1)\left(p^{B}-p^{U}\right) & 1 \leq \tau \leq 2, \end{cases}
\end{equation*}
subject to the trust-region constraints. Here, $p^{U}$ and $p^{B}$ represent first- and second-order directions, where $p^{U}=-\frac{g^{\top}g}{g^{\top}Bg}g$ and $g$ and $B$ denote the respective gradient and (approximate) Hessian information.
In fact, $\lambda e_k$ and $\alpha$ in $s_k(\alpha)$ play a similar role to $p^{B}-p^{U}$ and $\tau-1$. However, the overall scaling in $s_k(\alpha)$ in terms of the stepsize $\alpha$ appears to be different and we use a linesearch globalization to determine $\alpha$. We refer to \cite{NocWri06,steihaug1983conjugate} for further details.

\begin{algorithm}[t]
    \caption{A Linesearch Normal Map Semismooth Newton Method}
    \label{algo:main}
    \begin{algorithmic}[1]
        \Require Select $z_0\in\mathbb{R}^n$, $B_0 \in \mathbb{S}^n$, $\bar L, \lambda, \tau_{-1} > 0$, $\sigma, \rho, \gamma \in (0,1)$ and $\{\epsilon_k\} \subset \R_{+}$. Set $k=0$. 
        \While{$F_{\text{nor}}^{\lambda}(z_k)\neq0$}
        \State Choose $D_k\in\partial\text{prox}_{\lambda\varphi}(z_k)$ and set $M_k=B_kD_k+\frac{1}{\lambda}(I-D_k)$.
        \State Run CG to solve \eqref{eq3-2} with tolerance $\epsilon = \epsilon_k \geq 0$ returning $q_k$;
        \State Compute $d_k=- F_{\text{nor}}^{\lambda}(z_k)$, $e_k=q_k/\lambda- M_k q_k$ and choose \texttt{flag} as in \eqref{eq:cur};
        \State Choose the maximal stepsize  $\alpha_k=\{1,\rho,\rho^2, \dots\}$ using Algorithm~\ref{algo:ls};
        \State Set $z_{k+1} = z_k + s_k(\alpha_k)$ and choose $B_{k+1} \approx \nabla^2 f(\prox{z_{k+1}})$;
        \State $k\gets k+1$;
        \EndWhile
    \end{algorithmic}
\end{algorithm}

\begin{algorithm}[t]
    \caption{Backtracking Linesearch}
    \label{algo:ls}
    \begin{algorithmic}[1]
        \Require $z_k, d_k, e_k\in\mathbb{R}^n$,  $\bar L, \lambda, \tau_{k-1}>0$, $\sigma, \rho, \gamma\in (0,1)$. Set $t=0$. 
        \For{$t=0,1, \dots$}
        \State Compute $\alpha$, $L_k^{(t)}$, $\tau_{k}^{(t)}$ as in \eqref{eq:Lk}.
        \If{\eqref{eq:armijo}  holds for $\alpha$}
        \State \Return $\alpha_k =\alpha$, $\tau_k = \tau_k^{(t)}$, $L_k = L_k^{(t)}$.
        \EndIf
        \EndFor
    \end{algorithmic}
\end{algorithm}


\section{Convergence Analysis}\label{sec:convergence}
In this section, we present global convergence results for Algorithm~\ref{algo:main} and
convergence of the iterates under the Kurdyka-{\L}ojasiewicz (KL) property. In addition, we derive the locally fast, q-superlinear rate of convergence. Our assumptions and proof techniques are inspired by the analysis in \cite{ouyang2025trust} and in the extended manuscript \cite{OuyMil21}. 
For ease of exposition, let us define $\chi(z) := \|\Fnor{z}\|$ and recall that $x_k=\prox{z_k}$.

\subsection{Global Convergence} \label{ssec:global}
Based on the linesearch mechanism, we derive a sufficient condition for \eqref{eq:armijo}.

\begin{lemma}\label{lemma:welldef}
Let $\{z_k\}$ be generated by Algorithm~\ref{algo:main} and consider an iteration $k \in \N$ with $\chi(z_k) \neq 0$. Then, \eqref{eq:armijo} is satisfied if, at some $t$-th inner iteration during the linesearch, it holds that
\[    \alpha_k \le \min\left\{1, \frac{2(1-\gamma)(1-\nu_k)}{(1+2\tau_k^{(t)})L_k^{(t)}\lambda +\tau_k^{(t)}}, \sqrt{\frac{1-\sigma}{2}}\eta_k, \frac{(1-\sigma)\eta_k\tau_k^{(t)}}{4(L_k^{(t)} \tau_k^{(t)} \lambda +1)(1+1/\eta_k)}\right\}. \]
\end{lemma}
\begin{proof}
    Due to the specific ways of how $L_k^{(t)}$ and $\tau_k^{(t)}$ are computed in \eqref{eq:Lk}, Lemma~\ref{lemma:des} is applicable. In particular, if $\alpha \le \bar \alpha_k^{(t)} := \min\{1,\frac{2(1-\gamma)(1-\nu_k)}{(1+2\tau_k^{(t)})L_k^{(t)}\lambda+\tau_k^{(t)}}\}$, it holds that
        \begin{align}
            \nonumber H(\tau_k^{(t)}, z_k+\alpha \lambda (d_k +\alpha e_k))-H(\tau_k^{(t)}, z_k) & \leq -\frac{\tau_k^{(t)}\lambda\alpha}{2} \chi(z_k)^2 - \frac{\nu_k}{\lambda\alpha} \|p_\alpha - x_k\|^2 \\  & \hspace{-32ex}+\frac{\tau_k^{(t)} \lambda\alpha}{2} \|\alpha e_k\|^2 +\Big[L_k^{(t)}\tau_k^{(t)}\lambda  + 1 -\tau_k^{(t)}\Big] \lambda\alpha \|\alpha e_k\| \|d_k+\alpha e_k\|, \label{eq:des2}
            \end{align}
    where $x_k=\prox{z_k}$. We now discuss two cases. First, if $\texttt{flag} = \texttt{FO}$, then \eqref{eq:armijo} naturally holds by setting $e_k=0$ in \eqref{eq:des2} and noting that $\sigma \in (0,1)$. Second, in the case $\texttt{flag} = \texttt{SO}$, we further impose $\alpha \le c_k^{(t)}$ where 
    \begin{align}\label{eq:ckt}
        c_k^{(t)}:=\min\Big\{\sqrt{\frac{1-\sigma}{2}}\eta_k, \frac{(1-\sigma)\eta_k \tau_k^{(t)} }{4(L_k^{(t)} \tau_k^{(t)} \lambda +1)(1+1/\eta_k)},1\Big\},
    \end{align}
    Combining the gradient-related test \eqref{eq:cur}, $\alpha \le 1$, and $\tau_k^{(t)} \ge 0$, this then yields
    \begin{align}\frac{\tau_k^{(t)} \lambda \alpha}{2} \| \alpha e_k\|^2  & \overset{\eqref{eq:ckt}}{\le} \frac{\tau_k^{(t)} \lambda \alpha}{2}  \frac{1-\sigma}{2} \eta_k^2 \|e_k\|^2 \overset{\eqref{eq:cur}}{\le} \frac{\tau_k^{(t)} \lambda \alpha}{2} \frac{1-\sigma}{2}\chi(z_k)^2, \label{eq:usecur1} \\
        & \hspace{-14ex}\Big[L_k^{(t)} \tau_k^{(t)}\lambda + 1 -\tau_k^{(t)}\Big] \lambda \alpha \| \alpha e_k\|\|d_k+ \alpha e_k\| \label{eq:usecur2}\\
        &\overset{\eqref{eq:cur}}{\le} \lambda \alpha\Big[L_k^{(t)} \tau_k^{(t)} \lambda + 1\Big]  \frac{ \alpha}{\eta_k} \Big(1+ \frac{1}{\eta_k}\Big) \chi(z_k)^2 \overset{\eqref{eq:ckt}}{\le} \frac{\tau_k^{(t)} \lambda \alpha}{2} \frac{1-\sigma}{2}\chi(z_k)^2. \nonumber
    \end{align}
    Here, in the first inequality of \eqref{eq:usecur2}, we also applied $\alpha \le 1$ and $\tau_k^{(t)} \ge 0$.
    Thus, \eqref{eq:armijo} holds if $\alpha \le \min\{\bar \alpha_k^{(t)}, c_k^{(t)}\}$ which finishes the proof. 
    \end{proof}
Similar to \cite[Theorem 4.8]{ouyang2025trust}, we now list two standard assumptions on $f$ and $\vp$ and present the global convergence results. In contrast to \cite[Assumption B.2]{ouyang2025trust}, we do not need to assume boundedness of the matrices $\{B_k\}$. This relaxation is possible due to the adaptive control of the term $\|e_k\|$ by the gradient-related test \eqref{eq:cur}---as seen in \eqref{eq:usecur1} and  \eqref{eq:usecur2}.

\begin{assumption}\label{assum2-1} We consider the conditions: \vspace{0.5ex}
    \begin{enumerate}[label=\textup{\textrm{(A.\arabic*)}},topsep=0pt,itemsep=0ex,partopsep=0ex,leftmargin=8ex]
    \item \label{A1} The gradient $\nabla f$ is Lipschitz continuous on $\dom{\vp}$ with modulus $L$.
    \item \label{A2} The objective function $\psi$ is bounded from below (on $\dom{\partial\vp}$).
    \end{enumerate}
\end{assumption}

\begin{thm}\label{thm:global}
Let \ref{A1}--\ref{A2} be satisfied and let $\{z_k\}$ be generated by Algorithm~\ref{algo:main}. The algorithm then either terminates after finitely many steps or we have
$$\lim_{k\rightarrow\infty}\chi(z_k)=0 \quad \text{and} \quad \sum_{k=0}^\infty {\nu}_k \|x_{k+1}-x_k\|^2 < \infty.$$
\end{thm}
\begin{proof}
We first show that the linesearch procedure will always stop within finitely many inner iterations and the stepsizes $\alpha_k$, $k \in \N$, are bounded from below. Due to \ref{A1}, $\nu_k \le \nu$ and the non-increasing property of the sequence $\{\tau_k\}$, we have $L_k^{(t)} \le \max\{\bar L, L\} =: L_{\max}$ and  $\tau_k^{(t)} \ge  \min\{ \frac{2\gamma(1-\nu)}{L_{\max}^2 \lambda^2 +2}, \tau_{-1}\}=:\tau_{\min}$. Combining Lemma~\ref{lemma:welldef} and the backtracking mechanism,  it follows
\begin{align}
    \nonumber \alpha_k & \ge \rho \min\Big\{1, \frac{2(1-\gamma)(1-\nu)}{(1+2\tau_{-1})L_{\max}\lambda +\tau_{-1}}, \sqrt{\frac{1-\sigma}{2}}\eta_k, \frac{(1-\sigma)\eta_k\tau_{\min}}{4(L_{\max} \tau_{-1} \lambda +1)(1+1/\eta_k)}\Big\} \\
    & \geq C_1 \min\Big\{1,\eta_k,\frac{2\eta_k^2}{1+\eta_k} \Big\} \geq C_1 \min\{1,\eta_k,\eta_k^2\}\label{eq:alphafinal}
\end{align}
for some suitable $C_1 > 0$. By assumption \ref{A2}, the function values $\{H(\tau_k, z_k)\}$ are bounded from below. By the linesearch condition \eqref{eq:armijo}, the sequence $\{H(\tau_k, z_k)\}$ is non-decreasing and thus, it converges to some $\zeta \in \R $. We now consider the case where the algorithm does not stop within finitely many iterations. Combining \eqref{eq:armijo}, \eqref{eq:alphafinal}, $\alpha_k \le 1$ and the definition of $\eta_k$ in \eqref{eq:cur}, we then have 
    \begin{align}
        \nonumber\infty &> H(\tau_0,z_0) - \zeta  = \sum_{ k\ge 0}H(\tau_{k},z_k) - H(\tau_{k+1},z_{k+1}) \ge \sum_{ k\ge 0}H(\tau_{k},z_k) - H(\tau_{k},z_{k+1})\\
        \nonumber & \overset{\eqref{eq:armijo}}{\ge} \sum_{ k\ge 0} \frac{\sigma \lambda \tau_k \alpha_k}{2}\chi(z_k)^2 + \frac{\nu_k}{\lambda \alpha_k} \|x_{k+1}-x_k\|^2 \\
        \nonumber & \overset{\eqref{eq:alphafinal}}{\ge} \sum_{ k\ge 0} C_2 \min\{1, \eta_k ,\eta_k^2\}\chi(z_k)^2 + \frac{\nu_k}{\lambda} \|x_{k+1}-x_k\|^2 \\
        &  \overset{\eqref{eq:cur}}{\ge} \sum_{ k\ge 0} C_2 \min\{1,\eta^2, b_k\chi(z_k)^{q}, b_k^2\chi(z_k)^{2q}\}\chi(z_k)^2 + \frac{\nu_k}{\lambda} \|x_{k+1}-x_k\|^2. \label{eq:sum}
    \end{align}
where $C_2 := \frac{C_1\sigma \lambda \tau_{\min} }{2}$ and the second inequality is due to the monotonicity of $\{\tau_k\}$. This readily yields the desired results. 
\end{proof}

\subsection{Convergence Properties Under the \texorpdfstring{Kurdyka- {\L}ojasiewicz}{Kurdyka- Lojasiewicz} Inequality} 
\label{ssec:KL}
In this subsection, we investigate convergence properties of the iterates $\{z_k\}$ under the Kurdyka-{\L}ojasiewicz (KL) inequality. Our analysis builds on the highly successful KL framework studied in  \cite{absil2005convergence,AttBol09,AttBolSva13,BolSabTeb14} and is motivated by recent KL-based results for nonsmooth Newton methods \cite{SteThePat17,TheStePat18,ouyang2025trust} and the proximal gradient method \cite{jia2023convergence}. By $\mathfrak{S}_\eta$, we denote the class of continuous and concave desingularizing functions $\varrho : [0,\eta) \to \R_+$ such that 
\[ \varrho \in C^1((0,\eta)), \quad \varrho(0) = 0, \quad  \varrho^\prime(t) > 0 \quad \forall~t \in (0,\eta). \]
%
In addition, the set $\mathfrak L := \{ \varrho : \R_+ \to \R_+ : \exists~c > 0, \, \theta \in [0,1): \varrho(t) = c t^{1-\theta} \}$ denotes a subclass of {\L}ojasiewicz functions. 
Clearly, we have $\mathfrak L \subset \mathfrak S_\eta$ for all $\eta > 0$. Here, we only state the KL-property for functions of the form $\psi = f + \varphi$, where $f$ is continuously differentiable and $\vp: \Rn \to \Rex$ is convex, lsc, and proper.
\begin{defn} \label{def:kl} Let $\psi = f + \vp$ be a proper, lsc function as specified above. We say that $\psi$ has the Kurdyka-{\L}ojasiewicz (KL-)property at the point $\bar x \in \dom{\partial\vp}$ if there are $\eta \in (0,\infty]$, a neighborhood $U$ of $\bar x$, and a function $\varrho \in \mathfrak{S}_\eta$ such that for all $x \in U \cap \{x \in \Rn : 0 < \psi(x)-\psi(\bar x) < \eta\}$ the following KL-inequality holds:
\begin{equation} \label{eq:kl-ineq} \varrho^\prime(\psi(x) - \psi(\bar x)) \cdot \dist(0,\partial \psi(x)) \geq 1. \end{equation}
If the mapping $\varrho$ can be chosen from $\mathfrak L$ and satisfies $\varrho(t) = c t^{1-\theta}$ for some $c > 0$ and $\theta \in [0,1)$, then we say that $\psi$ has the {\L}ojasiewicz-property at $\bar x$ with exponent $\theta$. 
\end{defn}

The KL-inequality is satisfied for the rich class of subanalytic and semialgebraic functions \cite{lojasiewicz1963,lojasiewicz1993,kurdyka1998,BolDanLew06}. Thus, KL-based analysis techniques enjoy broad applicability in practice. We now formulate our main assumptions. 

\begin{assumption} \label{ass:kl} We consider the conditions: \vspace{0.5ex}
    \begin{enumerate}[label=\textup{\textrm{(B.\arabic*)}},topsep=0pt,itemsep=0ex,partopsep=0ex,leftmargin=8ex]
    \item \label{C1} Suppose there exists $\bar z \in \mathcal A := \{z: \liminf_{k\to\infty}\|z_k-z\| = 0\}$ such that $\psi$ has the KL-property at $\bar x = \prox{\bar z}$.
    \item \label{C3} There is an accumulation point $\bar z \in \mathcal A$ such that $\psi$ satisfies the {\L}ojasiewicz-property at $\bar x = \prox{\bar z}$ with exponent $\theta \in [0,1)$.
    \item \label{C2} We assume that $\nu_k$ is given by $ \nu_k = \min\{\nu ,a_k^2\|\prox{z_k+s_k(\alpha_k)}-x_k\|^{2p}\}$ where $a_k = (k\ln^{2}(k+1))^p$ and $p>0$.
    \end{enumerate}
\end{assumption}

Assumption~\ref{C2} specifies the growth behavior of the parameters $\{\nu_k\}$. This will play an important role in the convergence proof. More generally, the constants $a_k$ in \ref{C2} are only required to satisfy $\sum_{k=0}^{\infty} a_k^{-1/p} < \infty$, cf. \cite{ouyang2025trust}. Here, we work with a more specific form for $a_k$ to simplify the analysis. Next, we show that the KL-property can be transferred from $\psi$ to the merit function $z \mapsto H(\tau_0,z)$. Lemma~\ref{lemma:calc-kl-mer} is a generalization of \cite[Lemma 5.3]{ouyang2025trust}, where only the {\L}ojasiewicz case is considered.

\begin{lemma} \label{lemma:calc-kl-mer} 
    Let $\psi$ satisfy the KL-property at a stationary point $\bar x = \prox{\bar z} \in \crit(\psi)$. Then, there exist $\eta \in (0,\infty]$, a neighborhood $V$ of $\bar z$, and a function $\varrho_H \in \mathfrak S_\eta$ such that for all $z \in V \cap \{z \in \Rn: 0 < H(\tau_{0}, z)-H(\tau_{0}, \bar z) < \eta\}$, we have 
    \be \label{eq:kl-mer} \varrho_H^\prime(H(\tau_{0}, z) - H(\tau_{0}, \bar z)) \cdot \chi(z) \geq 1. \ee
    If $\psi$ satisfies the {\L}ojasiewicz-property at $\bar x$ with exponent $\theta$, then $\varrho_H$ can be selected from $\mathfrak L$ and we have $\varrho_H(t) = ct^{1-\theta^\prime}$ where $c > 0$ and $\theta^\prime = \max\{\theta,\frac12\}$.
\end{lemma}
\begin{proof} We only need to prove the first statement here.
Let $\eta > 0$, $\varrho \in \mathfrak S_\eta$, and $U$ be given as in Definition~\ref{def:kl}.
Our proof is based on the following two key steps:
    \begin{itemize}
        \item[\rmn{(i)}] There exists a strictly decreasing function $\zeta : (0,\eta) \to \R$ such that $\zeta(t) \ge \varrho^\prime(t)$ for all $t \in (0,\eta)$.
        \item[\rmn{(ii)}] Let us set $\varrho_H(t)=  \int_{0}^{t} \varrho_H^\prime(s)\,\mathrm{d}s$ where $({1}/{\varrho_H^{\prime}})^{-1}(t) = 2 \max\{(1/\zeta)^{-1}(t),\frac{\tau_0 \lambda}{2} t^2 \}$. Then, we have $\varrho_H \in \mathfrak S_\eta$ and the KL-type inequality \eqref{eq:kl-mer} is satisfied.
    \end{itemize}
     Statement (i) is clear as we can set $\zeta(t)=\varrho'(t) + \eta -t$, which meets the requirement since $\eta -t$ is strictly decreasing and non-negative on $(0,\eta)$.
    Next, let us turn to part (ii). There exists $\epsilon > 0$ such that $B_\epsilon(\bar x) \subset U$. We now choose $\delta > 0$ sufficiently small such that $H(\tau_0,\bar z) = \psi(\bar x) < H(\tau_0,z) < H(\tau_0,\bar z) + \eta$ and $\|\prox{z}-\bar x\| \leq \epsilon$ for all $z \in B_\delta(\bar z)$. Hence, for such $z$, we can infer $\prox{z} \in U$ and $\psi(\prox{z}) < \psi(\bar x) + \eta$. If $\psi(\prox{z}) > \psi(\bar x)$, then \eqref{eq:kl-ineq} is applicable and it follows
    $$ \zeta(\psi(\prox{z}) - \psi(\bar x)) \chi(z) \geq \varrho^\prime(\psi(\prox{z}) - \psi(\bar x)) \chi(z) \ge 1,$$
    %
    where we used $\Fnor{z} \in \partial\psi(\prox{z})$. Since $\zeta$ is strictly decreasing, the inverse function $s \mapsto (1/\zeta)^{-1}(s)$ exists (on $(0,\eta)$) and  we have
    \begin{align}
            \nonumber H(\tau_0, z) - H(\tau_0, \bar z)  &= \psi(\prox{z}) - \psi(\bar x) + \frac{\tau_0\lambda}{2} \chi^2(z) \\ \nonumber & \le (1/\zeta)^{-1} (\chi(z)) + \frac{\tau_0 \lambda}{2} \chi^2(z) \\
            & \le 2 \max\Big\{(1/\zeta)^{-1} (\chi(z)),\frac{\tau_0 \lambda}{2} \chi^2(z)\Big\} = ({1}/{\varrho_H^{\prime}})^{-1} (\chi(z)).
            \label{eq:kl-general}
    \end{align}
    Moreover, \eqref{eq:kl-general} naturally holds when $\psi(\prox{z}) \leq \psi(\bar x)$. As the maximum of two strictly increasing functions, the mapping $s \mapsto (1/\varrho_H^{\prime})^{-1}(s)$ is also strictly increasing. Thus, we can take its inverse in \eqref{eq:kl-general} and construct $\varrho_H$ as shown in part (ii). The KL-type inequality \eqref{eq:kl-mer} then follows readily from \eqref{eq:kl-general}.   Since the inverse of a strictly increasing function is also strictly increasing, $1/\varrho_H^{\prime}$ is strictly increasing and hence, $\varrho_H$ is a concave function. 
    It is then easy to see that $ \varrho_H \in \mathfrak S_\eta$.
\end{proof}

We now extend \cite[Theorem 5.5]{ouyang2025trust} to Algorithm~\ref{algo:main} and show convergence of the whole sequence $\{z_k\}$. In stark contrast to \cite{ouyang2025trust}, our proof does not require the boundedness of the iterates $\{z_k\}$, cf.\ \cite[Assumption C.2]{ouyang2025trust}. This is made possible by adapting the induction technique recently used in \cite{jia2023convergence}. 

\begin{thm} \label{thm:kl}
Let assumptions \ref{A1}--\ref{A2} and \ref{C2} hold and let the sequence $\{z_k\}$ be generated by Algorithm~\ref{algo:main}. Suppose $\bar z$ is an accumulation point of $\{z_k\}$ at which condition \ref{C1} is satisfied. Then, we have $\sum_{k=0}^\infty \|x_{k+1}-x_k\| < \infty$ and the whole sequence $\{z_k\}$ converges to $\bar z$ with $\bar x = \prox{\bar z} \in \crit(\psi)$.
\end{thm}

\begin{proof}
    It suffices to discuss the case when the algorithm does not stop in finitely many steps. Since $\{\tau_k\}$ is non-increasing and using $\tau_k \ge \tau_{\min}$, \eqref{eq:armijo}, one step of the AM-GM inequality, and \ref{C2}, we obtain
    \begin{align}
        \nonumber \frac{H(\tau_k, z_k)- H (\tau_{k+1}, z_{k+1})}{\chi(z_k) } &\ge \frac{H(\tau_k, z_k)- H (\tau_{k}, z_{k+1})}{\chi(z_k) }
        \ge \sqrt{2 \sigma \tau_{\min} \nu_k}  \|x_{k+1}- x_k\|\\
         & \ge C \min\{\sqrt{\nu}, a_k\|x_{k+1} -x_k\|^p\}  \|x_{k+1}- x_k\|, \label{eq:kl_descent}
    \end{align}
    where $C=\sqrt{2 \sigma \tau_{\min}}$. Moreover, since the algorithm does not stop after finitely many steps and we have $\chi(\bar z)=0$ by Theorem~\ref{thm:global}, we can infer $H(\tau_k, z_k)- H( \tau_k, \bar z)=H(\tau_k, z_k)- H(\bar \tau, \bar z) \to 0$ as $k \to \infty$ by \eqref{eq:sum}, where $\bar \tau = \lim_{k \to \infty} \tau_k$. Due to $\chi(\bar z)=0$, this also implies $H(\tau_0, z_k)- H(\tau_0, \bar z) \to 0$ as $k \to \infty$. Besides, since $\{H(\tau_k, z_k)\}$ is non-increasing and $\tau_0 \ge \tau_k$, we have $H(\tau_0, z_k) \ge H(\tau_k, z_k) >  H(\bar \tau, \bar z) = H(\tau_0, \bar z)$. (The condition $H(\tau_k, z_k) = H(\bar\tau,\bar z)$ would imply $\chi(z_k) = 0$ by \eqref{eq:armijo}). 

    Let $\eta > 0$, $V \subset \Rn$, and $\varrho_H \in \mathfrak S_\eta$ now be given as in Lemma~\ref{lemma:calc-kl-mer}. Then, there exists $ K_\eta \in \N$ such that for all $k \ge K_\eta$, we have $0 < H(\tau_0, z_k)-H(\tau_0, \bar z) \le \eta$. Let $\{z_{k_\ell}\}$ be a subsequence of $\{z_k\}$ converging to $\bar z$ and let $\{x_{k_\ell}\}$ be the corresponding subsequence of $\{x_k\}$ converging to $\bar x$, where $x_k=\prox{z_k}$ and $\bar x=\prox{\bar z}$. Setting $\varrho_{k} = \varrho_H( H(\tau_k, z_k)-H(\tau_k, \bar z))$, there is $k_{\ell_0} \ge \max\{K_\eta, 3\}$ such that $B_\delta (\bar z) \subset V$ where 
    \begin{align*}
    \delta &:= (1+\lambda L) \bar \gamma + \max_{k \ge k_{\ell_0}} \lambda \chi(z_k),\\
    \bar \gamma &:= \|x_{k_{\ell_0}}-\bar x\| + \frac{\varrho_{k_{\ell_0}}}{C \sqrt{\nu}} +  \Big(\frac{\varrho_{k_{\ell_0}}}{C(\ln (k_{\ell_0}-1))^p}\Big)^{1/(1+p)}.
    \end{align*}
    The existence of such $k_{\ell_0}$, $\delta$ and $\bar \gamma$ follow from $x_{k_\ell} \to \bar x$, $\varrho_{k_\ell}\to 0$, $\ln(k_\ell-1) \to \infty$ as $\ell \to \infty$ and $\chi(z_k) \to 0$. Next, we show by induction that the following three statements hold for all $k \ge k_{\ell_0}$:
    \begin{equation}\label{eq:state}
        x_k \in B_{\bar \gamma} (\bar x), \quad  z_k \in B_\delta(\bar z), \quad \|x_{k_{\ell_0}}-\bar x\| + {\sum}_{i=k_{\ell_0}}^{k} \|x_{i+1} -x_i\|\le \bar \gamma.
    \end{equation}
    When $k=k_{\ell_0}$, the first statement in \eqref{eq:state} holds by the definition of $\bar \gamma$. The second statement is due to 
    \begin{align} \nonumber \|z_k -\bar z\| & = \| x_k - \lambda \nabla f(x_k) + \lambda  \Fnor{z_k} -  \bar x + \lambda \nabla f(\bar x) \| \\ & \le (1+\lambda L) \|x_k -\bar x\| + \lambda \chi(z_k) \le \delta. \label{eq:zk-xk-lip}
    \end{align}
    To show the third statement, since $z_{k_{\ell_0}} \in B_\delta (\bar z) \subset V$ and $k_{\ell_0} \ge K_\eta$, \eqref{eq:kl-mer} is applicable for $k=k_{\ell_0}$. We then have
    \begin{align*}
        &C \min\{\sqrt{\nu}, a_{k_{\ell_0}}\|x_{k_{\ell_0}+1} -x_{k_{\ell_0}}\|^p\} \|x_{k_{\ell_0}+1} - x_{k_{\ell_0}}\| \\
        &\overset{\eqref{eq:kl_descent}}{\le} (H(\tau_{k_{\ell_0}}, z_{k_{\ell_0}})- H (\tau_{k_{\ell_0}+1}, z_{k_{\ell_0}+1}) )/\chi(z_{k_{\ell_0}})\\
        & \overset{\eqref{eq:kl-mer}}{\le} \varrho_H^\prime(H(\tau_0, z_{k_{\ell_0}})- H (\tau_0, \bar z))  (H(\tau_{k_{\ell_0}}, z_{k_{\ell_0}})- H (\tau_{k_{\ell_0}+1}, z_{k_{\ell_0}+1}))\\
        & \le \varrho_H^\prime(H(\tau_{k_{\ell_0}}, z_{k_{\ell_0}})- H (\tau_{k_{\ell_0}}, \bar z)) (H(\tau_{k_{\ell_0}}, z_{k_{\ell_0}})- H (\tau_{k_{\ell_0}+1}, z_{k_{\ell_0}+1}))\\ &\le \varrho_{k_{\ell_0}}- \varrho_{k_{\ell_0}+1} \le \varrho_{k_{\ell_0}}.
    \end{align*}
    Here, the third inequality uses the monotonicity of $\varrho_H^\prime$ and $\tau_0 \ge \tau_{k_{\ell_0}}$ and the last two inequalities use the concavity and nonnegativity of $\varrho_H$. Thus, it holds that
    \begin{align*}
    \varrho_{k_{\ell_0}} / C &\ge  \min\{\sqrt{\nu}, a_{k_{\ell_0}}\|x_{k_{\ell_0}+1} -x_{k_{\ell_0}}\|^p\} \|x_{k_{\ell_0}+1} -x_{k_{\ell_0}}\|\\
    &= \min\{\sqrt{\nu}\|x_{k_{\ell_0}+1} -x_{k_{\ell_0}}\|, a_{k_{\ell_0}}\|x_{k_{\ell_0}+1} -x_{k_{\ell_0}}\|^{1+p}\}.
    \end{align*}
    The third statement in \eqref{eq:state} is then satisfied since $1/a_{k_{\ell_0}} =1/(k_{\ell_0} \ln^2 (k_{\ell_0}+1))^p \le 1/(\ln(k_{\ell_0}-1))^p$.  Suppose the conditions in \eqref{eq:state} hold for all $ k_{\ell_0} \le i \le k$. By the triangle inequality, the definition of $\bar \gamma$ and $\delta$, and mimicking \eqref{eq:zk-xk-lip}, we have 
    \begin{align*}
    \|x_{k+1} -\bar x\| &\le \|x_{k_{\ell_0}} -\bar x\| + {\sum}_{i=k_{\ell_0}}^{k}\|x_{i+1}-x_i\| \le \bar \gamma,\\
    \|z_{k+1}-\bar z\| &\le (1+\lambda L) \|x_{k+1} -\bar x\| + \lambda \chi(z_{k+1}) \le \delta.
    \end{align*}
    This proves $x_{k+1} \in B_{\bar \gamma} (\bar x)$, $z_{k+1} \in B_\delta (\bar x)$ and hence, \eqref{eq:kl-mer} is applicable for all $k+1 \ge i \ge k_{\ell_0}$. To verify the third statement, we similarly note that
    \begin{equation}\label{eq:kl_sum}
        \begin{aligned}
            & {\sum}_{i=k_{\ell_0}}^{k+1} C \min\{\sqrt{\nu}, a_i\|x_{i+1} -x_i\|^p\} \|x_{i+1} -x_i\|\\
            &\le {\sum}_{i=k_{\ell_0}}^{k+1} (H(\tau_i, z_i)- H (\tau_{i+1}, z_{i+1}))/\chi(z_i)\\
            & \le {\sum}_{i=k_{\ell_0}}^{k+1} \varrho_H'(H(\tau_0, z_i)- H (\tau_0, \bar z)) (H(\tau_i, z_i)- H (\tau_{i+1}, z_{i+1})) \\
            & \le {\sum}_{i=k_{\ell_0}}^{k+1}  \varrho_H'(H(\tau_i, z_i)- H (\tau_i, \bar z)) (H(\tau_i, z_i)- H (\tau_{i+1}, z_{i+1})) \\
            &\le {\sum}_{i=k_{\ell_0}}^{k+1}  \varrho_{i}- \varrho_{i+1} = \varrho_{k_{\ell_0}} - \varrho_{k+2} \le \varrho_{k_{\ell_0}}.
        \end{aligned}
    \end{equation}
    Next, defining $\mathcal{I}_1 = \{k+1 \ge i \ge k_{\ell_0} \mid \nu_i=\nu\}$ and $\mathcal{I}_2 = \{k+1 \ge i \ge k_{\ell_0} \mid i \notin \mathcal{I}_1\}$ and using the reversed H\"older inequality, we obtain
\begin{align*}
       \frac{\varrho_{k_{\ell_0}} }{C}&\overset{\eqref{eq:kl_sum}}{\ge} {\sum}_{i=k_{\ell_0}}^{k+1} \min\{\sqrt{\nu}, a_i\|x_{i+1} -x_i\|^p\} \|x_{i+1} -x_i\| \\
        &= \sqrt{\nu}  {\sum}_{i \in \mathcal{I}_1}  \|x_{i+1} -x_i\| +   {\sum}_{i\in\mathcal{I}_2}  a_i\|x_{i+1} -x_i\|^{1+p} \\
        &\ge  \sqrt{\nu}  {\sum}_{i \in \mathcal{I}_1}  \|x_{i+1} -x_i\| +   \Big [{\sum}_{i\in\mathcal{I}_2}  a_i^{-1/p} \Big  ]^{-p} \Big [{\sum}_{i\in\mathcal{I}_2} \|x_{i+1} -x_i\| \Big 
 ]^{1+p} \\
        &=  \sqrt{\nu}  {\sum}_{i \in \mathcal{I}_1}  \|x_{i+1} -x_i\| + \Big   [{\sum}_{i\in\mathcal{I}_2}  \frac{1}{i \ln^2(i+1)} \Big  ]^{-p}
        \Big [{\sum}_{i\in\mathcal{I}_2} \|x_{i+1} -x_i\| \Big 
 ]^{1+p} \\
        &\ge  \sqrt{\nu}  {\sum}_{i \in \mathcal{I}_1}  \|x_{i+1} -x_i\| + \Big  [\ln(k_{\ell_0}-1)\Big  ]^{p}\Big  [{\sum}_{i\in\mathcal{I}_2} \|x_{i+1} -x_i\|\Big  ]^{1+p}.
    \end{align*} 
    Here, the second inequality is due to the reversed H\"older inequality and the last inequality uses the bound $\sum_{i \in \mathcal{I}_2} \frac{1}{i \ln^2(i+1)} \le \frac{1}{\ln(k_{\ell_0}-1)}$. This shows \eqref{eq:state} also holds for $k+1$ and completes the induction. Finally, since \eqref{eq:state} is satisfied for all $k \ge k_{\ell_0}$, taking $k \to \infty$, we can deduce that $\sum_{k=0}^{\infty} \|x_{k+1}-x_k\| < \infty$. Hence, $\{x_k\}$ is a Cauchy sequence and it follows $x_k \to \bar x$. Moreover, the convergence of $\{x_k\}$ also yields convergence of $\{z_k\}$ in the sense that
    $$z_k = x_k - \lambda \nabla f(x_k) + \lambda  \Fnor{z_k} \to  \bar x - \lambda  \nabla f(\bar x) =: \bar z, \quad k \to \infty,$$
    and thus, we have $\bar x = \prox{\bar z} \in \crit(\psi)$.    
    \end{proof}

    Similar to \cite[Theorem 5.5, part (ii)]{OuyMil21}, we now quantify the convergence rate of the sequences $\{z_k\}$ and $\chi(z_k)$ depending on the KL exponent $\theta$.
    
    \begin{thm} \label{thm:klspeed}
    Suppose \ref{A1}--\ref{A2} and \ref{C2} are satisfied and let $\{z_k\}$ be generated by Algorithm~\ref{algo:main}. Let $\bar z$ be an accumulation point of $\{z_k\}$ at which \ref{C3} holds with exponent $\theta \in [0,1)$. Then, $\{\chi(z_k)\}$ and $\{z_k\}$ converge with the following rates:
             \[ \chi(z_k) =  O(k^{-\frac{1}{\omega-1}}) \quad \text{and} \quad \|z_{k}-\bar z\| =  O(k^{-\frac{1-\theta^\prime}{(1+p)\theta^\prime(\omega-1)}}), \]
where $\theta^\prime = \max\{\theta,\frac12\}$ and $\omega = \frac{(1+2q)\theta^\prime}{1-\theta^\prime}$. Moreover, we have $\liminf_{k\to\infty}k^{1+r}\chi(z_k)=0$ for all $r \in (0,\frac{1}{\omega-1})$.
    
            %
        \end{thm}
        \begin{proof}
        By Lemma~\ref{lemma:calc-kl-mer}, the KL inequality \eqref{eq:kl-mer} holds at $\bar z$ with exponent $\theta^\prime = \max\{\theta,\frac12\}$. Hence, all steps in the proof of Theorem~\ref{thm:kl} are applicable (ensuring $z_k \to \bar z$) and there is $K_\eta\in \N$ such that for all $k \ge K_\eta$, we have
            \begin{equation}\label{eq:klg1}
                \begin{aligned}
                    \varrho_{k}& = \varrho_H(H(\tau_k, z_k) - H(\tau_k, \bar z))  = c \Big[ \frac{c(1-\theta^\prime)}{\varrho_H^\prime(H(\tau_k, z_k) - H(\tau_k, \bar z))} \Big]^{\frac{1-\theta^\prime}{\theta^\prime}} \\
                   &\le c \Big[ \frac{c(1-\theta^\prime)}{\varrho^\prime_H(H(\tau_0, z_k) - H(\tau_0, \bar z))} \Big]^{\frac{1-\theta^\prime}{\theta^\prime}} \overset{\eqref{eq:kl-mer}}{\leq} c(c(1-\theta^\prime))^{\frac{1-\theta^\prime}{\theta^\prime}} \chi(z_k)^{\frac{1-\theta^\prime}{\theta^\prime}}.
               \end{aligned}
            \end{equation}
            Here, we used $\varrho_H(t) = c t^{1-\theta^\prime}$ and $\varrho_H(t) = c [{c(1-\theta^\prime)}/{\varrho^\prime_H(t)}]^{(1-\theta^\prime)/\theta^\prime}$ in the second equation, and the monotonicity of $\varrho^\prime_H$ and $\tau_0 \ge \tau_k$ in the first inequality. Moreover, recalling \eqref{eq:sum} in Theorem~\ref{thm:global}, it follows
            \begin{equation}\label{eq:alpha_bound}
                \frac{\sigma \lambda \tau_k \alpha_k}{2} \ge C_2 \min\{1,\eta^2, b_k\chi(z_k)^{q}, b_k^2\chi(z_k)^{2q}\}
            \end{equation}
            for some constant $C_2 >0$. We further introduce
            \begin{equation*}
            \begin{aligned}
            \mathcal{I}_1 & = \{ i \ge k \mid \min\{1,\eta^2\} < \min\{ b_i\chi(z_i)^{q}, b_i^2\chi(z_i)^{2q}\} \}\\
            \mathcal{I}_2 & = \{ i \ge k \mid b_i\chi(z_i)^{q} < \min\{1,\eta^2,   b_i^2\chi(z_i)^{2q}\} \}, \quad \mathcal{I}_3 = \{ i \ge k \mid i \notin \mathcal I_1 \cup \mathcal I_2 \}\\
            \Gamma_k & = {\sum}_{i=k}^{\infty} \chi(z_i), \quad \Gamma_{k, \mathcal{I}_j} = {\sum}_{i \in \mathcal{I}_j} \chi(z_i), \ j=1,2,3.
            \end{aligned}
            \end{equation*}
            It then follows
            \begin{equation*}
                \begin{aligned}
                    \varrho_k/C_2 & \ge {\sum}_{i=k}^{\infty} \frac{H(\tau_i, z_i)- H (\tau_{i+1}, z_{i+1})}{C_2\chi(z_i)} \overset{\eqref{eq:armijo}}{\ge} {\sum}_{i=k}^{\infty} \frac{\sigma \lambda \tau_i \alpha_i}{2C_2} \chi(z_i)\\
                    & \overset{\eqref{eq:alpha_bound}}{\ge} \min\{1,\eta^2\} {\sum}_{i \in \mathcal{I}_1} \chi(z_i) + {\sum}_{i \in \mathcal{I}_2} b_i \chi(z_i)^{1+q} + {\sum}_{i \in \mathcal{I}_3} b_i^2 \chi(z_i)^{1+2q} \\
                    & \ge \min\{1,\eta^2\} \Gamma_{k, \mathcal{I}_1} +  \Big({\sum}_{i \in \mathcal{I}_2} b_i^{-1/q}\Big)^{-q} \Gamma_{k, \mathcal{I}_2}^{1+q} + \Big({\sum}_{i \in \mathcal{I}_3} b_i^{-2/(2q)}\Big)^{-2q}\Gamma_{k, \mathcal{I}_3}^{1+2q}\\
                    & \ge \min\{1,\eta^2\} \Gamma_{k, \mathcal{I}_1} +  \ln^q(k-1) \Gamma_{k, \mathcal{I}_2}^{1+q} + \ln^{2q}(k-1)\Gamma_{k, \mathcal{I}_3}^{1+2q}.
                \end{aligned}
            \end{equation*}
            Here, the first inequality essentially follows as in \eqref{eq:kl_sum}, the fourth inequality uses the reversed H\"older inequality and the last inequality applies the bound ${\sum}_{i \in \mathcal{I}_\ell} b_i^{-1/q} \le 1/\ln(k-1)$, $\ell = 2,3$. Moreover, since $\varrho_k \to 0$ as $k \to \infty$, there exist some $k' \ge \max\{K_\eta, 3\}$ and some constant $D^\prime>0$ such that for all $k \ge k'$, it holds that
                \begin{align} \label{eq:klg2}
                   \Gamma_k & = \Gamma_{k, \mathcal{I}_1} + \Gamma_{k, \mathcal{I}_2} + \Gamma_{k, \mathcal{I}_3}\\ \nonumber
                   & \le \frac{\varrho_k}{C_2\min\{1,\eta^2\}} + \Big(\frac{\varrho_k}{C_2 \ln^q(k-1)}\Big)^{\frac{1}{1+q}} +  \Big(\frac{\varrho_k}{C_2 \ln^{2q}(k-1)}\Big)^{\frac{1}{1+2q}} \le \Big(\frac{\varrho_k}{D^\prime}\Big)^{\frac{1}{1+2q}}.
                \end{align}
            Setting $\omega = \frac{(1+2q)\theta^\prime}{1-\theta^\prime}$, $C_\theta = [(D^\prime/c)^{\theta^\prime/(1-\theta^\prime)}/(c(1-\theta^\prime))]^{-\frac{1}{\omega}}$ and combining \eqref{eq:klg1}, \eqref{eq:klg2}, this yields
            $$ \Gamma_k^{-\omega}(\Gamma_k - \Gamma_{k+1}) \ge C_\theta^{-\omega}. $$
            This recursion allows us to establish the stated rate for $\{\chi(z_k)\}$.
            The rate for $\{z_k\}$ can then be obtained by following the last part of the proof of Theorem~\ref{thm:kl}. We omit detailed computations here and refer to\cite[Proof of Theorem 2]{AttBol09} and \cite[Theorem 5.5, part (ii)]{OuyMil21}. (\cite{OuyMil21} handles the special case $q=0$.)
        \end{proof}

\begin{remark}\label{remark:liminf}
The rates in Theorem~\ref{thm:klspeed} are generally slower compared to the known KL-based rates for nonsmooth Newton methods \cite{SteThePat17,TheStePat18,OuyMil21,liu2024inexact}. This is mainly due to the adaptive nature of our globalization strategy which introduces a dependence on the globalization parameters $p, q > 0$. In contrast to \cite[Theorem 5.5, part (ii)]{OuyMil21}, we do not require a priori boundedness of the Hessian approximations $\{B_k\}$ in Theorem~\ref{thm:klspeed}.
\end{remark}

\subsection{Local Superlinear Convergence}
\label{ssec:local}
We start this subsection by first listing several local assumptions.
\begin{assumption} \label{assum3-1} Let $\{z_k\}$ be generated by Algorithm \ref{algo:main} and suppose that $\bar z \in \mathcal A$ is an accumulation point of the sequence $\{z_k\}$. We then consider: \vspace{.5ex}
    \begin{enumerate}[label=\textup{\textrm{(C.\arabic*)}},topsep=0pt,itemsep=0ex,partopsep=0ex,leftmargin=8ex]
    \item \label{D2} The function $f$ is twice continuously differentiable on $\dom{\vp}$.
    \item \label{D3} The mapping $\proxs$ is semismooth at $\bar z$.
    \item \label{D4} There exist $\kappa_M > 0$ and $K \in \N$ such that for all $k \geq K$ the matrix $D_k M_{k}$ is positive semidefinite and $M_k$ is invertible with $\|M_k^{-1}\| \leq \kappa_M$. 
    \item \label{D5} The matrices $\{B_k\}$ satisfy the following Dennis-Mor\'{e}-type condition
    \[  \lim_{k \to \infty} \frac{\|[B_k - \nabla^2 f(x_k)](x_k - \bar x)\|}{\|z_k - \bar z\|} = 0 \quad \text{where} \quad \bar x = \prox{\bar z}. \]
    \item \label{D6}The error threshold $\epsilon_k$ used in the CG method satisfies $\epsilon_k \leq \mathcal E(\Fnor{z_k})$ where $\mathcal E : \Rn \to \R_+$ is continuous with $\mathcal E(0) = 0 $ and $\mathcal E(h) = o(\|h\|)$ as $h \to 0$.
    \item \label{D7} There is $\kappa_B$ such that $\|B_k\| \le \kappa_B$ for all $k \in \N$.
    \end{enumerate}
\end{assumption}

We proceed with several remarks. Assumptions \ref{D2} and \ref{D3} are standard requirements that ensure semismoothness of the normal map. Assumption \ref{D4}, which is our main curvature and boundedness condition, is partially related to the CD-regularity (see, e.g., \cite{qi1997semismooth}) of the normal map $\oFnor$ at $\bar z$. In particular, suppose that $\oFnor$ is CD-regular at $\bar z$. By \cite{clarke1990optimization}, it holds that $\partial \Fnor{\bar z}h = \mathcal M^\lambda(\bar z)h$ for all $h \in \Rn$ and hence, every matrix $M \in \mathcal M^\lambda(\bar z)$ must be invertible. Since the set-valued mapping $\mathcal M^\lambda : \Rn \rightrightarrows \R^{n \times n}$ is upper semicontinuous and locally bounded (this follows from the continuity of $\nabla^2 f$, $\proxs$ and from the properties of $\partial\proxs$, cf.\ \cite[Proposition 2.6.2]{clarke1990optimization}), we can then infer $\|M^{-1}\| \leq \bar\kappa$ for some $\bar\kappa > 0$ and all $M \in \mathcal M^\lambda(z)$ and $z$ in a neighborhood of $\bar z$, cf.\ \cite[Proposition 3.1]{QiSun93}. Thus, in this case and if $B_k = \nabla^2 f(x_k)$, $k \in \N$, the nonsingularity condition in \ref{D4} is satisfied. We further note that positive semidefiniteness of the matrices $D_kM_k$ can be ensured if $B_k$ is positive semidefinite. \ref{D4} can also be linked to second-order optimality conditions if $B_k = \nabla^2 f(x_k)$, see  \cite[Section 7]{OuyMil21} and \cite{ouyang2024soc}. A similar version of the Dennis-Moré condition in \ref{D5} was recently used in \cite{mannelhybrid}. Both \ref{D5} and \ref{D7} are naturally satisfied when the full Hessian $B_k = \nabla^2 f(x_k)$ is used and $z_k \to \bar z$. \ref{D6} implies that the linear systems \eqref{eq3-2} are solved with increasing accuracy as $k$ increases and the tolerance parameter $\epsilon_k$ is connected to $\Fnor{z_k}$.

We now first apply \ref{D2}--\ref{D4} to provide a bound for $\|z_k-\bar z\|$ and the merit function $H$ in terms of the criticality measure $\chi(z_k)$.

\begin{lemma}\label{lemma:new}
Let \ref{A1}, \ref{D2}--\ref{D4} hold and let $\bar z \in \mathcal A$ be given. Then there exists $\bar\delta > 0$ such that for all $z_k \in B_{\bar \delta}(\bar z)$, we have
\begin{equation} \label{eq:kl-from-d4} \|z_k - \bar z\| \leq 2\kappa_M \cdot \chi(z_k) \quad \text{and} \quad H(\tau_0,z_k) - H(\tau_0,\bar z) \leq C_{\varrho}\cdot\chi(z_k)^2, \end{equation}
where $C_{\varrho} := 2\kappa_M + 2L\kappa_M^2 + \frac{\tau_0\lambda}{2}$.
\end{lemma}
\begin{proof} As mentioned in Section~\ref{ssec:pre} and combining assumptions \ref{D2} and \ref{D3}, we can infer that $\Fnors$ is semismooth at $\bar z$ with respect to the set-valued mapping $\mathcal M^\lambda$, cf.\ \cite[Theorem 7.5.17]{facchinei2007finite}. Hence, there exists $\bar\delta > 0$ such that
\[ {\sup}_{M \in \mathcal M^\lambda(z)}~\|\Fnor{z}-M(z-\bar z)\| \leq \frac12\kappa_M^{-1} \|z - \bar z\| \]
for all $z \in B_{\bar\delta}(\bar z)$, where $\kappa_M$ is taken from \ref{D4}. If $z_k \in B_{\bar\delta}(\bar z)$, this yields
\begin{align*} \|z_k - \bar z\| = \|M_k^{-1}[M_k(z_k-\bar z)]\| & \leq \kappa_M[\|\Fnor{z_k}-M_k(z_k-\bar z)\| + \|\Fnor{z_k}\|] \\ & \leq \frac12 \|z_k - \bar z\| + \kappa_M \|\Fnor{z_k}\| 
\end{align*}
and we can conclude $\|z_k-\bar z\| \leq 2\kappa_M\chi(z_k)$. Recalling $x_k=\prox{z_k}$ and $\bar x=\prox{\bar z}$, we further have
\begin{align*}
	H(\tau_0, z_k) - H(\tau_0, \bar z)
    &=\psi(\prox{z_k})-\psi(\prox{\bar z}) + \frac{\tau_0 \lambda}{2}\chi(z_k)^2\\
	& \hspace{-8ex} \le \iprod{\nabla f(x_k)}{x_k-\bar x} + \frac{L}{2}\|x_k -\bar x\|^2 + \varphi(x_k)- \varphi(\bar x) + \frac{\tau_0 \lambda}{2}\chi(z_k)^2\\
	& \hspace{-8ex} \le \iprod{\nabla f(x_k)}{x_k-\bar x} + \frac{L}{2}\|x_k -\bar x\|^2 + \frac{1}{\lambda}\iprod{z_k-x_k}{x_k -\bar x} + \frac{\tau_0 \lambda}{2}\chi(z_k)^2\\
	& \hspace{-8ex} = \iprod{\Fnor{z_k}}{x_k -\bar x} + \frac{L}{2}\|x_k -\bar x\|^2 +  \frac{\tau_0 \lambda}{2}\chi(z_k)^2\\
	& \hspace{-8ex} \le \chi(z_k)\|x_k-\bar x\| + \frac{L}{2}\|x_k -\bar x\|^2 +  \frac{\tau_0 \lambda}{2}\chi(z_k)^2 \\ & \hspace{-8ex} \leq \Big[2\kappa_M + 2L\kappa_M^2 + \frac{\tau_0\lambda}{2}\Big] \chi(z_k)^2.
\end{align*}
Here, we use the Lipschitz continuity of $\nabla f$, i.e., condition \ref{A1}, in the first inequality, $(z_k-x_k)/\lambda \in \partial \varphi(x_k)$ in the second inequality and $\|x_k-\bar x\| \leq \|z_k-\bar z\| \le 2\kappa_M \chi(z_k)$ in the last inequality. This finishes the proof.
\end{proof}

Lemma~\ref{lemma:new} implies that an iterative variant of the {\L}ojasiewicz-type property \eqref{eq:kl-mer} with $\varrho_H(t) = 2\sqrt{C_{\varrho}} t^{1/2}$ and exponent $\theta=\frac12$ is satisfied---provided that $z_k$ is sufficiently close to $\bar z$. Moreover, following the previous proofs, we can conclude that the results of Theorems~\ref{thm:kl} and \ref{thm:klspeed} still hold in this situation without explicitly requiring the KL conditions \ref{C1} or \ref{C3}. In particular, under assumptions \ref{A1}--\ref{A2}, \ref{C2}, and \ref{D2}--\ref{D4}, we have $z_k \to \bar z$ and $\liminf_{k\to\infty} k^{1+r}\chi(z_k) = 0$ for all $r \in (0,\frac{1}{2q})$. Next, invoking \cite[Lemma 6.5]{ouyang2025trust}, we establish descent properties of the merit function $H(\tau,z)$ along a sequence of directions $\{d_k\}$ that converges superlinearly with respect to $\{z_k\}$.  

\begin{lemma}\label{lemma:des_local}
    Let \ref{A1}--\ref{A2}, \ref{C2}, and \ref{D2}--\ref{D7} be satisfied and let $\{\bar d_k\}$ be a superlinearly convergent sequence in the sense 
    \begin{equation*} 
    \|z_k + \bar d_k - \bar z\| = o(\|z_k - \bar z\|) \quad k \to \infty. \end{equation*}
    Then, there exists $\bar \sigma > 0$ such that for every $\eta \in (0,1)$ there exists a constant $K_\eta \geq K$ where for all $k \geq K_\eta$, we have
    %
    \be \label{eq:p-des} H(\tau_k,z_k+\bar d_k)-H(\tau_k, z_k) \leq - \frac{\bar \sigma}{2\lambda} \|\prox{z_k+\bar d_k} - x_k\|^2 -\frac{\eta \lambda \tau_k}{2}\chi(z_k)^2. \ee
\end{lemma}
\begin{proof}
As mentioned and thanks to Lemma~\ref{lemma:new}, the results of Theorem~\ref{thm:kl} are still valid under \ref{A1}--\ref{A2}, \ref{C2}, and \ref{D2}--\ref{D4} and we have $z_k \to \bar z$. As a consequence, \cite[Lemma 6.5]{ouyang2025trust} is applicable, which establishes \eqref{eq:p-des} for fixed $\tau_k = \tau$. However, by the proof of \cite[Lemma 6.5]{ouyang2025trust}, we can obtain the following immediate extension: there are $\bar \sigma >0$ and $K_{\bar \sigma} \in \N$ such that 
      \begin{align} 
         H(\tau_k, z_k + \bar d_k)  & \leq H(\tau_k,z_k)- \frac{\bar\sigma}{2\lambda} \|\hat x_{k+1}-x_k\|^2 -\frac{\tau_k \lambda}{2}\chi(z_k)^2 + o( \|z_k - \bar z\|^2),  \label{eq:barsigma} 
      \end{align}
      for all $k \ge K_{\bar \sigma}$, where $\hat x_{k+1} = \prox{x_k + \bar d_k}$. Recalling $\tau_{\min}$ from the proof of Theorem~\ref{thm:global}, there is $\N \ni K_\eta \ge K_{\bar \sigma}$ such that we have $o( \|z_k - \bar z\|^2) \le c \|z_k - \bar z\|^2$ for all $k \ge K_\eta$ with $c=(1-\eta)\tau_{\min} \lambda/(8\kappa_M^2)$. Hence, combining \eqref{eq:barsigma}, \eqref{eq:kl-from-d4} (from Lemma~\ref{lemma:new}) and $\tau_k \ge \tau_{\min}$, \eqref{eq:p-des} is true for all $k \ge K_\eta$.
\end{proof}

We now consider the case $\texttt{flag} = \texttt{SO}$ and define $\bar s_k := s_k(1) = \lambda (d_k +e_k)$. The next lemma establishes a bound of $\|\bar s_k\|$ in terms of $\chi(z_k)$. 

\begin{lemma}
    \label{lemma:bound} Let \ref{D4} and \ref{D7} be satisfied. Setting $\kappa_s :=\max\{\lambda, \kappa_M (2+ \lambda\kappa_B)\}$, it holds that
    \[ \|\bar s_k\| \leq \kappa_s \chi(z_k) \quad \forall~k \geq K. \]
\end{lemma}

\begin{proof} The proof is similar to the derivation of \cite[Lemma 6.6]{ouyang2025trust} and is included for completeness. If $\epsilon_k > \chi(z_k)$, it follows $\|D_k\Fnor{z_k}\| \leq \|D_k\| \chi(z_k) < \epsilon_k$ (due to $\|D_k\| \leq 1$). Thus, in this case, the CG method terminates in the first step with $e_k = q_k = q_0 = 0$ and $\|\bar s_k\| = \lambda\|d_k\| = \lambda \chi(z_k)$. On the other hand, if $\epsilon_k \leq \chi(z_k)$, then by Lemma~\ref{lemma3-8} (ii), we obtain $\|D_k(M_kq_k+\Fnor{z_k})\| \leq \epsilon_k$. Invoking Lemma~\ref{lemma3-8} (iii) and $\bar s_k = \lambda (d_k + e_k) = q_k - \lambda(M_kq_k + \Fnor{z_k})$, this yields $\|M_k\bar s_k + \Fnor{z_k}\| \leq \|I-\lambda B_k\|\epsilon_k \leq (1+\lambda\kappa_B)\chi(z_k)$ and $\|\bar s_k\| \leq \kappa_M \|M_k\bar s_k\| \leq \kappa_M(2+\lambda\kappa_B)\chi(z_k)$. This finishes the proof.
\end{proof}
We now present a linesearch-based counterpart to \cite[Theorem 6.2]{ouyang2025trust} and prove the fast local convergence of Algorithm~\ref{algo:main}.

\begin{thm} \label{thm:main-local-conv}
    Suppose that the conditions \ref{A1}--\ref{A2}, \ref{C2}, and \ref{D2}--\ref{D7} are satisfied and let $\{z_k\}$ be generated by Algorithm~\ref{algo:main}. Furthermore, let us assume that the method does not terminate after finitely many steps. 
    Then, we have:
    \begin{itemize}
    \item[\rmn{(i)}] There exists $\bar k \in \N$ such that for all $k \geq \bar k$, the algorithm always uses the second-order direction $\bar s_k$ with step size $\alpha_k=1$. Moreover, the sequence $\{z_k\}$ converges q-superlinearly to the limit point $\bar z$. 
    \item[\rmn{(ii)}] In addition, if $\proxs$ is $\beta$-order semismooth at $\bar z$ for $\beta \in (0,1]$, the function $\mathcal E$ in \ref{D6} satisfies $\mathcal E(h) = O(\|h\|^{1+\beta})$ as $h \to 0$, and if we choose $B_k = \nabla^2 f(x_k)$ and $\nabla^2 f$ is Lipschitz continuous near $\bar x$, then the rate of convergence is of order $1+\beta$. 
    \end{itemize}
\end{thm}

\begin{proof}   
    By Theorem~\ref{thm:global} and invoking Theorems~\ref{thm:kl} and \ref{thm:klspeed} under \ref{D2}--\ref{D4}, it follows $z_k \to \bar z$, $\chi(z_k) \to \chi(\bar z)=0$, and $\liminf_{k\to\infty} k^{1+r}\chi(z_k) = 0$, $r \in (0,\frac{1}{2q})$.
    We first mimic the discussion in the proof of \cite[Theorem 6.2]{ouyang2025trust}. By \ref{D4} and Lemma~\ref{lemma3-8} (ii), the CG method will return an $\epsilon_k$-accurate solution of \eqref{eq3-1} if $k \geq K$. Similar to the proof of Lemma~\ref{lemma:bound} and using Lemma~\ref{lemma3-8} (iii), \ref{D4}, and \ref{D6}, this implies
\begin{align*}
    \|z_k + \bar s_k - \bar z\| & \leq \kappa_M[\|\Fnor{z_k}-\Fnor{\bar z}- M_k(z_k-\bar z)\| + \|\Fnor{z_k}+M_k\bar s_k\|] \\ & \leq \kappa_M[\|\Fnor{z_k}-\Fnor{\bar z}- H_k(z_k-\bar z)\| + (1+\lambda \kappa_B) \mathcal E(\Fnor{z_k}) \\ & \hspace{-4ex} + \|\nabla^2 f(x_k)-B_k\|\|x_k-\bar x - D_k(z_k-\bar z)\| + \|[\nabla^2f(x_k)-B_k](x_k-\bar x)\|]
\end{align*}
for all $k \geq K$, where $H_k := \nabla^2 f(x_k)D_k + {\lambda^{-1}}(I-D_k)$. As argued in the proof of Lemma~\ref{lemma:new}, $\Fnors$ is semismooth at $\bar z$. Moreover, the Lipschitz continuity of $\nabla f$ and $\oprox$ implies that $\oFnor$ is Lipschitz continuous (with some constant $L_F$). Hence, by \ref{D6}, we have $\mathcal E(\Fnor{z_k}) = o(\|z_k-\bar z\|)$ and using \ref{D2},\ref{D3},\ref{D5}, and \ref{D7}, we may infer
    \begin{equation}\label{eq:superlinear}
        \|z_k+ \bar s_k - \bar z \| = o(\|z_k - \bar z\|),\quad k \to \infty.
    \end{equation}
    To prove part (i), it then suffices to verify that the gradient-related test \eqref{eq:cur} and linesearch condition \eqref{eq:armijo} hold for $\bar s_k$ when $k$ is sufficiently large. If $b_k\chi(z_k)^q \le  \frac{1}{\kappa_s +\lambda}:=\epsilon_q$, where $\kappa_s$ is defined in Lemma~\ref{lemma:bound}, we have
    $$ \|e_k\| \le \|\bar s_k\| + \|\lambda d_k\| \le (\kappa_s +\lambda) \chi(z_k) \le \chi(z_k)/(b_k \chi(z_k)^{q})  \le  \chi(z_k)/\eta_k,$$
    i.e.,  \eqref{eq:cur} holds.
    Moreover, if we have $k \ge K_\eta$ and $a_k\chi(z_k)^p \le \frac{ \sqrt{\bar \sigma/2}}{\kappa_s^p}:=\epsilon_p$, then \eqref{eq:armijo} is satisfied by using \eqref{eq:p-des} with $\eta=\sigma$ and 
    \begin{align*}
     \nu_k &\le a_k^2\|\prox{z_k + \bar s_k}-\prox{z_k}\|^{2p} \le a_k^2 \|\bar s_k\|^{2p} \le \kappa_s^{2p} a_k^2 \chi(z_k)^{2p} \le \bar \sigma/2.
    \end{align*}
    Here, $K_\eta$ and $\bar \sigma$ are defined in Lemma~\ref{lemma:des_local}. 
  By the choice of $a_k$, we also have
    \begin{equation}\label{eq:limsup}
        \limsup_{k \to \infty} \frac{a_k}{k^{p(1+r)}}  = \limsup_{k \to \infty} \frac{ (k \ln^2 (k+1))^{p}}{k^{p(1+r)}} = \limsup_{k \to \infty} \frac{\ln^{2p} (k+1)}{k^{p r}} =0.
    \end{equation}
    Combining $\liminf_{k\to\infty} k^{1+r}\chi(z_k) = 0$, $r \in (0,\frac{1}{2q})$, and \eqref{eq:limsup}, this yields
    \begin{equation}\label{eq:liminf2} 
        \liminf_{k\to \infty} a_k\chi(z_k)^p  \le  \limsup_{k\to \infty} \frac{a_k}{k^{p(1+r)}} \liminf_{k\to \infty}~(k^{1+r}\chi(z_k))^p =0.
    \end{equation} 
    Without loss of generality, we now assume $p \le q$ and denote the set $S:=\{k \ge K_\eta : a_k\chi(z_k)^p \le \min\{\epsilon_p,\epsilon_q, 1\}\}$. Thanks to \eqref{eq:liminf2}, $S$ is nonempty and has infinitely many elements. Moreover, for all $k \in S$, due to $a_k\chi(z_k)^p \le 1$ and $p \le q$, it follows $b_k\chi(z_k)^q \le a_k\chi(z_k)^p \le \epsilon_q$. Thus, we can infer $s_k(\alpha_k)=\bar s_k$ for all $k \in S$. 
    %
    %
    Using the Lipschitz continuity of $\Fnors$, \eqref{eq:kl-from-d4}, and \eqref{eq:superlinear}, we have
    \begin{equation}\label{eq:gksuperlinear}
        \chi(z_k+\bar s_k) \leq L_F \|z_k + \bar s_k - \bar z\|  = o(\|z_k-\bar z\|) =  o(\chi(z_k))
    \end{equation}
    for sufficiently large $k\in S$. Since $\lim_{k \to \infty} a_k/a_{k+1} =1$, \eqref{eq:gksuperlinear} implies that there is $\N \ni K^\prime \ge K_\eta$ such that $\chi(z_{k+1})^p \le (a_k/a_{k+1}) \chi(z_k)^p$ for all $k \ge K^\prime$ and $k \in S$. Let $\bar k \in \N$ be the first iteration satisfying $\bar k \ge K^\prime$ and $\bar k \in S$; we obtain:
        $$a_{\bar{k}+1}\chi(z_{\bar k+1})^{p} \le a_{\bar{k}+1} (a_{\bar{k}}/a_{\bar{k}+1})\chi(z_{\bar k})^{p} = a_{{\bar{k}}} \chi(z_{\bar k})^{p} \le \min\{\epsilon_p,\epsilon_q, 1\}.$$
     This shows $\bar{k}+1 \in S$ and by induction, it holds that $k \in S$ and $s_k(\alpha_k) =\bar s_k$ for all $ k \ge \bar k$. Thus, the stepsize $\alpha_k=1$ is accepted for all sufficiently large $k$ and the proof of part (i) is finished. The additional conditions in part (ii) of Theorem~\ref{thm:main-local-conv} imply that $\oFnor$ is $\beta$-order semismooth at $\bar z$. The estimate in \eqref{eq:superlinear} can then be improved to $\|z_k+\bar s_k - \bar z\| \leq O(\|z_k-\bar z\|^{1+\beta})$ which proves convergence of order $1+\beta$. 
\end{proof}


\section{Numerical Experiment}\label{sec:numerical}
In this section, we first list several implementational details. We demonstrate the efficiency of the proposed algorithm on a sparse logistic regression, a nonconvex image compression, and a nonlinear least square problem with group sparsity. All tests are performed on Matlab and a MacBook Pro with 2 GHz Quad-Core Intel Core i5 and 16 GB memory.

\subsection{Implementational Details}
\label{ssec:parameter}
We reuse some of the general settings from the original trust region-based normal map semismooth Newton method, cf. \cite[Section 7.1]{ouyang2025trust}. Specifically, we generate a comparable pair of initial points $x_0$ and $z_0$ to compare our normal map-based algorithm with other approaches. For given $x_0 \in \Rn$, we determine a corresponding initial point via $z_0 = \argmin_{\prox{z}=x_0}\|\Fnor{z}\|$. In our tests, such $z_0$ can be easily computed. Moreover, since we mainly consider large-scale problems, we use L-BFGS updates to approximate $\nabla^2 f$. Based on \cite{byrd1994representations}, we implement the following compact form of the (L-)BFGS scheme
\begin{equation*}
\begin{aligned}
S_k&=[\hat s_{k-m},\dots,\hat s_{k-1}], \quad Y_k=[y_{k-m},\dots,y_{k-1}],\\
B_{k}& =\gamma_k I-\begin{bmatrix} S_k & Y_k  \end{bmatrix} {\begin{bmatrix} \frac{1}{\gamma_k}S_k^\top S_k  & \frac{1}{\gamma_k}\mathcal{L}_k \\  \frac{1}{\gamma_k}\mathcal{L}_k^\top  & -\mathcal D_k  \end{bmatrix}}^{-1} \begin{bmatrix} S_k^\top  \\ Y_k^\top   \end{bmatrix},
\end{aligned}
\end{equation*}
where $\hat s_k =x_{k+1}-x_k $, $y_k=\nabla f(x_{k+1}) -\nabla f(x_k)$, $\mathcal{L}_k$ is the strictly lower part of $S_k^\top Y_k$, $\mathcal D_k$ is the diagonal part of $S_k^\top Y_k$, and $m \in \N$ is a memory parameter. We choose  $\gamma_k= \iprod{y_k}{y_k}/\iprod{\hat s_k}{y_k}$. In our tests, if an algorithm utilizes a quasi-Newton technique, then we apply L-BFGS approximations with memory $m=10$.  In the following, we will refer to Algorithm~\ref{algo:main} with L-BFGS approximations and full Hessians as LSSSN and LSSSN-H, respectively. Similarly, its trust region counterparts from \cite{ouyang2025trust} (with L-BFGS approximations and full Hessians) are referred to as TRSSN and TRSSN-H.

Evaluating the merit function $H(\tau, z)$ requires an additional computation of the gradient $\nabla f$, which can increase the overall costs. To reduce such calculations, we first conduct a linesearch test using the condition:
\be \label{eq:num-pmer} \psi(\prox{z_k+s_k(\alpha)})< H(\tau_{k-1},z_k). \ee
If \eqref{eq:num-pmer} fails, then it follows $H(\tau_k, z_k+s_k(\alpha))\geq \psi(\prox{z_k+s_k(\alpha)}) \geq H(\tau_{k-1},z_k) \ge H(\tau_{k},z_k)$, i.e., the linesearch condition \eqref{eq:armijo} can not hold. Moreover, if \eqref{eq:num-pmer} holds for some $\bar\alpha$, we then continue the original linesearch strategy \eqref{eq:armijo} using $\bar\alpha$ as the initial step size. This approach effectively reduces the frequency of computing $\Fnor{z_k+s_k}$. In \eqref{eq:cur} and \ref{C2}, we set $\eta=10^{-8}$, $\nu=10^{-3}$, $a_k := c k^{p}\ln^{2p}(k)$ and $b_k := c k^{q}\ln^{2q}(k)$ with $q=p=0.2$ and $c=10^{-3}$. We set $\bar L =1$, $\tau_{-1}=10^{-3}$ and $\gamma=0.9$ in \eqref{eq:Lk} (for computing $L_k$ and $\tau_k$) and $\sigma=10^{-4}$ in the linesearch condition \eqref{eq:armijo}. 

We refer to the CG method used within our linesearch framework \cite{dembo1983truncated} and within the trust region algorithm \cite[Algorithm 7.2]{NocWri06} as CG-LS and CG-TR, respectively. CG-LS differs from CG-TR primarily by the absence of trust region constraints. CG-LS also returns the iterate from the previous inner iteration when nonnegative curvature is detected. In early iterations, algorithms employing CG-LS often experience higher computational costs due to ``over-solving'' than those utilizing CG-TR, as the latter can
mitigate over-solving through the trust region radius. To enhance performance, we set the tolerance $\epsilon_k=\min\{\chi(z_k)^{1.4},0.1\}$ in LSSSN-H, which is more permissive than the tolerance $\min\{ \chi(z_k)^{2.5},0.01\}$ used in LSSSN, TRSSN, and TRSSN-H. The maximum number of CG iterations is capped at $10$ in LSSSN and TRSSN and $100$ in TRSSN-H. To further reduce over-solving, LSSSN-H also limits the maximum number of CG iterations to $10$ when $\chi(z_k)>10^{-4}$ and it is set to $100$ otherwise.  We stop the algorithms when $\|x_k -\text{prox}_{\varphi}(x_k - \nabla f(x_k))\| < 10^{-8}$.

\subsection{Sparse Logistic Regression}
\label{ssec:logistic}
First, we consider a sparse logistic regression problem of the form:
\begin{equation}
\label{eq:logreg-prob} 
\min_{x}~\psi(x)=f(x)+\varphi(x),\quad f(x):=\frac{1}{N}{\sum}_{i=1}^{N}f_i(x),\quad \varphi(x):=\mu\|x\|_1,
\end{equation}
where $f_i(x) :=\log(1+\exp(-b_i\cdot\langle a_i,x \rangle))$ denotes the logistic loss function and the data pairs $(a_i,b_i)\in\mathbb{R}^n\times\{-1,1\}$ are given. The Lipschitz constant of $\nabla f$ is given by $L={\|A\|_2^2}/(4N)$, where $A=(a_1,\dots,a_N)^\top\in\mathbb{R}^{N\times n}$. The proximity operator $\mathrm{prox}_{\mu\lambda\|\cdot\|_1}$ and its generalized derivatives are given by
\begin{equation*}
\begin{aligned}
 &\mathrm{prox}_{\mu\lambda\|\cdot\|_1}(z) = \mathrm{sgn}(z) \odot \max\left\{0,\vert z\vert-{\mu}{\lambda}\right\},\\
 &D(z) = \mathrm{diag}(d(z)) \quad \text{and} \quad d_{i}(z) = 
\begin{cases}
0  &\text{if~~}\vert z_i\vert\leq {\mu}{\lambda}, \\
1   & \text{otherwise}. 
\end{cases}
\end{aligned}
\end{equation*}
We compare our method (LSSSN, LSSSN-H) with its trust region counterparts (TRSSN, TRSSN-H), PNOPT \cite{lee2014proximal} and FISTA \cite{beck2009fast}.  PNOPT is a proximal Newton method, which uses a quasi-Newton approximation of $\nabla^2 f$ and does not need to compute the Lipschitz constant of $\nabla f$. FISTA is a first-order method with Nesterov-type acceleration. We use the known the Lipschitz constant as step size. For PNOPT, we use the source code released by the authors\footnote{\url{https://web.stanford.edu/group/SOL/software/pnopt/}} and all parameters are set to the default values, cf.\ \cite[Section 7.2]{ouyang2025trust}.

\begin{table}[t]
  \centering
   \begin{tabular}{lccclccc}  
 \cmidrule[1pt](){1-8}  
  Dataset & $N$ & $n$ & timeL(s) & Dataset & $N$ & $n$  & timeL(s) \\[0.5ex]  
  \cmidrule[.5pt](){1-8}\\[-1.5ex]
   \texttt{BIO} & 145\,751 & 75  & 0.13&\texttt{news20} & 19\,996  & 1\,355\,191 &0.22\\[0.5ex] 
   \texttt{CINA} & 16\,033  & 132 & 0.03& \texttt{rcv1} & 20\,242 & 47\,236  & 0.03 \\[0.5ex] 
   \texttt{covtype} & 581\,012 & 54 & 0.14& \texttt{real-sim} & 72\,309 & 20\,958 & 0.07    \\[0.5ex]    
\texttt{epsilon} & 400\,000 & 2\,000 &  106.7 & \texttt{gisette} & 6\,000 & 5\,000 & 0.07 \\[0.5ex]
\texttt{cifar10} & 50\,000 & 3\,072  &13.8    \\[0.5ex]
   \cmidrule[1pt](){1-8}\\[-1.5ex]
  \end{tabular}
  \caption{Information of the different datasets, ``timeL'' represents the time for computing $L$.}
  \label{table2}
  \end{table}

We test nine different datasets (\texttt{BIO}\footnote{\url{https://osmot.cs.cornell.edu/kddcup/datasets.html}}, \texttt{CINA}\footnote{\label{data}\url{http://www.causality.inf.ethz.ch/data}}, \texttt{covtype}\footnote{\label{libsvmnote}\url{https://www.csie.ntu.edu.tw/~cjlin/libsvmtools/datasets/}}, \texttt{gisette}\footref{libsvmnote}, \texttt{rcv1}\footref{libsvmnote}, \texttt{real-sim}\footref{libsvmnote}, \texttt{epsilon}\footref{libsvmnote}, \texttt{news20}\footref{libsvmnote}, \texttt{cifar10}\footref{libsvmnote}). More information about these datasets can be found in Table~\ref{table2}. We set $x^0 = 0$, $\lambda=10$  in all experiments and choose $\mu=0.01$ in \texttt{cifar10} and $\mu=0.002$ for all other datasets. 
The results of our comparison are shown in Figures~\ref{fig2a} and~\ref{fig2b}. We plot the relative error $\texttt{rel$\_$err} = (\psi(x)-\psi^*) /\max\{1,\psi^*\}$ with respect to the cpu-time in Figure~\ref{fig2a}. Here, $\psi^*$ is the lowest objective function value encountered by all algorithms during the experiment. The figure shows that LSSSN performs better than TRSSN in all datasets except \texttt{gisette}. LSSSN-H performs better than TRSSN-H in all datasets except \texttt{gisette} and \texttt{BIO}. Furthermore, PNOPT can quickly recover solutions with moderate accuracy, but is surpassed by the normal map-based algorithms in later iterations. FISTA, as a first-order method, generally performs worse than the other higher-order methods in this problem. For the large-scale datasets \texttt{epsilon} and \texttt{cifar10}, it is notable that the cpu-time gap between LSSSN (LSSSN-H) and TRSSN (LSSSN-H) is dominated by the time spent to compute $L$ in Table~\ref{table2}. 
Figure~\ref{fig2b} further shows that LSSSN (LSSSN-H) performs better than TRSSN (LSSSN-H) in terms of the number of iterations (except on \texttt{gisette}) and local superlinear convergence of LSSSN and TRSSN can be observed. These tests illustrate that the proposed linesearch-type normal map-based semismooth Newton method and the trust region counterpart (from \cite{ouyang2025trust}) have comparable performance. Moreover, adaptive estimation rather than explicit computation of $L$ improves the efficiency in large-scale problems.

\begin{figure}[t]
    \centering
    \subfigure[\texttt{CINA}.]{
    \includegraphics[width=3.6cm]{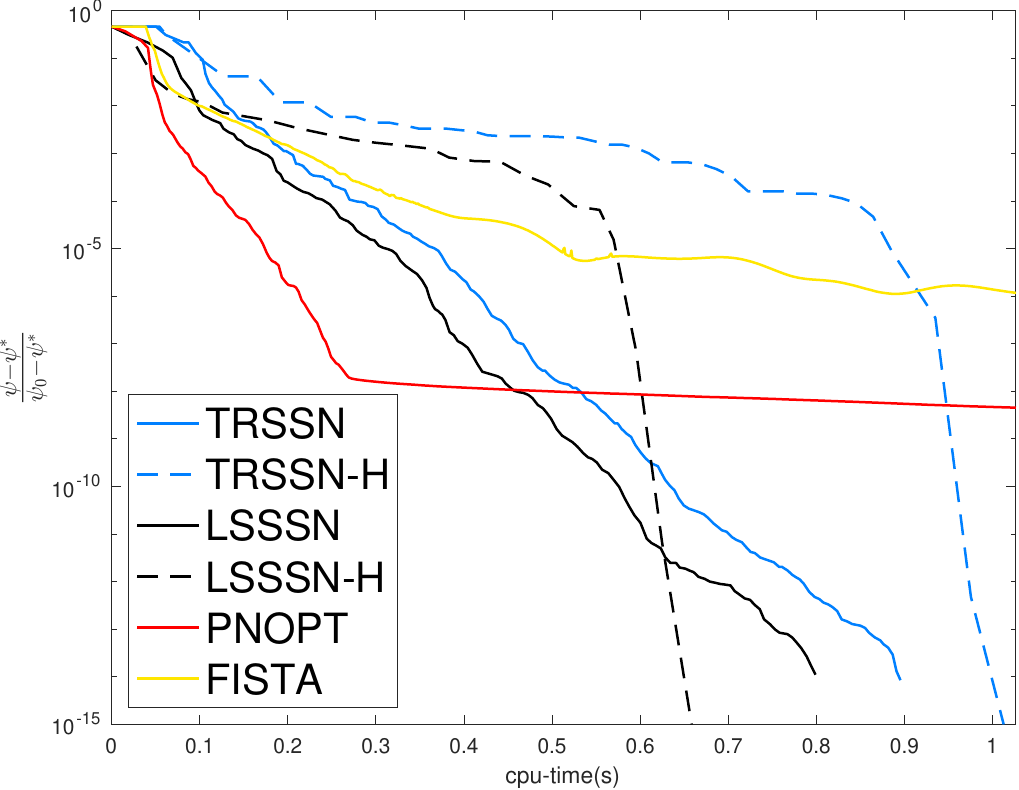} 
    }
    \subfigure[\texttt{BIO}.]{
    \includegraphics[width=3.6cm]{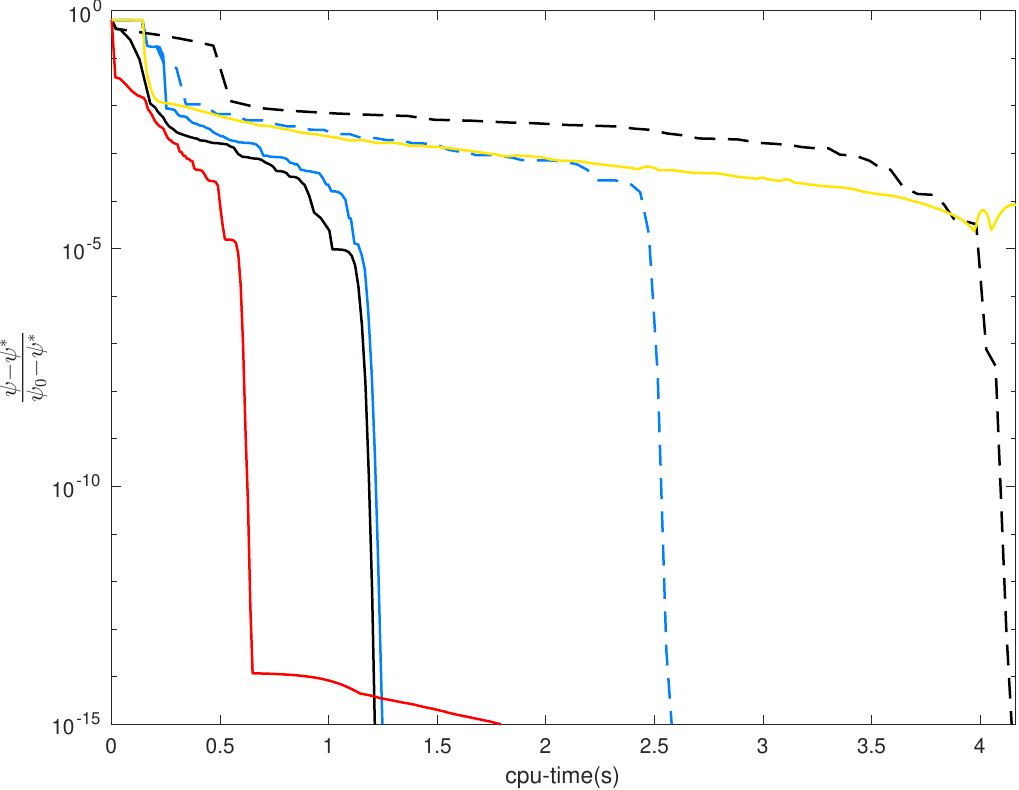}
    }
    \subfigure[\texttt{covtype}.]{
    \includegraphics[width=3.6cm]{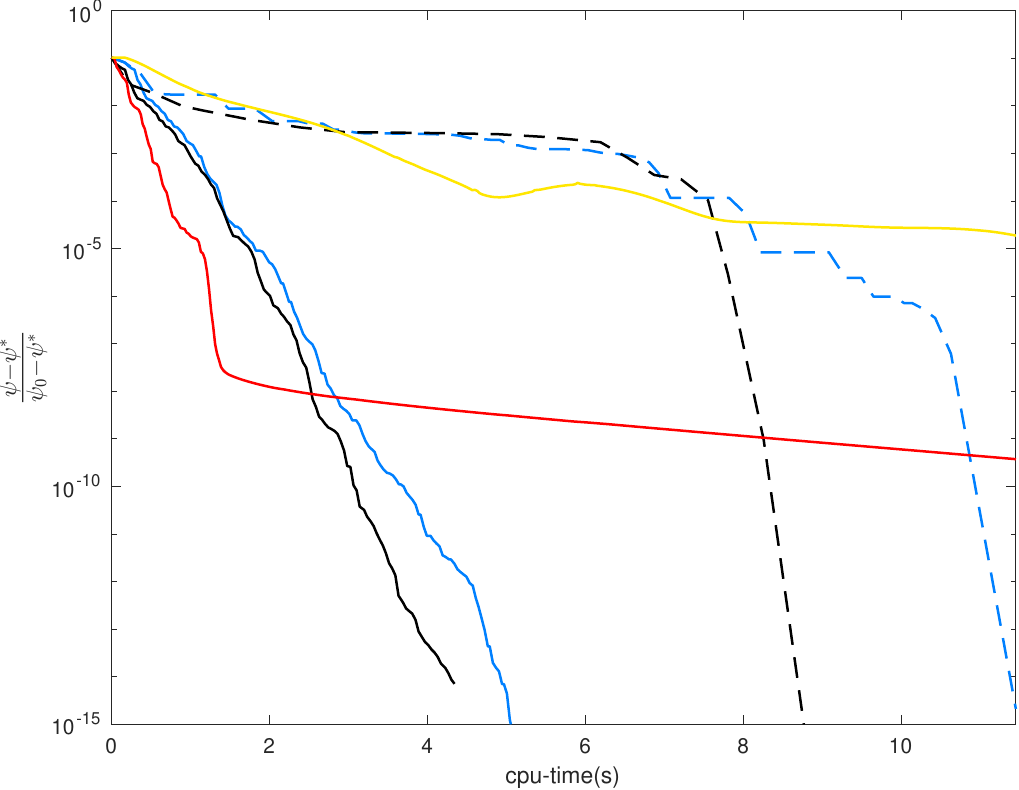}
    }
    \\[-1ex]
    \subfigure[\texttt{epsilon}.]{
    \includegraphics[width=3.6cm]{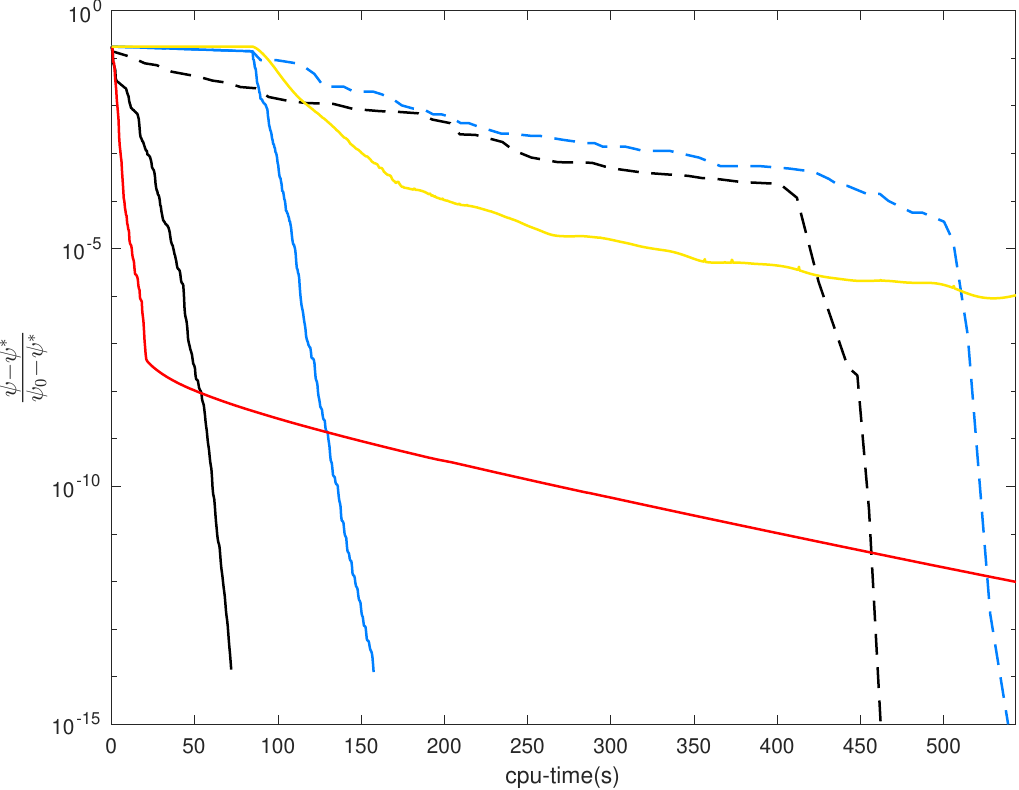}
    }
    \subfigure[\texttt{gisette}.]{
    \includegraphics[width=3.6cm]{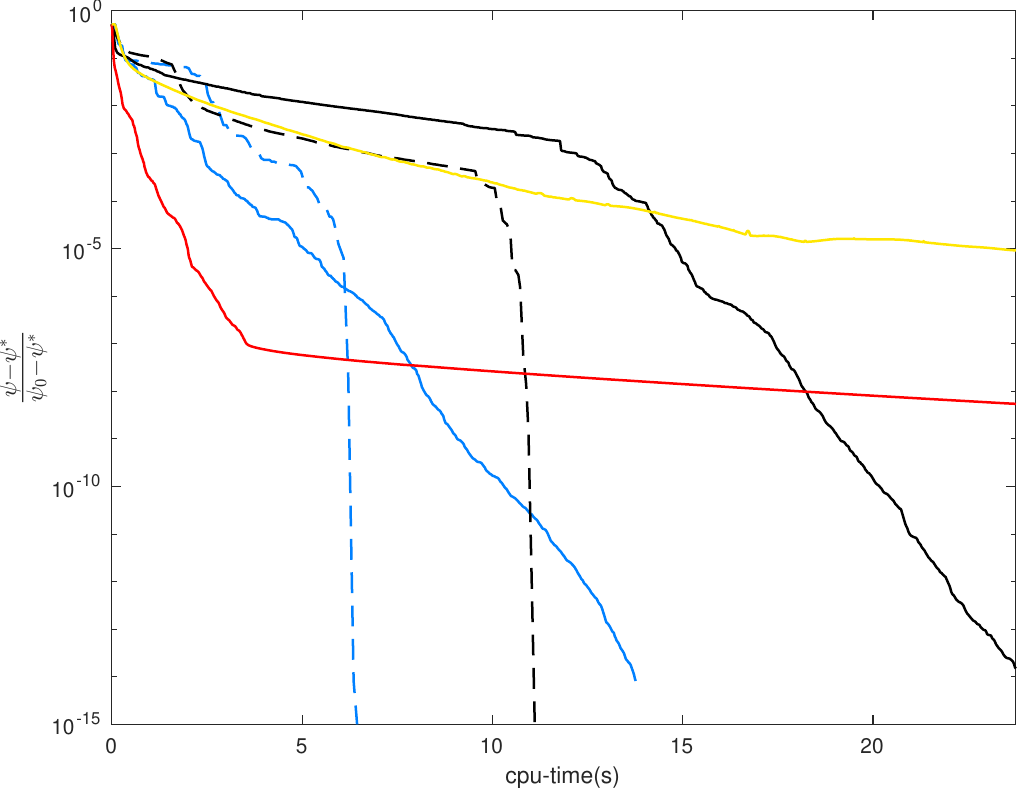}
    }
    \subfigure[\texttt{news20}.]{
    \includegraphics[width=3.6cm]{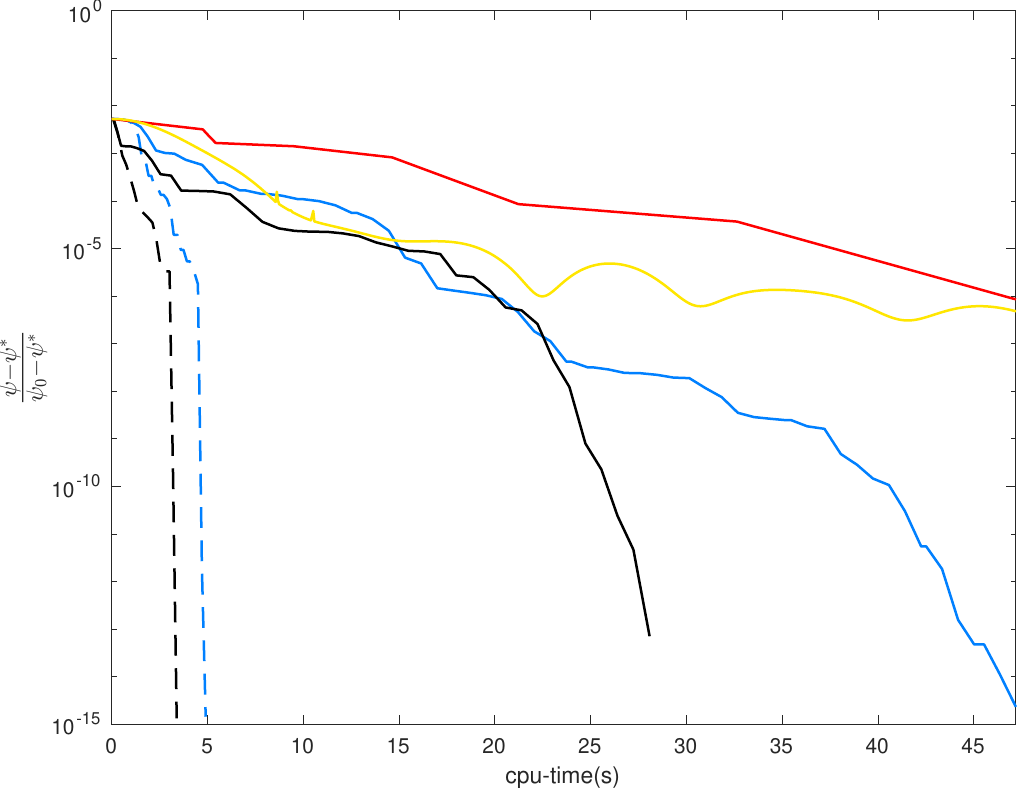} 
    }
    \\[-1ex]
    \subfigure[\texttt{rcv1}.]{
    \includegraphics[width=3.6cm]{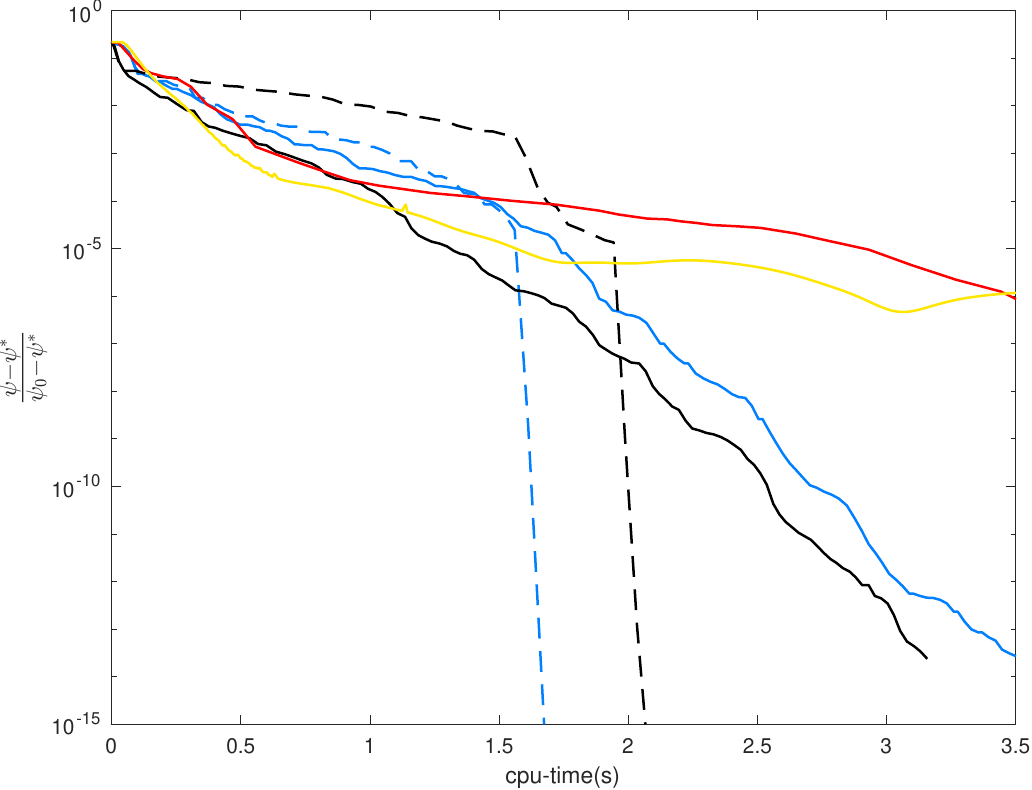}
    }
    \subfigure[\texttt{real-sim}.]{
    \includegraphics[width=3.6cm]{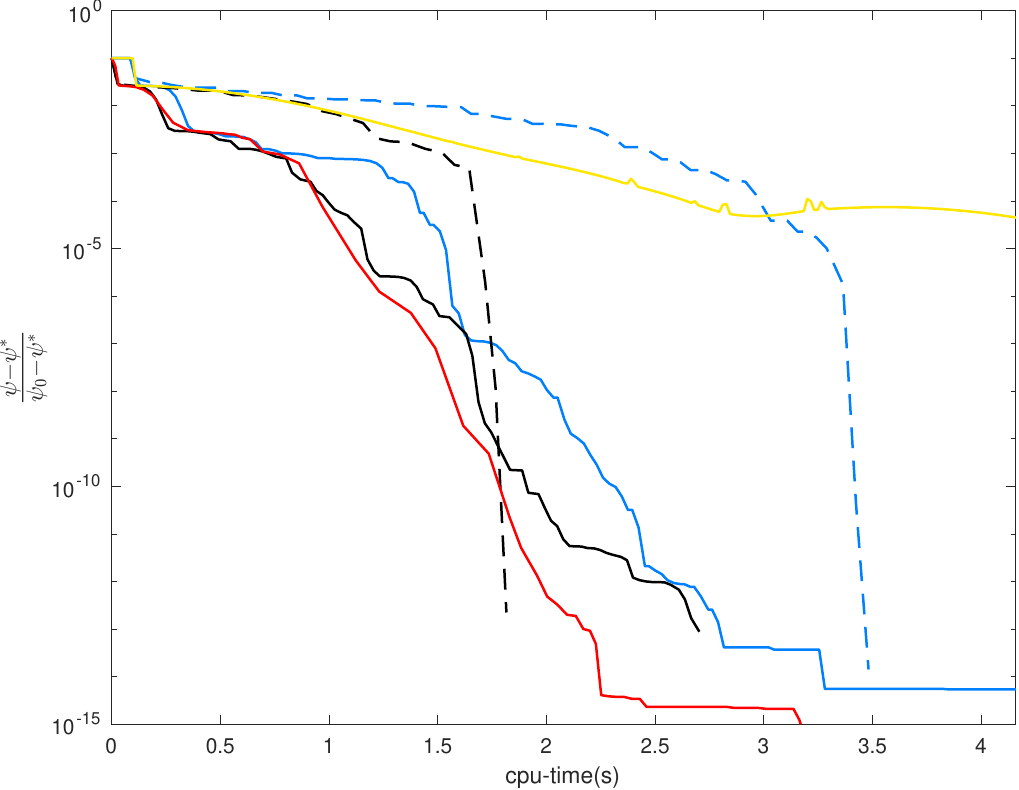}
    }
    \subfigure[\texttt{cifar10}.]{
    \includegraphics[width=3.6cm]{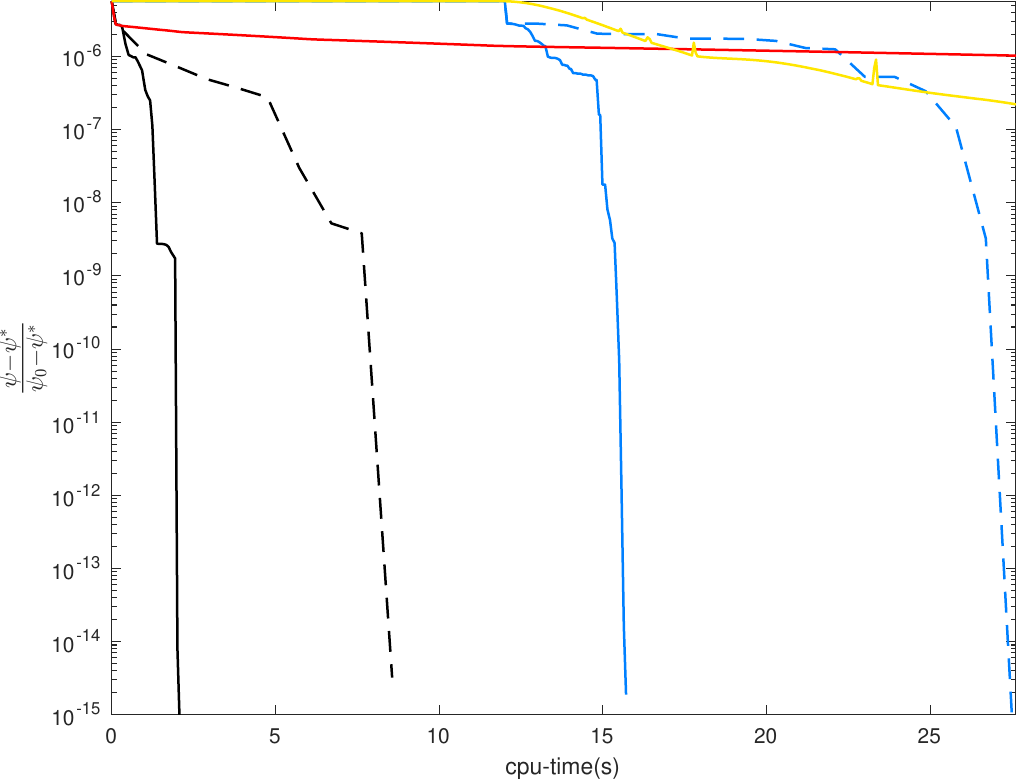}
    }
    \caption{Change of the relative error $\texttt{rel$\_$err}$ with respect to the cpu-time for solving the $\ell_1$-logistic
    regression problem \eqref{eq:logreg-prob}.}
    \label{fig2a}
    \end{figure} 

\begin{figure}[t]
    \centering
    \subfigure[\texttt{CINA}.]{
    \includegraphics[width=3.6cm]{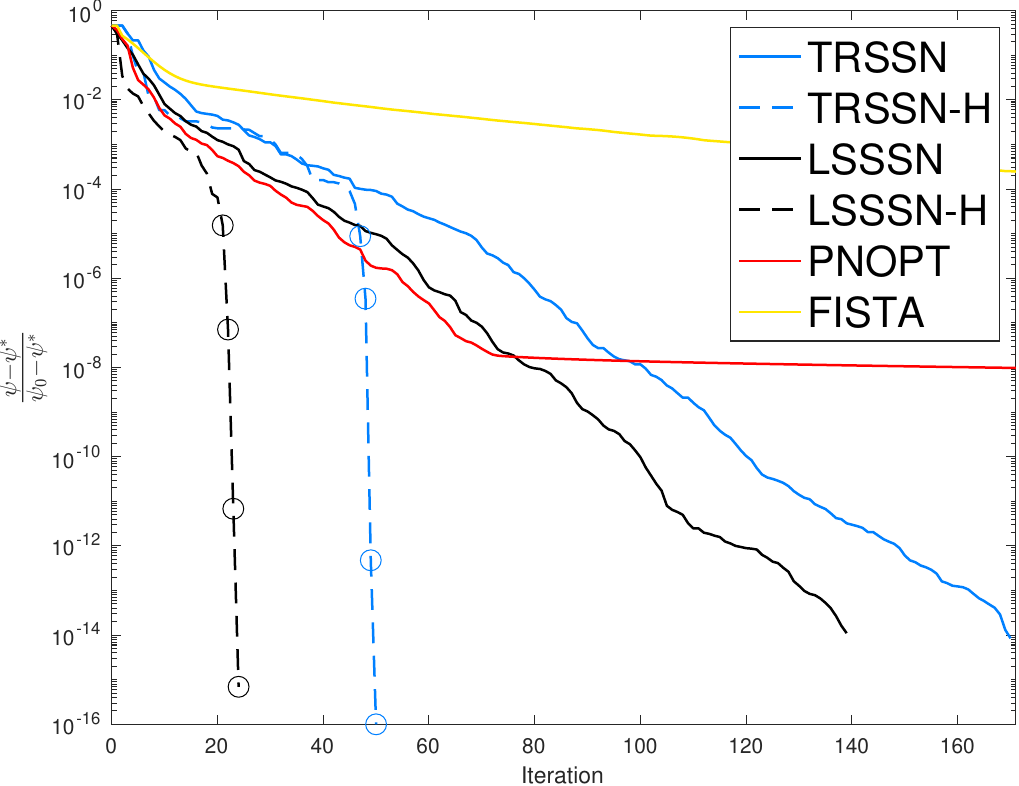} 
    }
    \subfigure[\texttt{BIO}.]{
    \includegraphics[width=3.6cm]{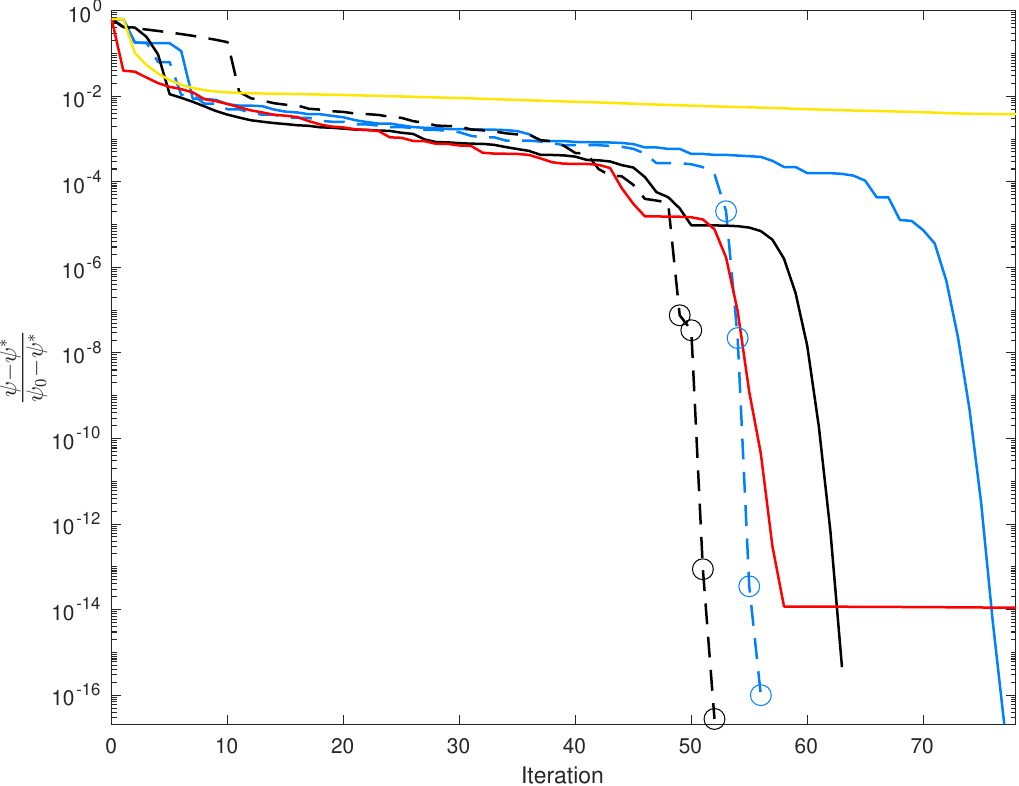}
    }
    \subfigure[\texttt{covtype}.]{
    \includegraphics[width=3.6cm]{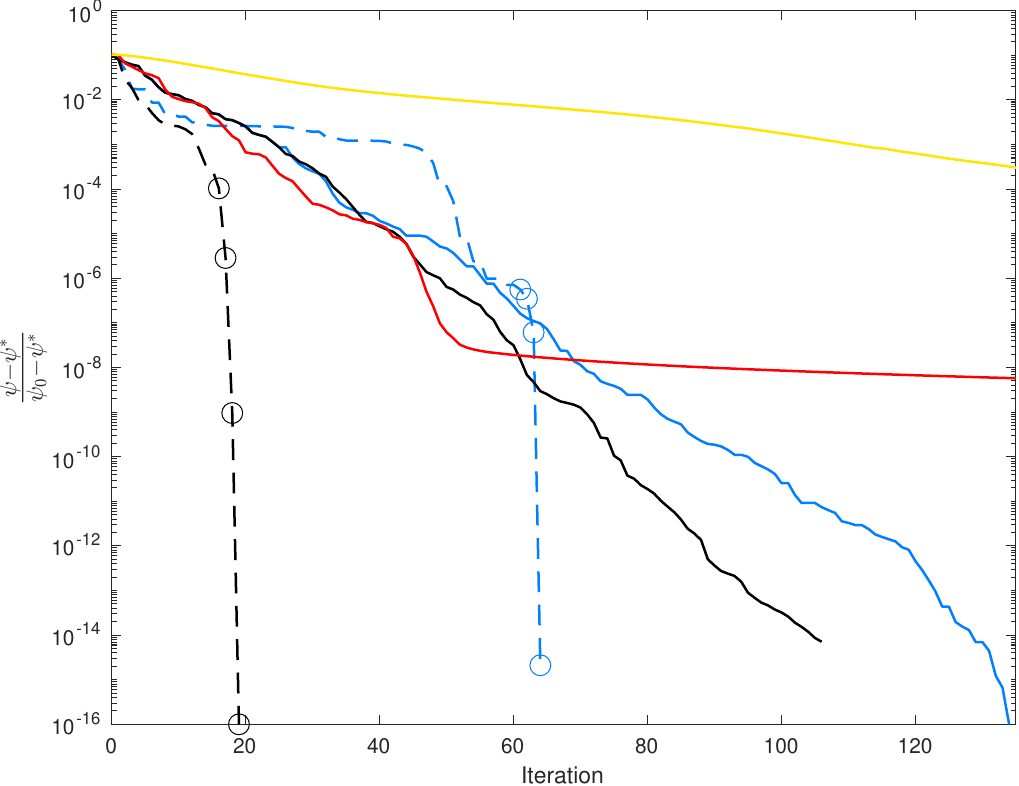}
    }
    \\[-1ex]
    \subfigure[\texttt{epsilon}.]{
    \includegraphics[width=3.6cm]{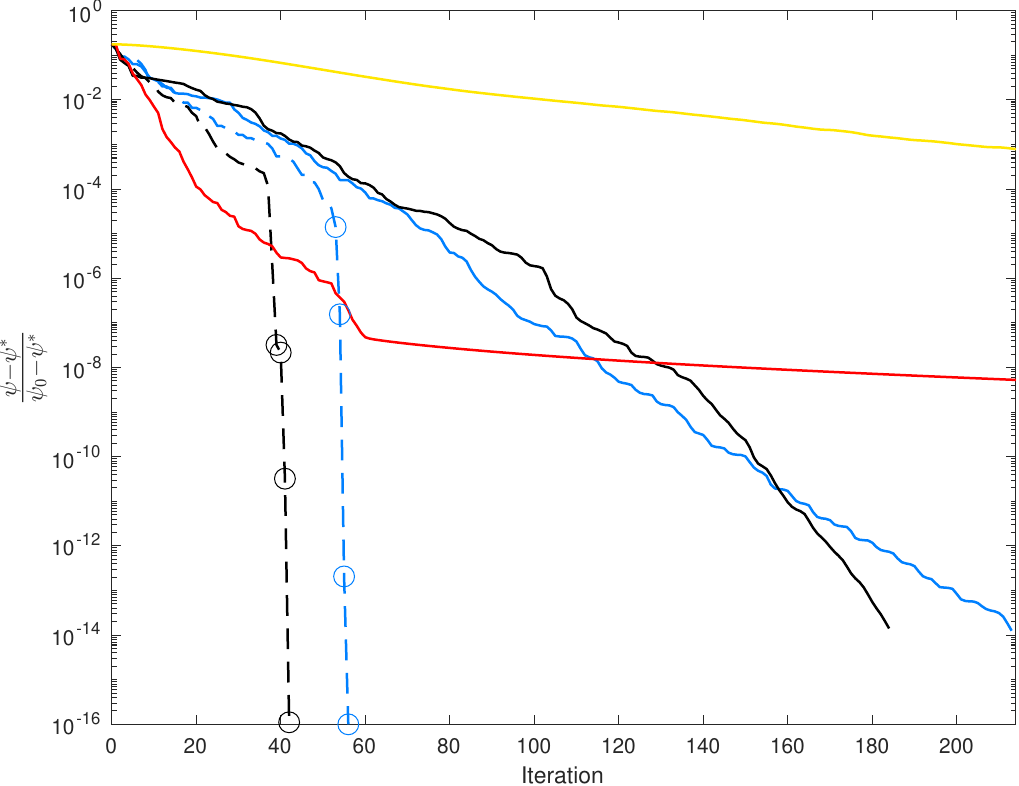}
    }
    \subfigure[\texttt{gisette}.]{
    \includegraphics[width=3.6cm]{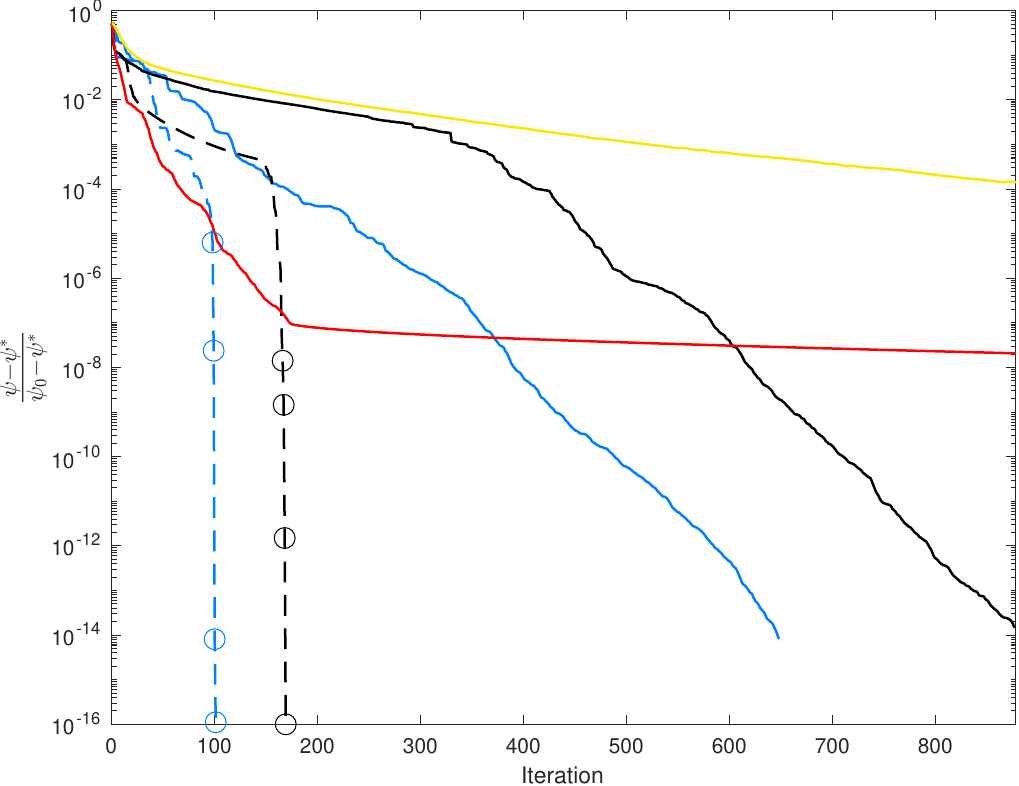}
    }
    \subfigure[\texttt{news20}.]{
    \includegraphics[width=3.6cm]{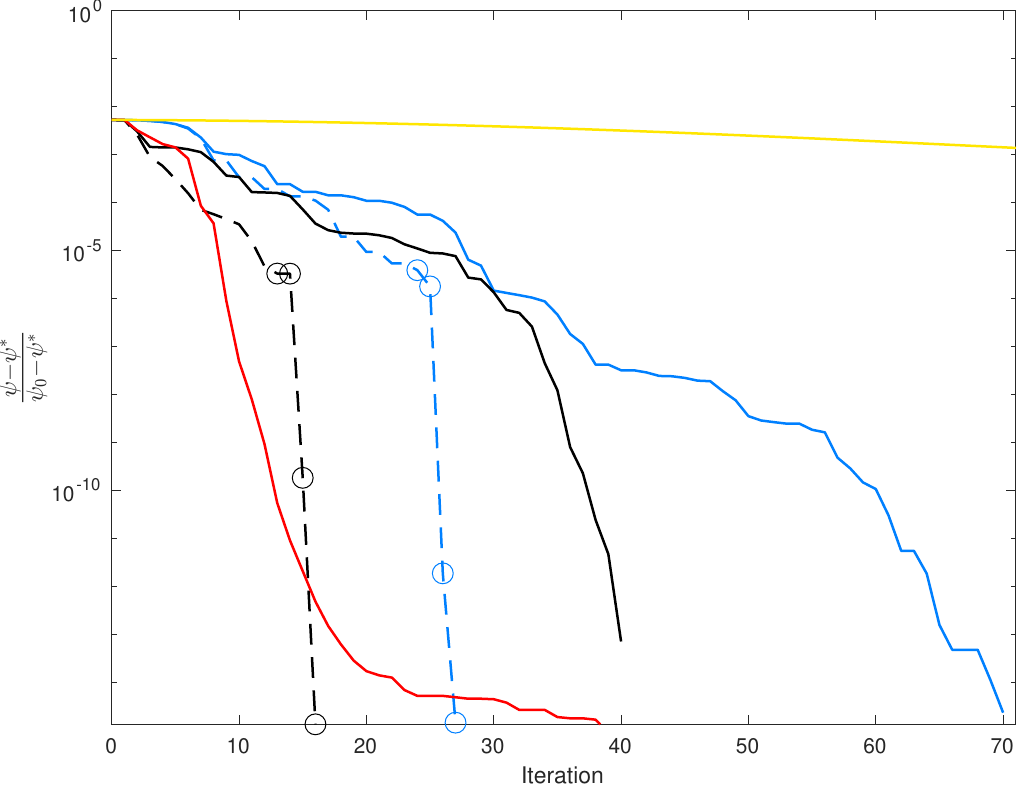} 
    }
    \\[-1ex]
    \subfigure[\texttt{rcv1}.]{
    \includegraphics[width=3.6cm]{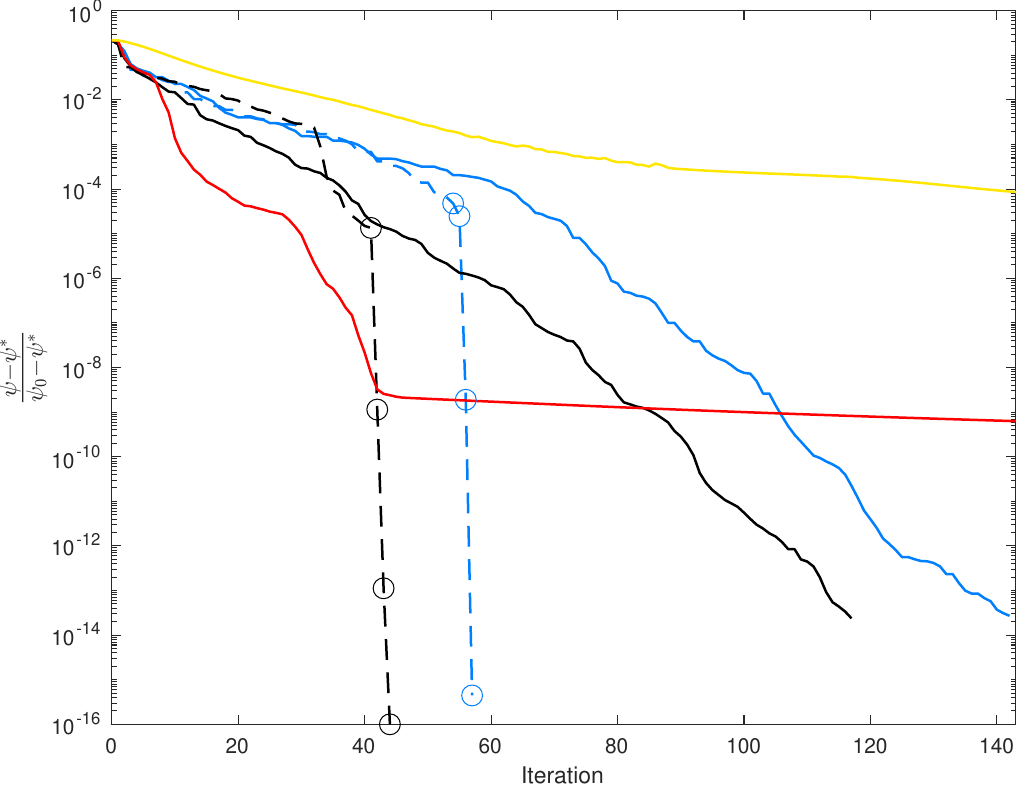}
    }
    \subfigure[\texttt{real-sim}.]{
    \includegraphics[width=3.6cm]{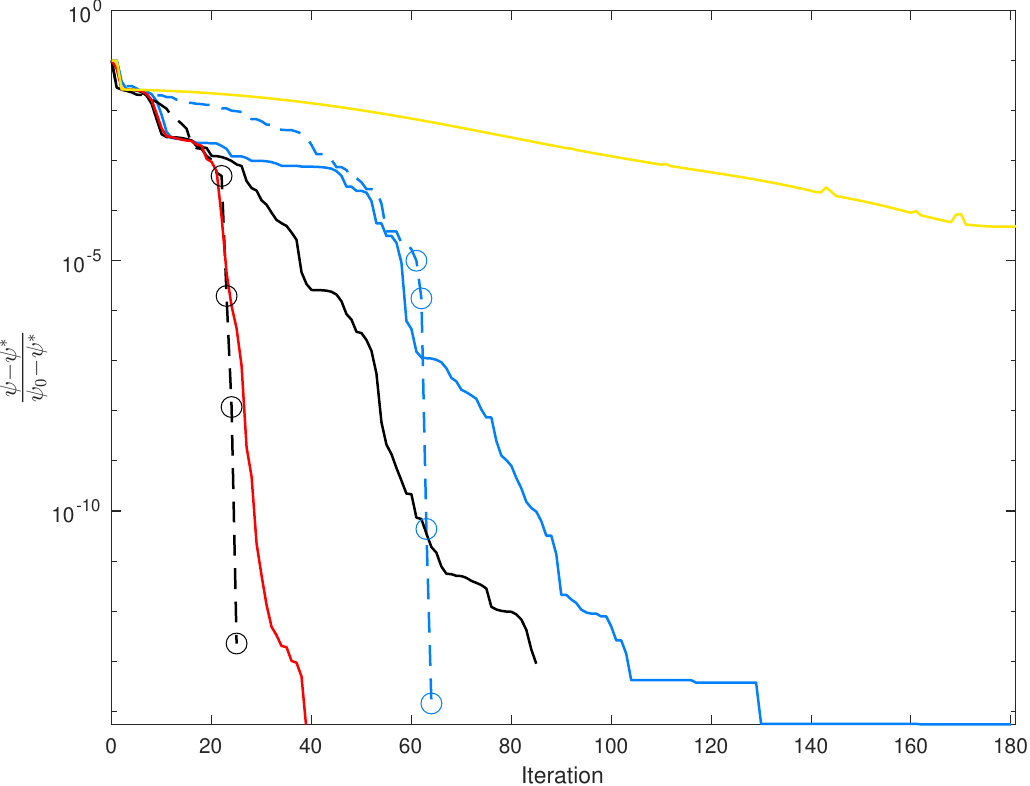}
    }
    \subfigure[\texttt{cifar10}.]{
    \includegraphics[width=3.6cm]{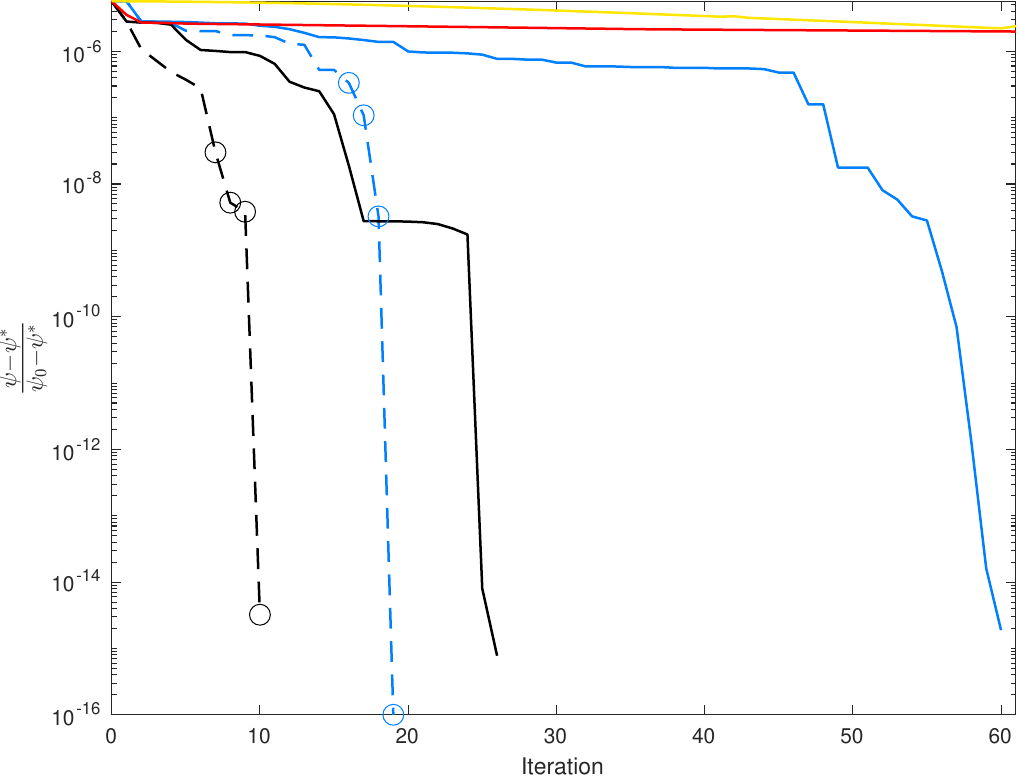}
    }
    \caption{Change of the relative error $\texttt{rel$\_$err}$ with respect to the number of iterations for solving the $\ell_1$-logistic
    regression problem \eqref{eq:logreg-prob}. We mark the last four iterations of TRSSN-H and LSSSN-H with $\circ$ to illustrate local superlinear convergence.} 
    \label{fig2b}
    \end{figure} 
    
 In Figure~\ref{fig1}, we discuss the performance of LSSSN and LSSSN-H on the datasets \texttt{CINA} and \texttt{rcv1} for different choices of $\epsilon_k$. Figure~\ref{fig1} illustrates that LSSSN is less sensitive than LSSSN-H in terms of the choice of inexactness parameters, especially regarding CPU time.  This shows the advantage of using the quasi-Newton approximation.  
 \begin{figure}[htbp]
    \centering
    \subfigure[\texttt{CINA} -- iterations.]{
    \includegraphics[width=5.5cm]{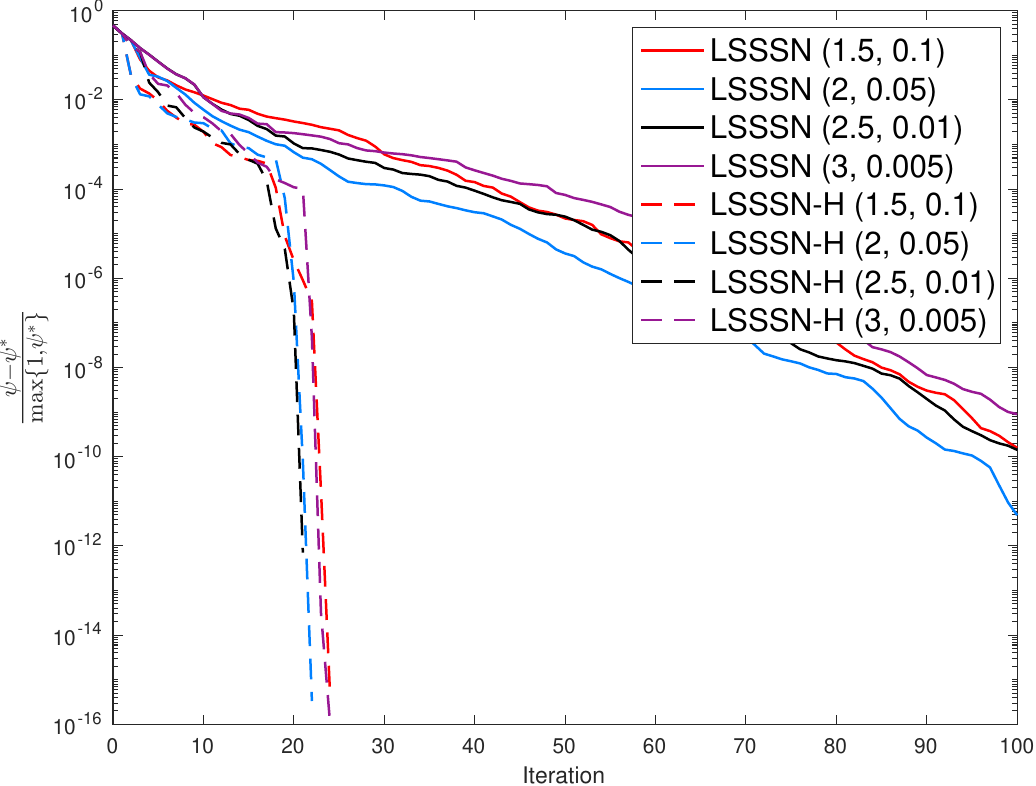}
    }
    \;
    \subfigure[\texttt{CINA} -- cpu-time.]{
    \includegraphics[width=5.5cm]{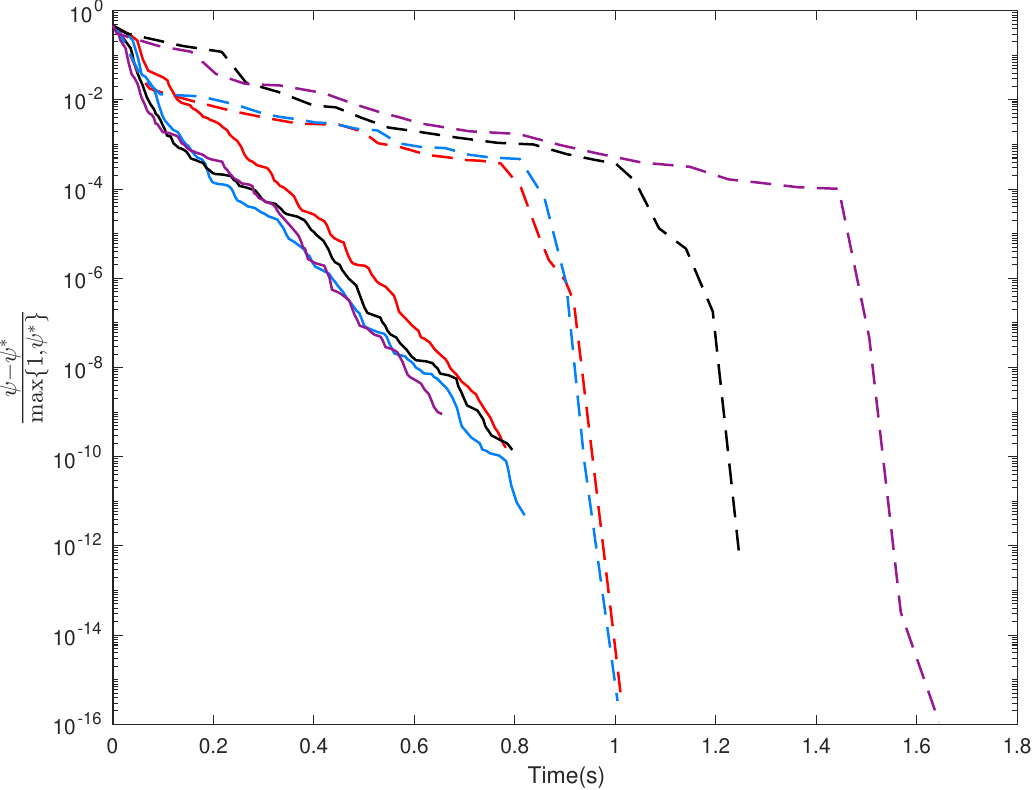}
    }
    \\[-1ex]
    \subfigure[\texttt{rcv1} -- iterations.]{
    \includegraphics[width=5.5cm]{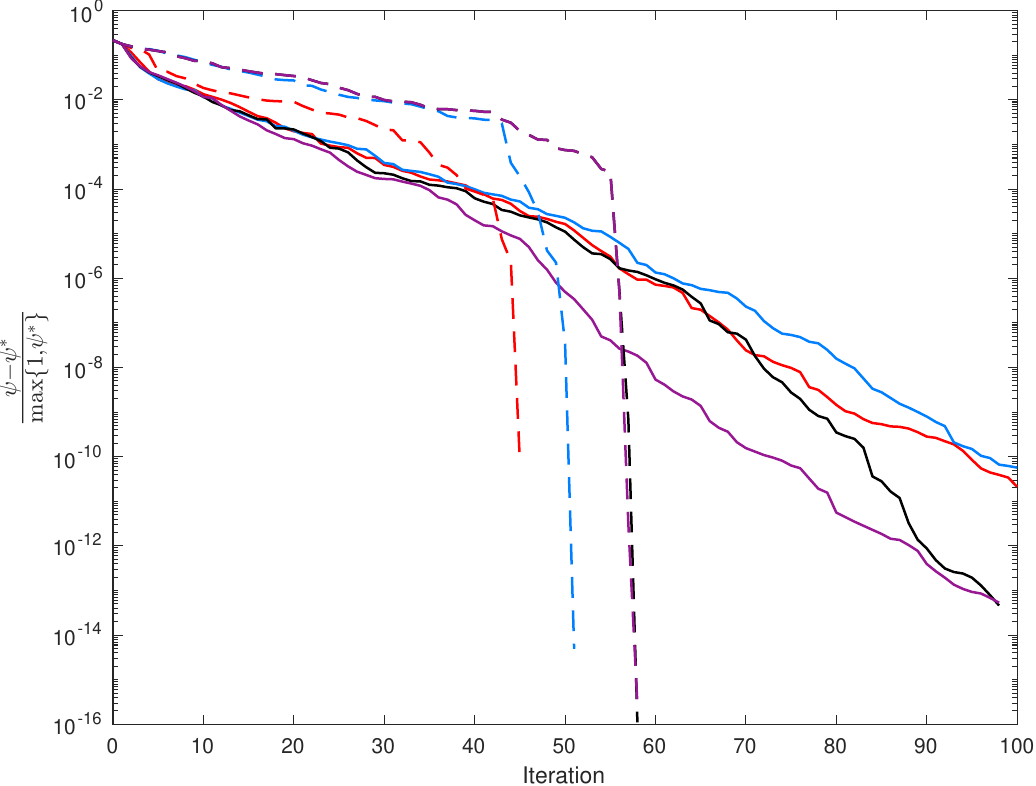}
    }
    \;
    \subfigure[\texttt{rcv1} -- cpu-time.]{
    \includegraphics[width=5.5cm]{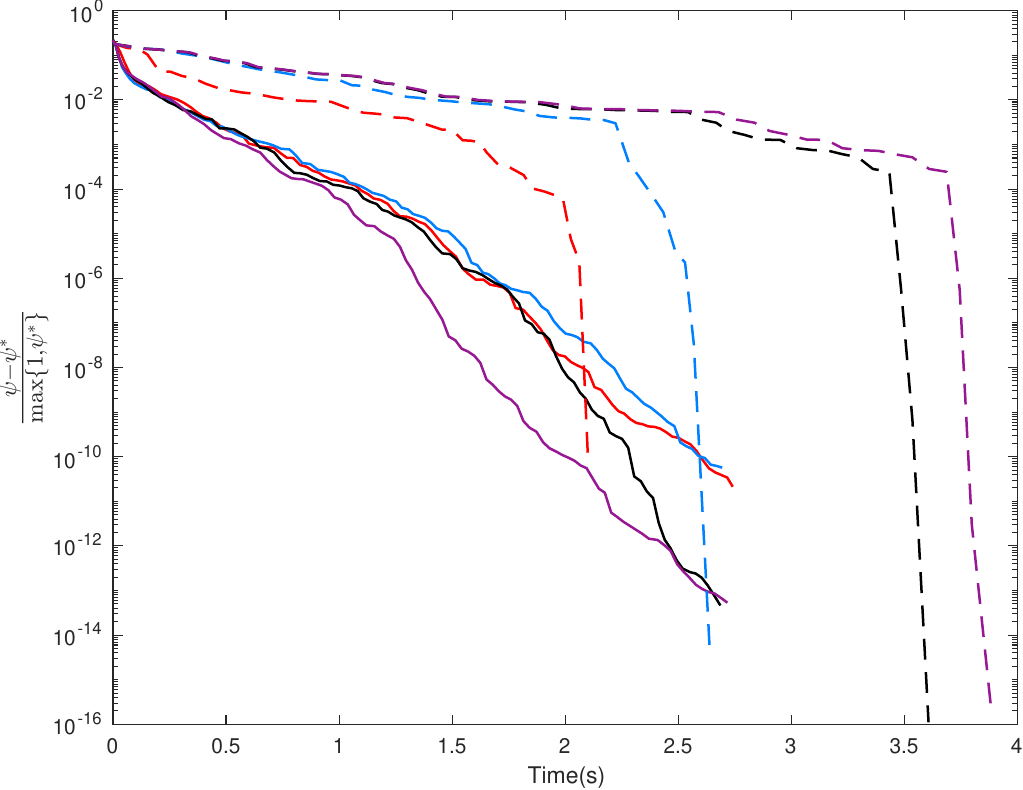}
    }
     
    \caption{Comparison of LSSSN and LSSSN-H on \texttt{CINA} and \texttt{rcv1} for different choices of $\epsilon_k = \min\{\chi(z_k)^{a},b\}$, where $(a,b) \in\{(1.5,0.1),(2,0.05),(2.5,0.01),(3,0.001)\}$.}
    \label{fig1}
    \end{figure}

\subsection{Linear Diffusion Based Image Compression}
\label{ssec:lineardiffusion}
Next, we test LSSSN on a linear diffusion based image compression problem \cite{galic2008image,schmaltz2009beating},  $\min_{c}~\psi(c):=f(c)+\varphi(c)$, where
\begin{equation}\label{eq:diff}
f(c):=\frac{1}{2}\|A(c)^{-1}\diag(c)u-u\|^2,\quad  \varphi(c):=\mu \|c\|_1+\iota_{[0,1]^n}(c). 
\end{equation}
Here, $u\in\mathbb{R}^n$ denotes the (stacked) ground truth image and $c\in\mathbb{R}^n$ is the inpainting or compression mask. This problem aims to find the optimal mask $c$ with the smallest density so that the image $u$ can still be reconstructed as $x=A(c)^{-1}\diag(c)u$ with reasonable quality using this mask. According to~\cite[Theorem 1]{yu2013decomposing}, it holds that
\begin{equation*}
\begin{aligned}
&\prox{z}=\mathcal P_{[0,1]^n}\circ\mathrm{prox}_{\mu\lambda \|\cdot\|_1}(z)=\max\{0,\min\{1,\mathrm{prox}_{\mu\lambda \|\cdot\|_1}(z)\}\},\\
&D(z) = \mathrm{diag}(d(z)) \quad \text{and} \quad d_{i}(z)= 
\begin{cases}
0 &  \text{if } z_i\leq {\mu}{\lambda} \; \text{ or } \; z_i\geq {\mu}{\lambda}+1 , \\
1 & \text{otherwise}.
\end{cases}
\end{aligned}
\end{equation*}
We refer to \cite[Section 7.3]{ouyang2025trust} for further details. We compare LSSSN with TRSSN, \text{iPiano}~\cite{OchCheBroPoc14}, and \text{SpaRSA}~\cite{WriNowFig09}. As the Lipschitz constant of $\nabla f$ can not be computed exactly, TRSSN also employs an adaptive strategy to estimate $L$ via $L_k$ and resets $\lambda=1/L_k$ each iteration. (Similar adaptive rules are also used in \text{iPiano} and \text{SpaRSA}, see \cite[Section 7.3]{ouyang2025trust}). Similar to TRSSN, we also reset $\lambda=1/L_k$ in LSSSN but only when $\|\Fnor{z_k}\| > 10^{-4}$. For all algorithms, we initialize $c_0$ as the vector with all components equal to $1$.
\begin{figure}[thbp]
  \centering
  \subfigure[\texttt{books} -- cpu-time.]{
  \includegraphics[width=5.5cm]{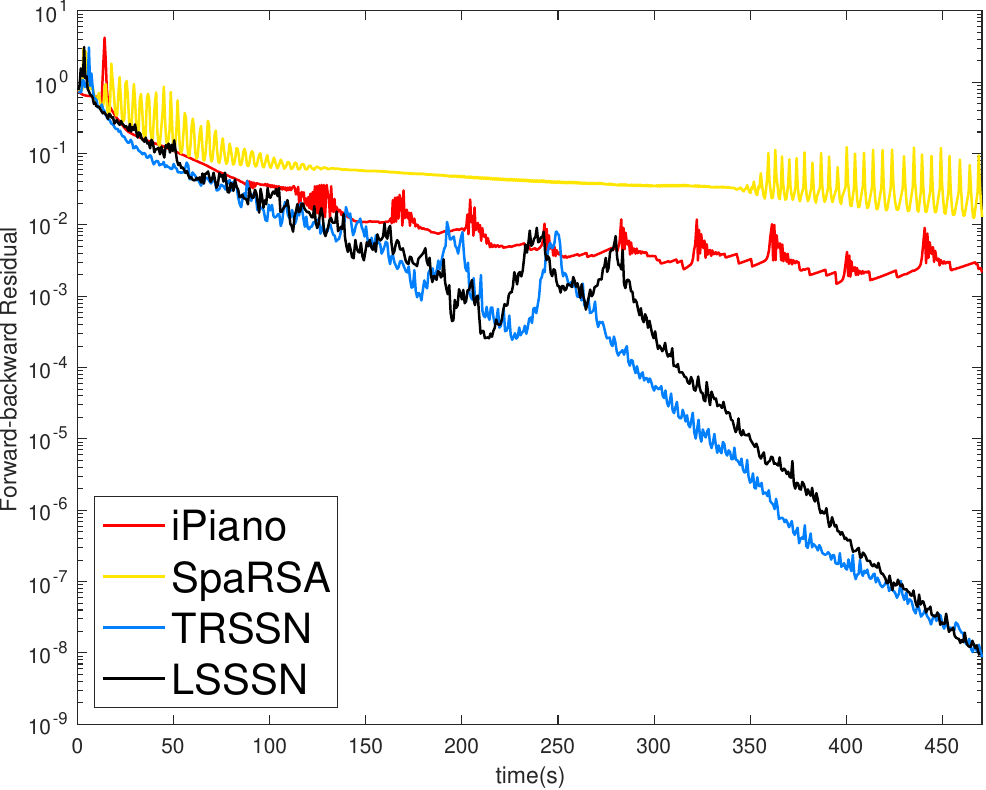}
  }
  \subfigure[\texttt{coffee} -- cpu-time.]{
  \includegraphics[width=5.5cm]{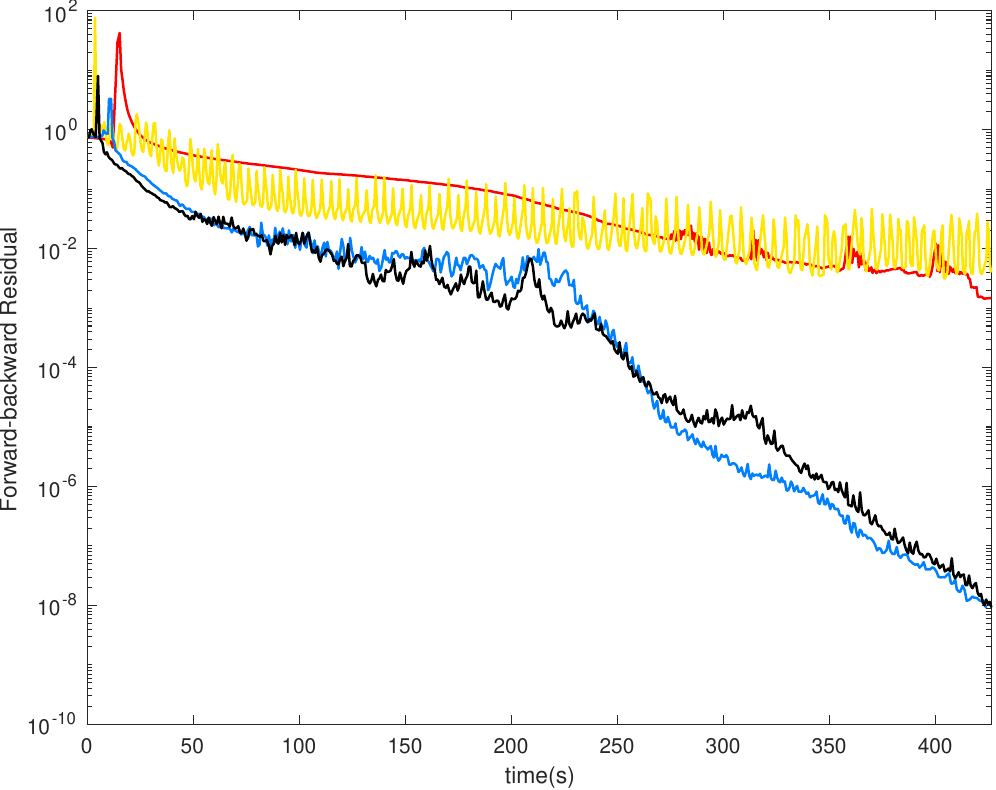}
  }
   \\[-1ex]
  \subfigure[\texttt{mountain} -- cpu-time.]{
  \includegraphics[width=5.5cm]{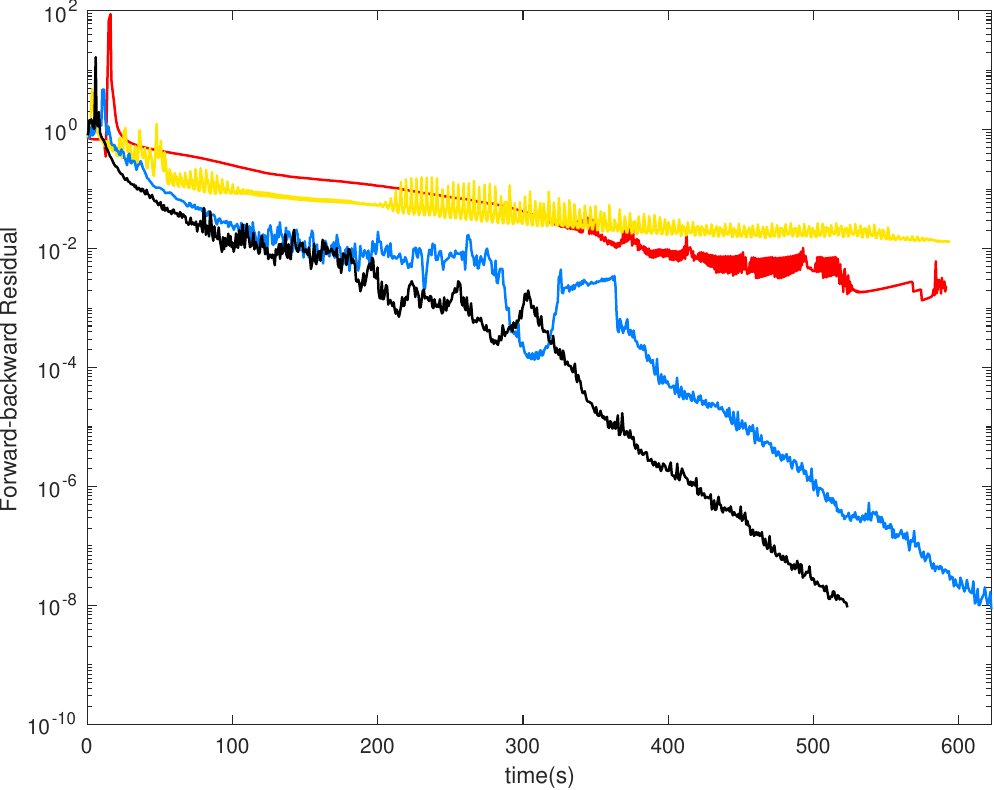}
  }
   \subfigure[\texttt{stones} -- cpu-time.]{
  \includegraphics[width=5.5cm]{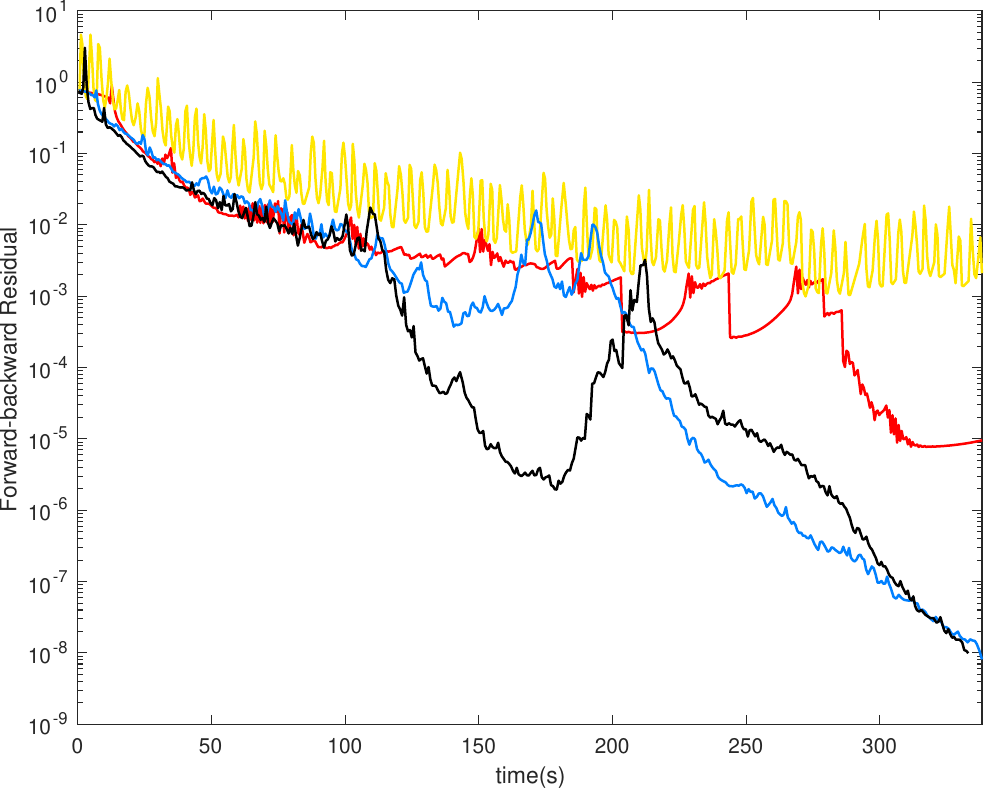}
  }
  \caption{Numerical comparison of iPiano, SpaRSA, TRSSN, and LSSSN on problem \eqref{eq:diff}. Plot of the norm of the natural residual with respect to the cpu-time for different images.}
  \label{fig3-1}
  \end{figure}

\begin{figure}[thbp]
  \centering
  \subfigure[$\mu=0.01$ -- cpu-time.]{
  \includegraphics[width=3.6cm]{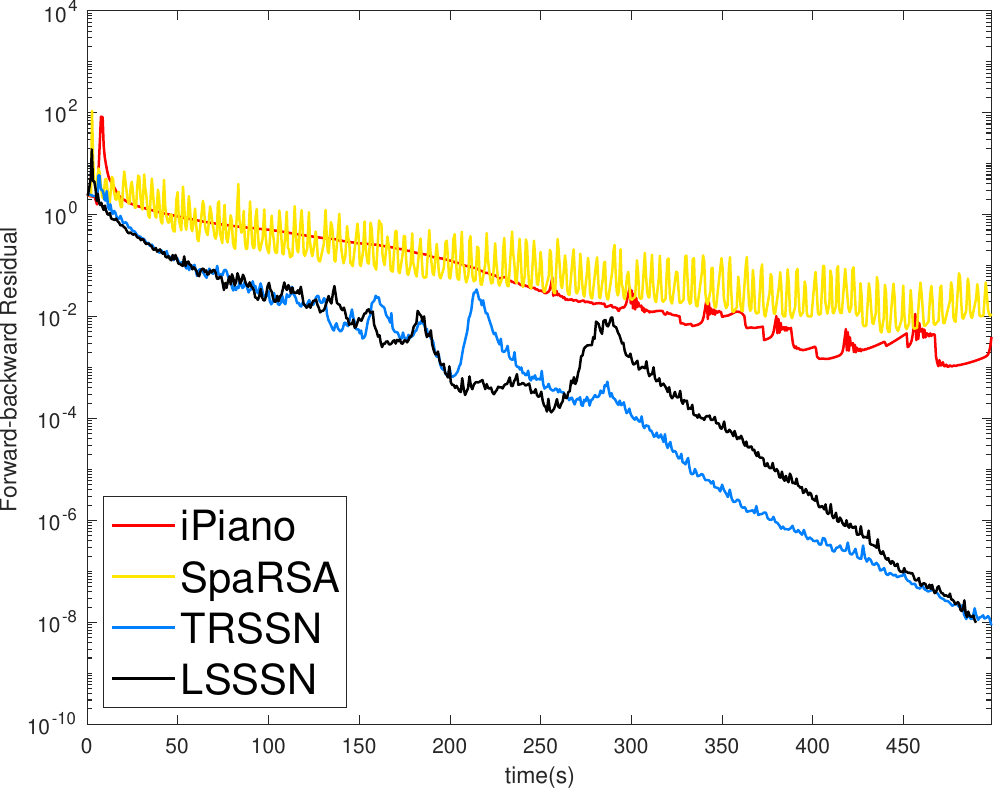}
  }
  \subfigure[$\mu=0.006$ -- cpu-time.]{
  \includegraphics[width=3.6cm]{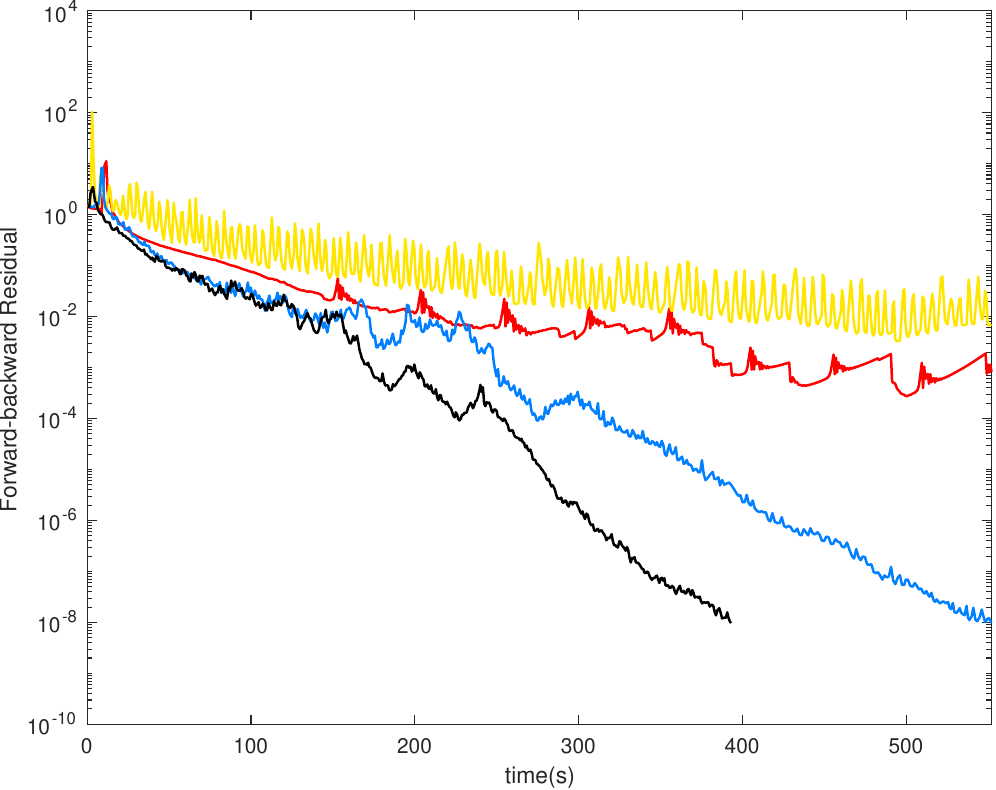}
  }
  \subfigure[$\mu=0.003$ -- cpu-time.]{
  \includegraphics[width=3.6cm]{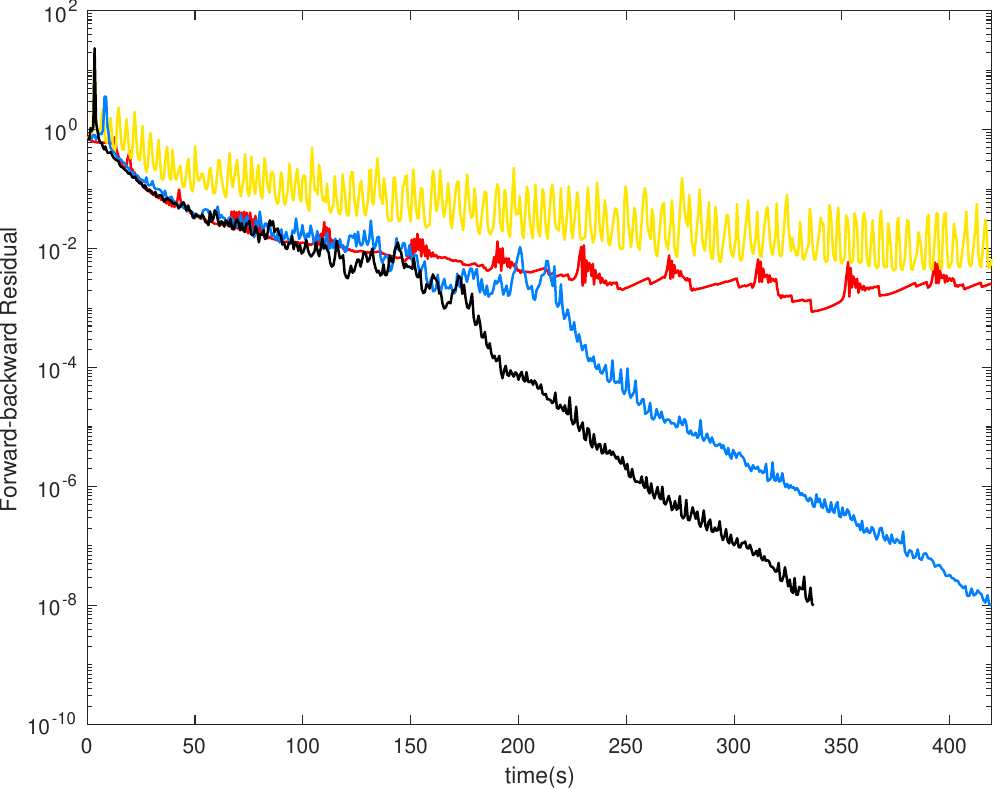}
  }
  \\[-1ex]
   \subfigure[$\mu=0.01$ -- cpu-time.]{
  \includegraphics[width=3.6cm]{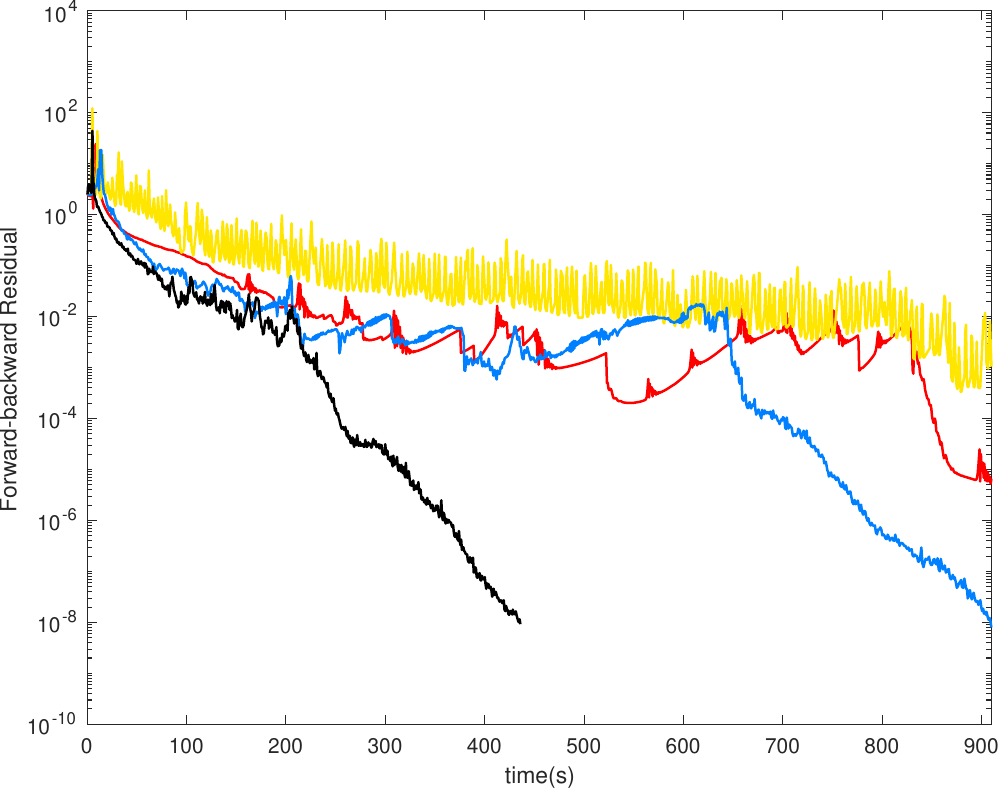}
  }
  \subfigure[$\mu=0.006$ -- cpu-time.]{
  \includegraphics[width=3.6cm]{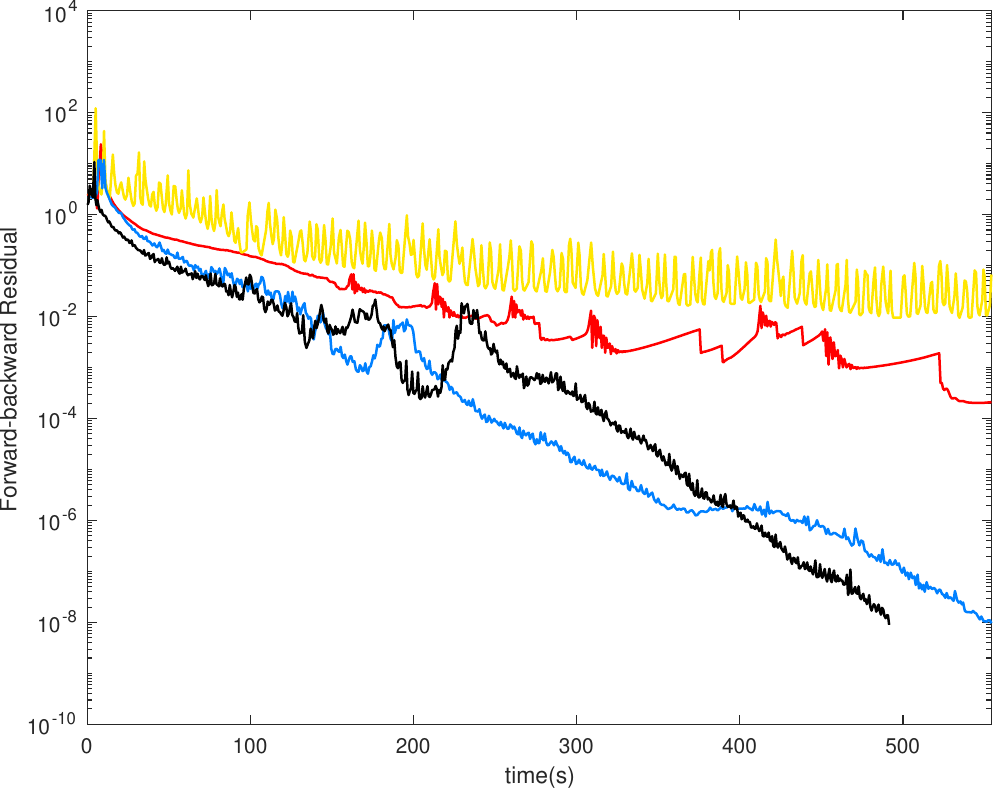}
  }
  \subfigure[$\mu=0.003$-- cpu-time.]{
  \includegraphics[width=3.6cm]{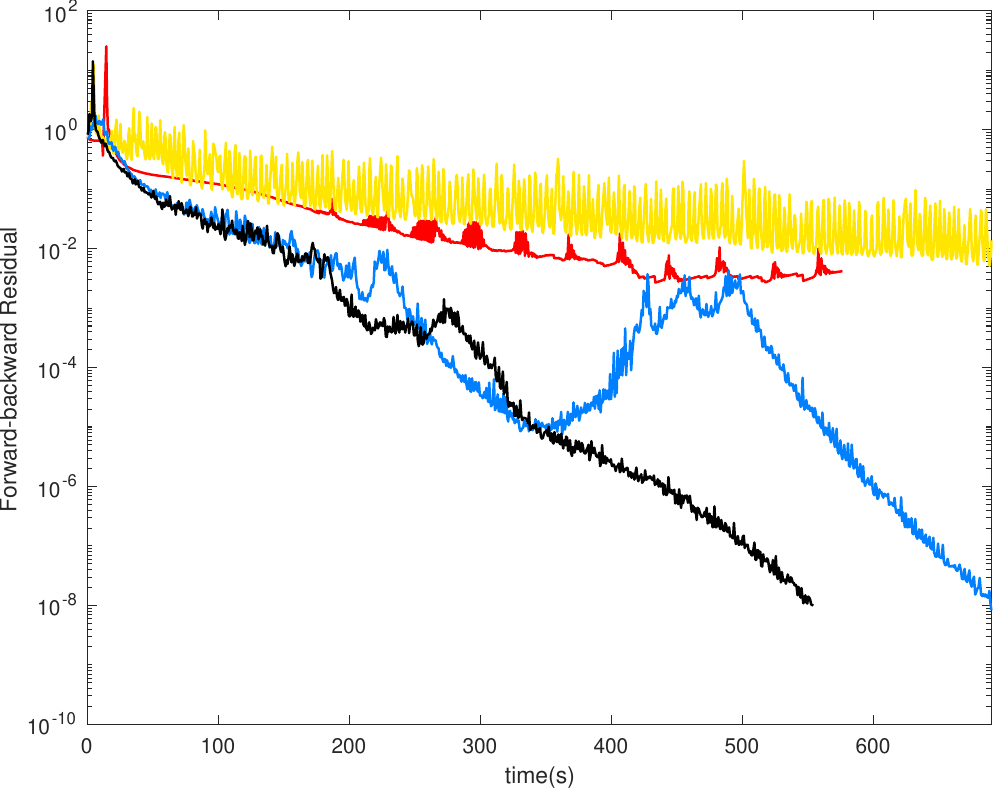}
  }
  \caption{Numerical comparison of iPiano, SpaRSA, TRSSN, and LSSSN on  problem \eqref{eq:diff}. Plot of the norm of the natural residual with respect to the required cpu-time for different $\mu$. The test images in the  subfigures (a)-(c) and (d)-(f) are \texttt{football} and \texttt{city}, respectively.}
  \label{fig3-2}
  \end{figure} 
 
  \begin{figure}[t!]
    \centering
    \subfigure[\texttt{mountain}, original figure]{
    \includegraphics[width=3.6cm]{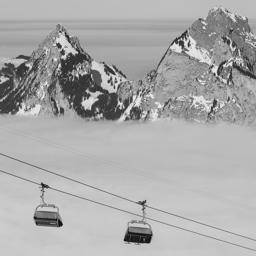}
    }
    \subfigure[mask (${ds}\!=\!11.35\%$)]{
    \includegraphics[width=3.6cm]{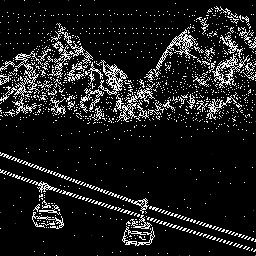}
    }
    \subfigure[reconstruction ($\mu=0.003$)]{
    \includegraphics[width=3.6cm]{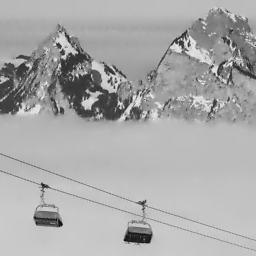}
    }
    
    \subfigure[\texttt{football}, original figure]{
      \includegraphics[width=3.6cm]{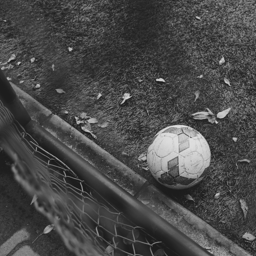}
      }
      \subfigure[ mask (${ds}\!=\!8.09\%$)]{
      \includegraphics[width=3.6cm]{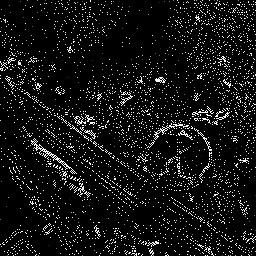}
      }
      \subfigure[reconstruction ($\mu=0.006$)]{
      \includegraphics[width=3.6cm]{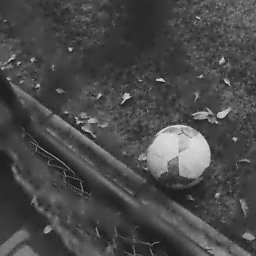}
    }
    \subfigure[\texttt{city}, original figure]{
      \includegraphics[width=3.6cm]{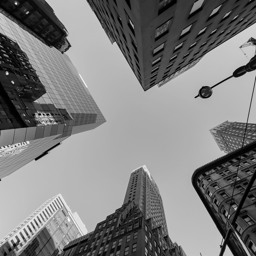}
      }
      \subfigure[mask (${ds}\!=\!7.13\%$)]{
      \includegraphics[width=3.6cm]{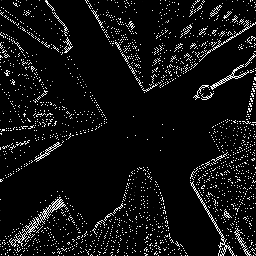}
      }
      \subfigure[reconstruction ($\mu=0.01$)]{
      \includegraphics[width=3.6cm]{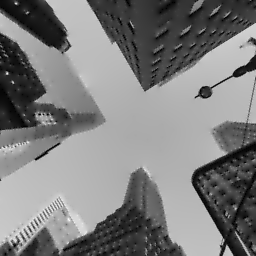}
      }
    \caption{Illustration of the different ground truth images, the inpainting or compression masks, and the corresponding reconstructions for \texttt{mountain}, \texttt{football}, and \texttt{city}. The density ($ds$) of the masks $c$ is calculated via ${ds} := 100\% \cdot [ \vert\{i: c_i > 0\}\vert / (256 \times 256)]$.}
    \label{fig4}
    \end{figure} 

    Our results are shown in Figures~\ref{fig3-1}, \ref{fig3-2} and \ref{fig4}.  We compare the performance of the algorithms on the six images  \texttt{books}, \texttt{coffee}, \texttt{mountain}, \texttt{stones}, \texttt{football} and \texttt{city}\footnote{Image credentials: \texttt{books} by Suzy Hazelwood; \texttt{coffee} by Atichart Wongubon; \texttt{mountain} by Denis Linine; \texttt{stones} by Travel Photographer; \texttt{football} by chancema, \texttt{city} by FOCA stock, all images can be found on StockSnap.}. All images are rescaled to size $256\times 256$, and we use $\mu = 0.003$ in Figure~\ref{fig3-1} and $\mu\in\{0.01, 0.006, 0.003\}$ in Figure~\ref{fig3-2}. LSSSN and TRSSN perform comparably on the images \texttt{books}, \texttt{coffee}, and \texttt{stones} in terms of cpu-time. Furthermore, for the remaining tested images, LSSSN outperforms all other algorithms, demonstrating its efficiency across various image types. As shown in Figure~\ref{fig3-2}, LSSSN requires less computational time if $\mu$ is small for tests on the image \texttt{football}. In addition, the performance of LSSSN is less sensitive to the choice of $\mu$ on the image \texttt{city}, while other algorithms appear to be less robust. As illustrated in Figure~\ref{fig4}, the parameter $\mu$ directly influences the sparsity of $c$---with higher values leading to sparser masks.

\subsection{Nonlinear Least-Squares With Group Lasso Penalty}\label{ssec:nonlinear}
Finally, we consider the nonlinear least square problem \cite{xu2020second} with $\min_{x}~\psi(x)=f(x)+\varphi(x)$ and
\begin{equation} \label{eq:prob-group}
f(x):=\frac{1}{2N}\sum_{i=1}^N(\sigma(\iprod{a_i}{x})-b_i)^2,\quad \varphi(x):=\mu\sum_{j=1}^p\|x_{g_j}\|_2,
\end{equation}
where $\sigma : \R \to \R$ is the sigmoid function $\sigma(x) := (1+\exp(-x))^{-1}$, $A=(a_1,\dots,a_N)^\top\in\mathbb{R}^{N\times n}$ is the data matrix and $b=(b_1,\dots,b_N)^\top$ again denotes the corresponding labels.  Moreover, the vector $x_{g_j}\in \R^{n_j}$ consists of a group of components of $x$. We randomly generate these groups without overlap---each group containing 16 elements. The Lipschitz constant $L$ of $\nabla f$ is given by $\|A\|^2/(12N)$. According to \cite{parikh2014proximal,milzarek2016numerical}, we have
\begin{equation*}
\begin{aligned}
&[\proxs(z)]_{g_j}= \max\Big\{0,1-\frac{\mu\lambda}{\|z_{g_j}\|}\Big\}z_{g_j},\quad 
D(z)_{[g_ig_j]}=0 \text{ if $i\neq j$},\\
&D(z)_{[g_ig_i]}=
\begin{cases}
\left(1-\frac{\mu\lambda}{\|z_{g_i}\|}\right)I+\frac{\mu\lambda}{\|z_{g_i}\|^3}z_{g_i}z_{g_i}^\top  & \|z_{g_i}\|> \mu\lambda,\\
0 & \text{otherwise}.
\end{cases}
\end{aligned}
\end{equation*}
We set $\lambda=10/L$, $\mu = 2/N$ and $x_0 = 0$. (Here, for a better numerical performance, we choose $\lambda$ in LSSSN to depend on $L$). We test six different datasets (\texttt{gisette}, \texttt{covtype}, \texttt{rcv1}, \texttt{CINA}, \texttt{news20}, \texttt{real-sim}) using the same algorithms as in Section~\ref{ssec:lineardiffusion} and ForBES~\cite{SteThePat17}. The latter one applies the semismooth Newton method to the natural residual $F^{\lambda}_{\mathrm{nat}}$ using a forward-backward envelope as the merit function. We use the code provided by the authors\footnote{\url{https://github.com/kul-forbes/ForBES}} and choose $\lambda=1/L$ as the initial value. All other parameters are set to be default values. The results are shown in Figure~\ref{nls-time} and Figure~\ref{nls-lambda}. 

\begin{figure}[t]
  \centering
  
  \subfigure[\texttt{gisette} -- cpu-time.]{
  \includegraphics[width=3.6cm]{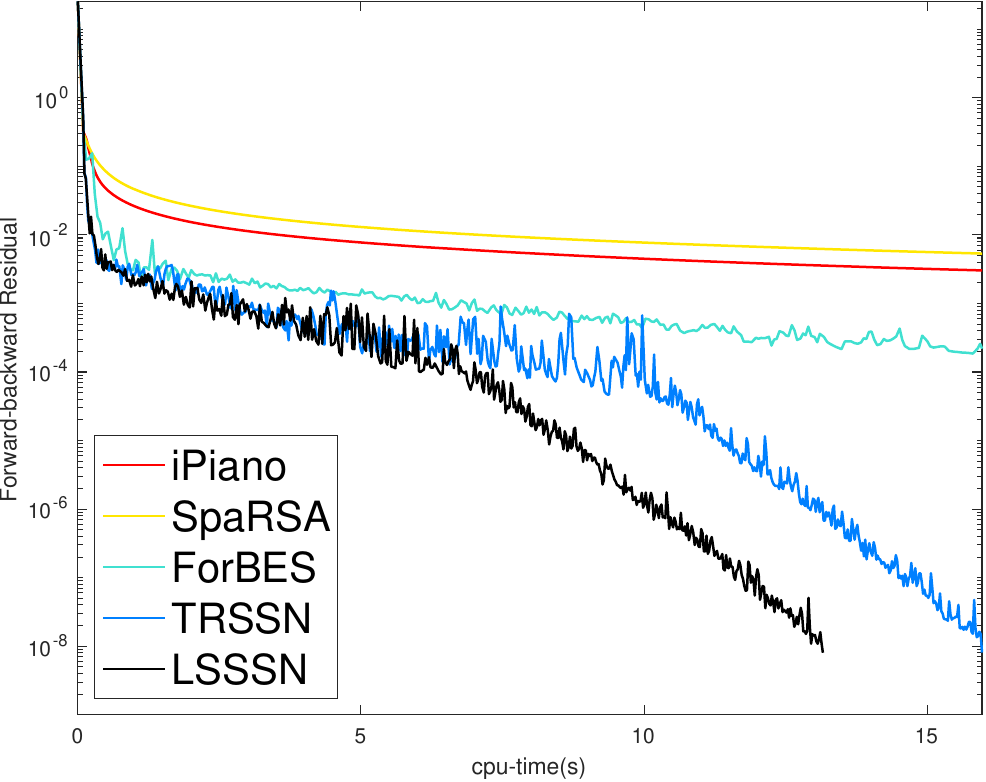}
  }
  \subfigure[\texttt{covtype} -- cpu-time.]{
  \includegraphics[width=3.6cm]{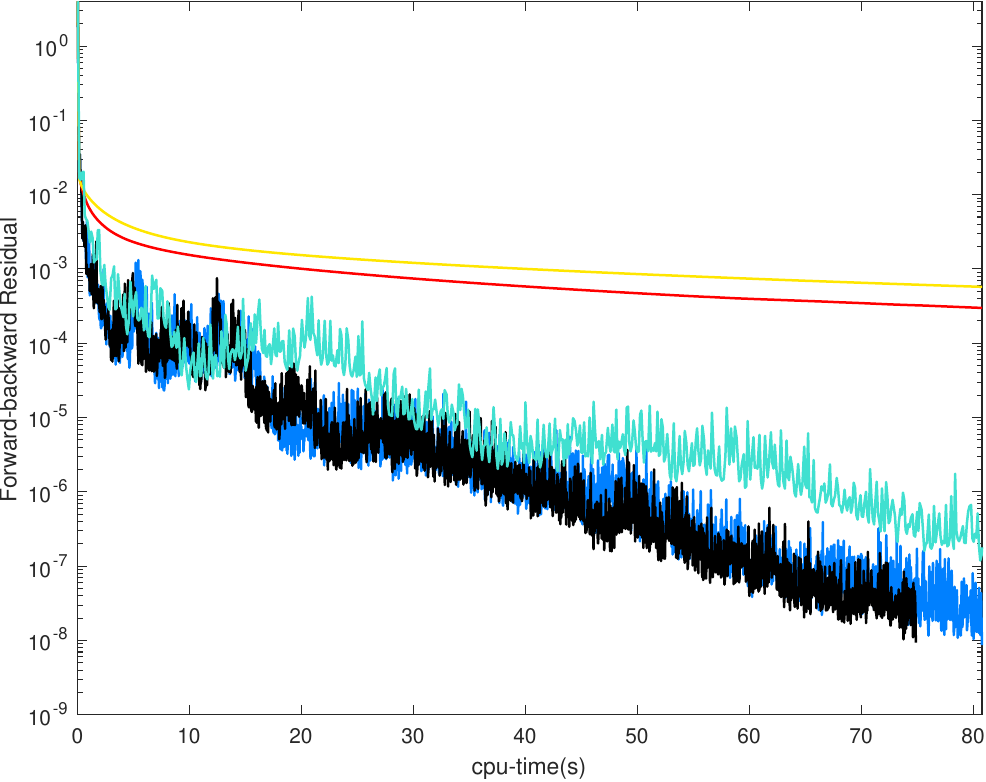}
  }
  \subfigure[\texttt{rcv1} -- cpu-time.]{
  \includegraphics[width=3.6cm]{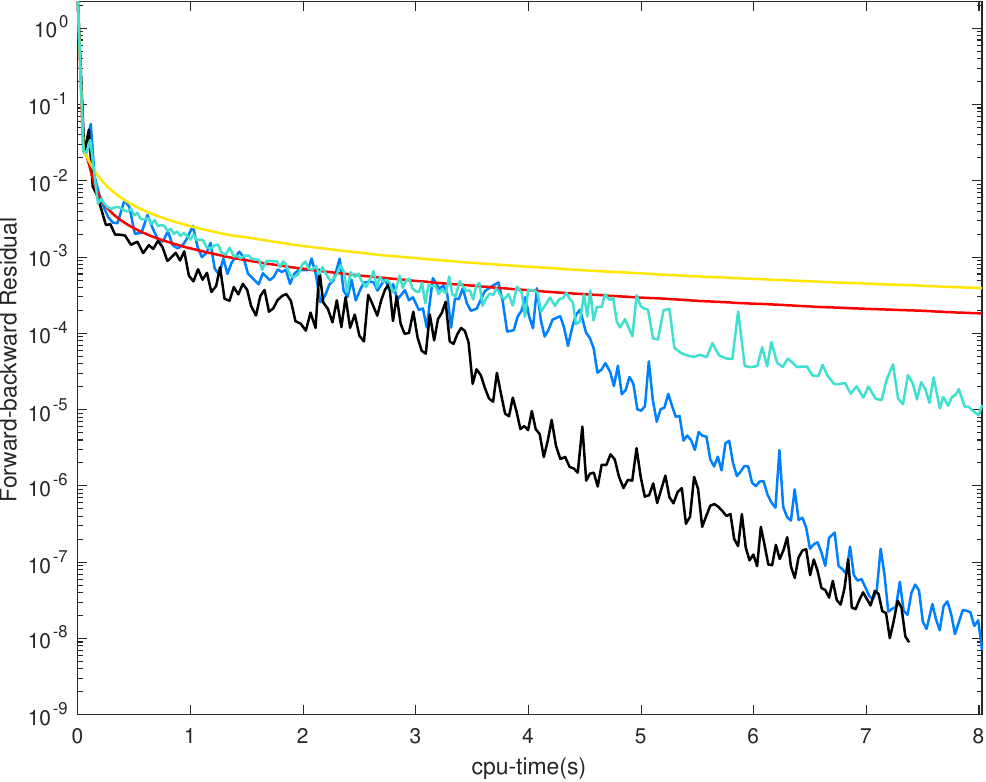}
  }
 \\
  \subfigure[\texttt{CINA} -- cpu-time.]{
  \includegraphics[width=3.6cm]{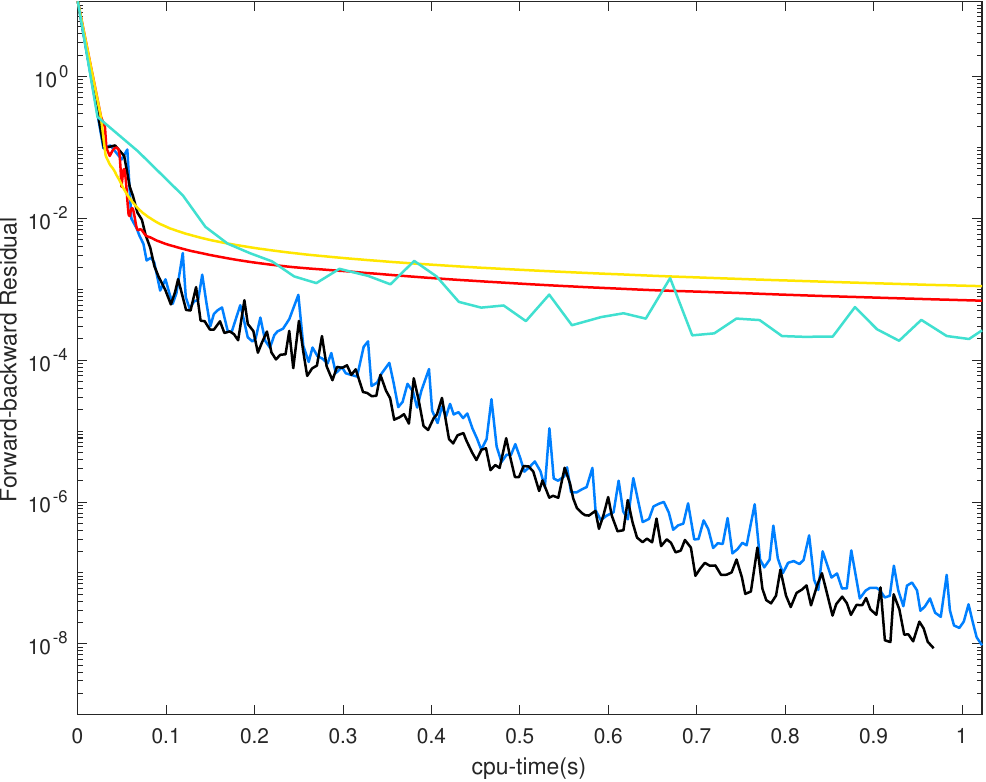}
  }
  \subfigure[\texttt{news20} -- cpu-time.]{
  \includegraphics[width=3.6cm]{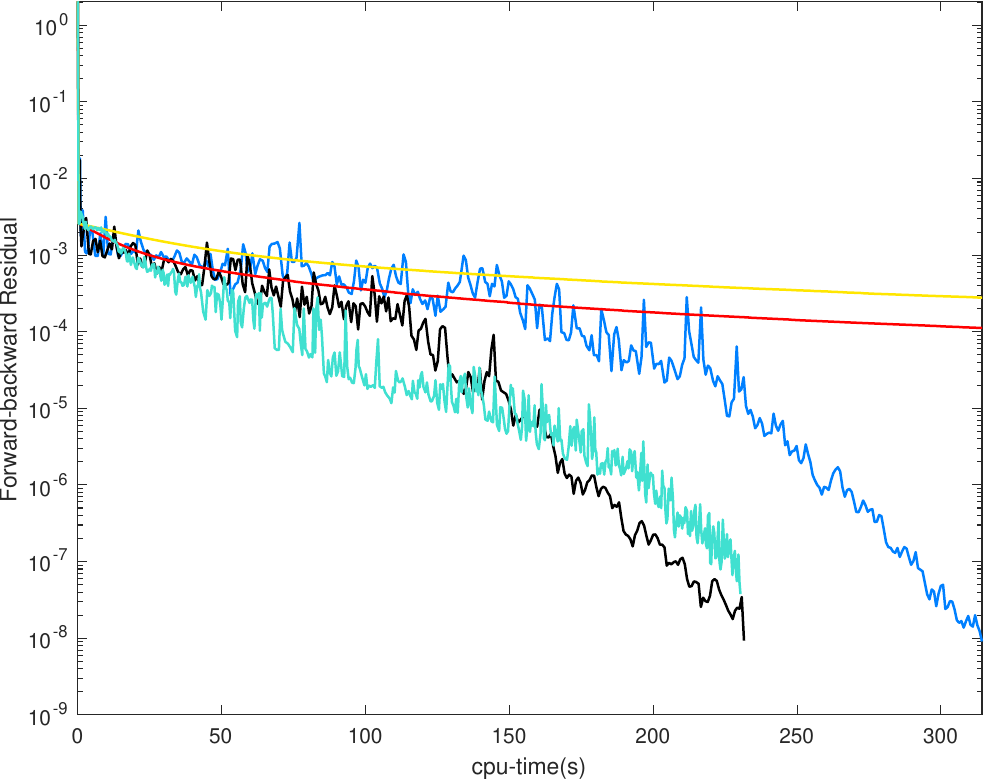}
  }
  \subfigure[\texttt{real-sim} -- cpu-time.]{
  \includegraphics[width=3.6cm]{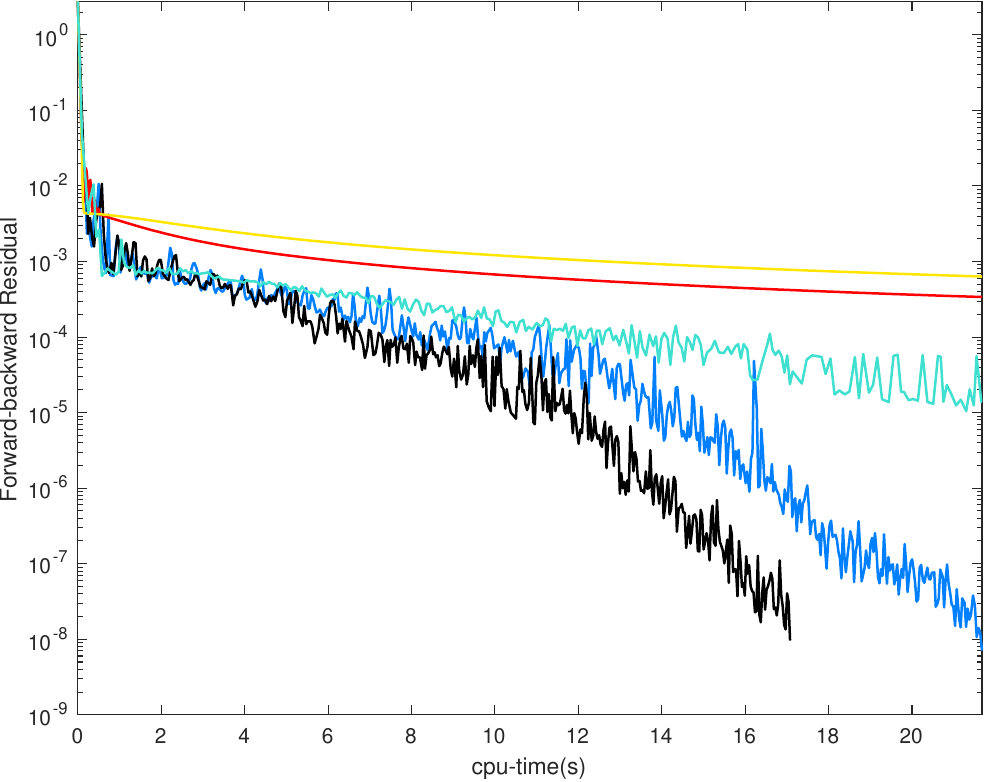}}
  \caption{Comparison of iPiano, ForBES, SpaRSA, TRSSN, and LSSSN on the group sparse least-squares problem \eqref{eq:prob-group}. Plot of the norm of the natural residual w.r.t.\ the cpu-time.}
  \label{nls-time}
  \end{figure}

\begin{figure}[t]
  \centering
  \subfigure[\texttt{gisette} -- cpu-time.]{
  \includegraphics[width=3.6cm]{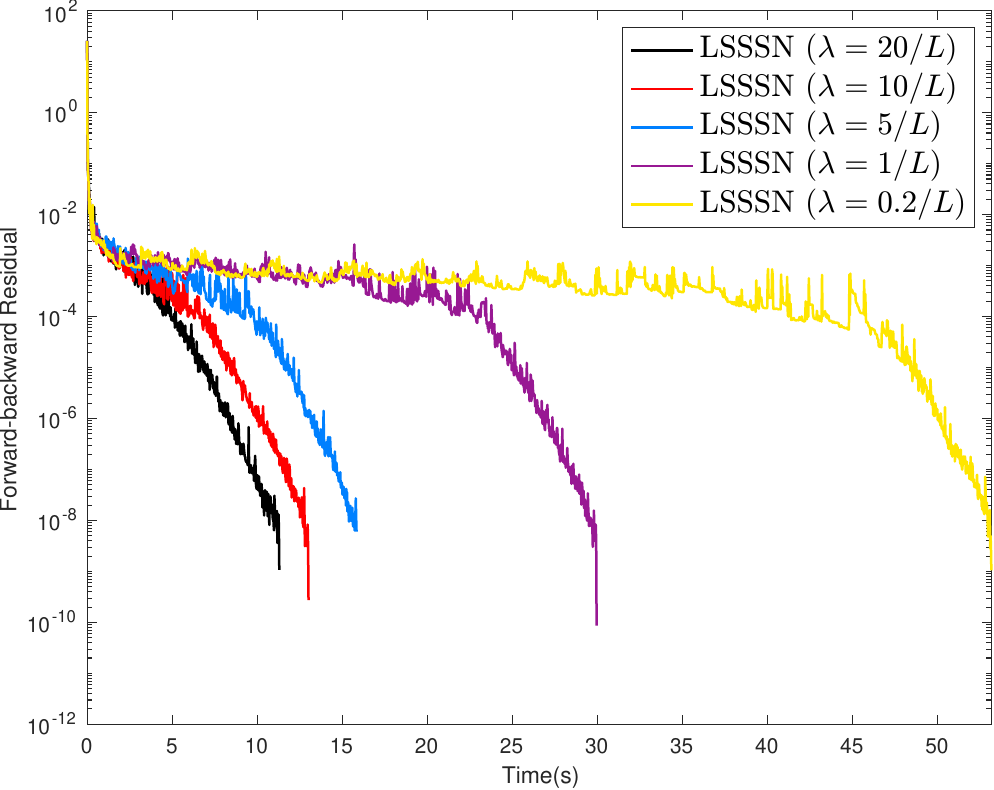}
  }
  \subfigure[\texttt{rcv1} -- cpu-time.]{
  \includegraphics[width=3.6cm]{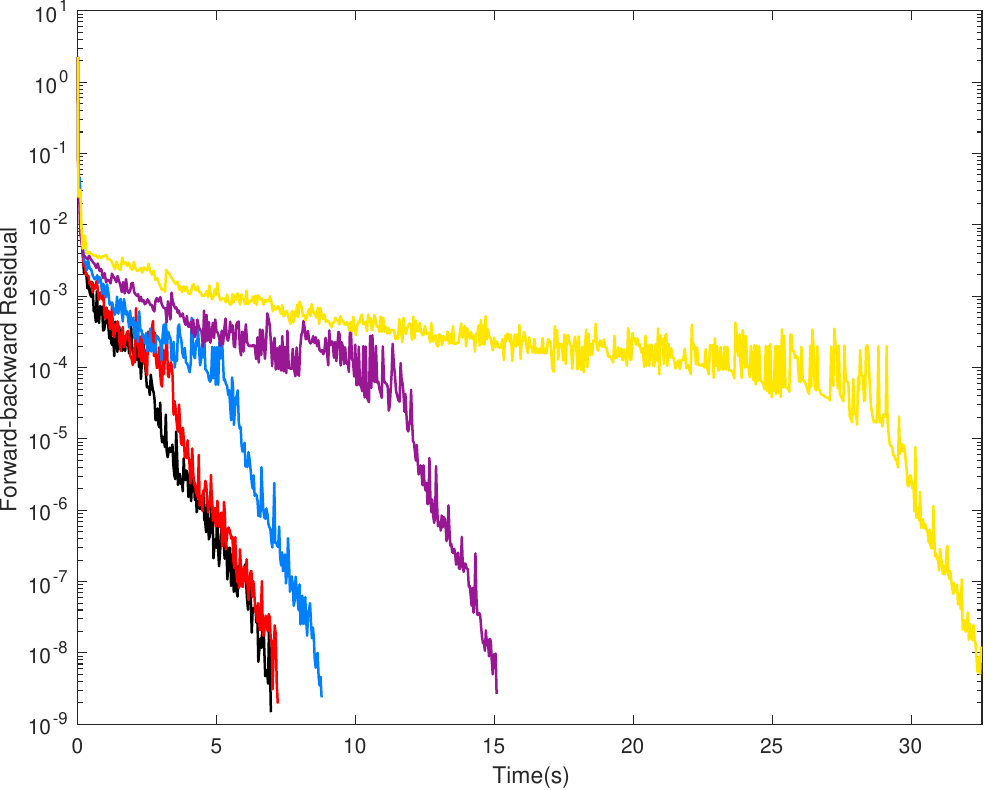}
  }
  \subfigure[\texttt{real-sim} -- cpu-time.]{
  \includegraphics[width=3.6cm]{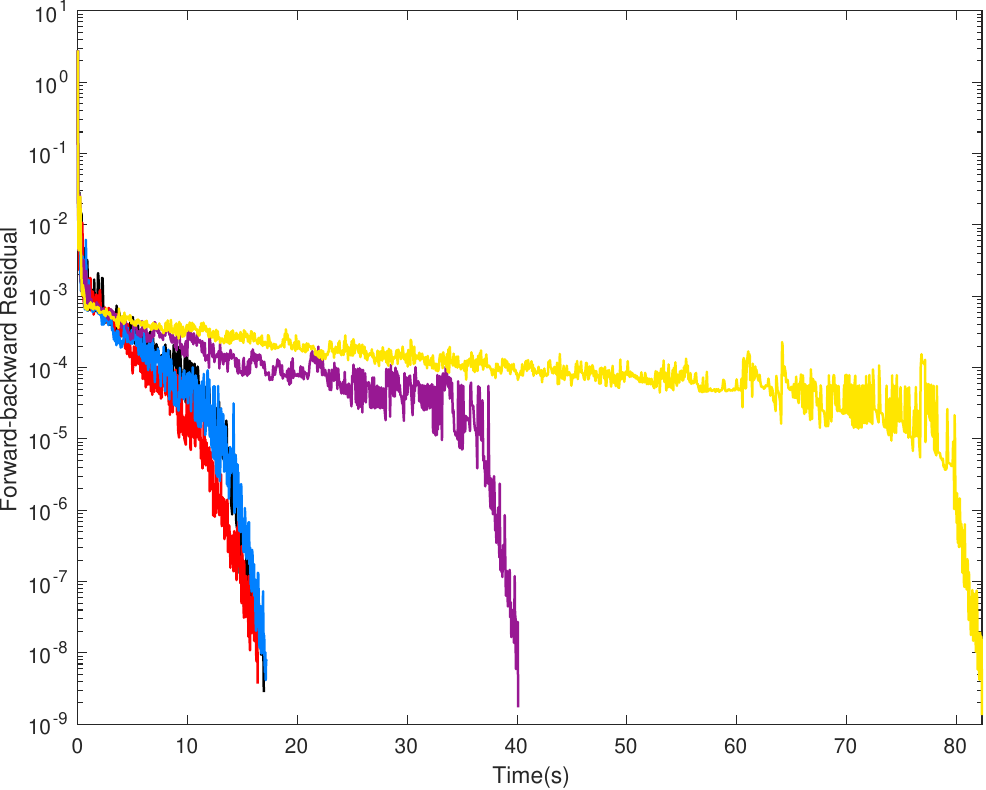}
  }
 \\
  \subfigure[\texttt{gisette} -- Iteration.]{
  \includegraphics[width=3.6cm]{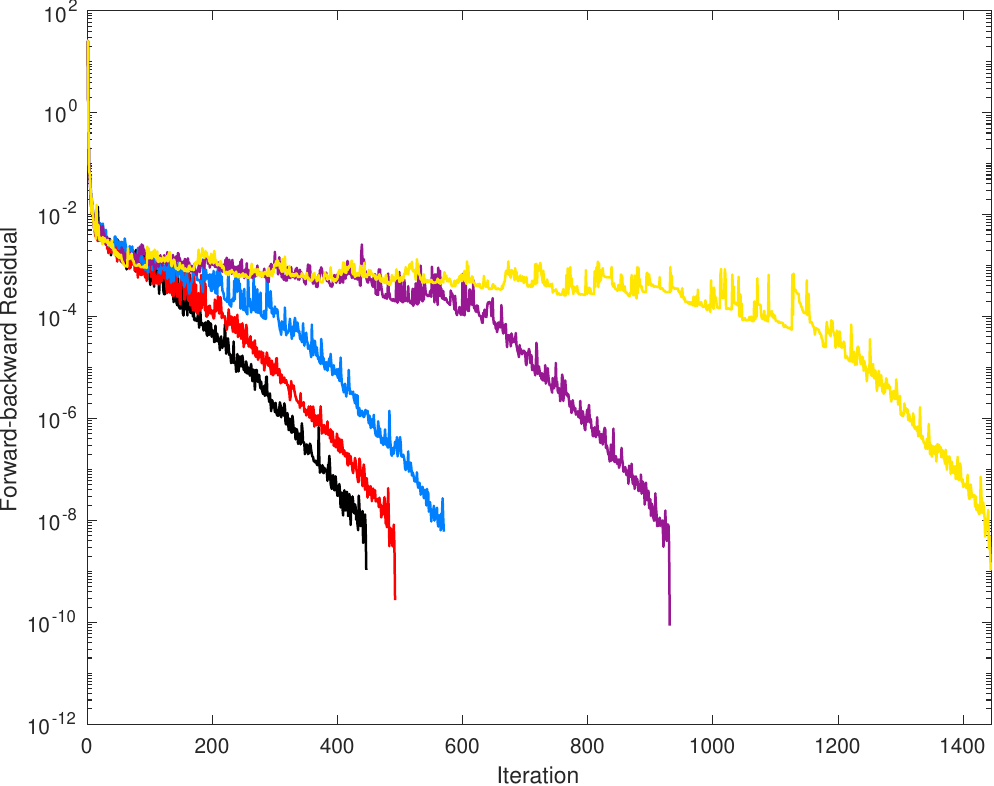}
  }
  \subfigure[\texttt{rcv1} -- Iteration.]{
  \includegraphics[width=3.6cm]{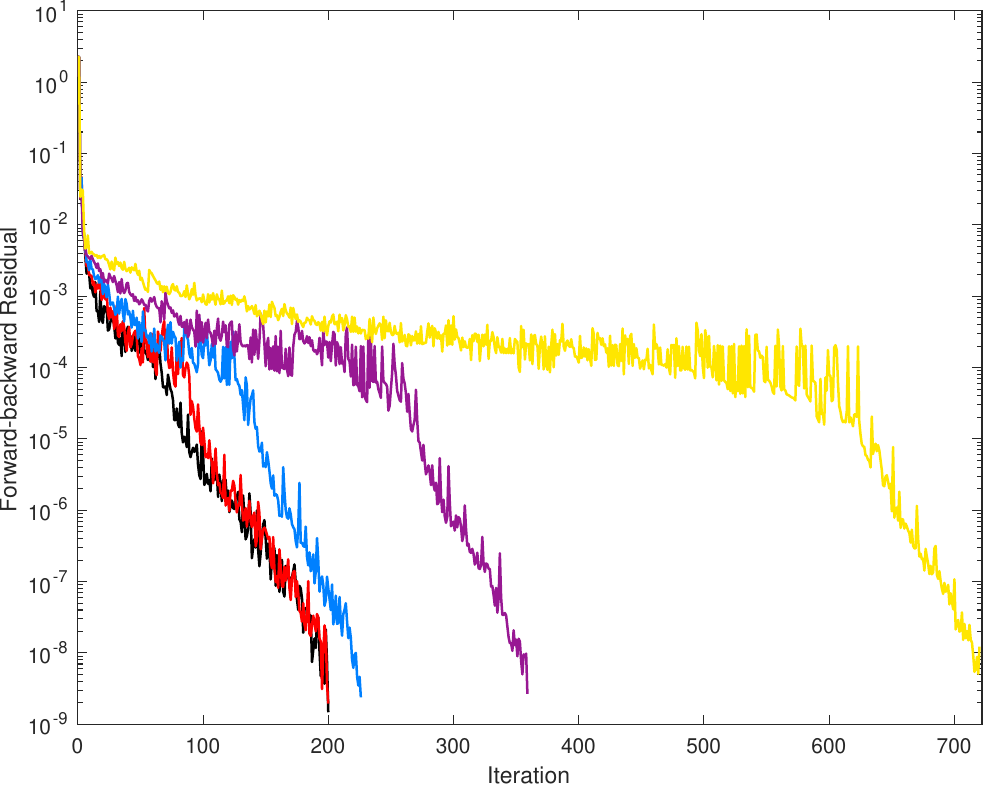}
  }
  \subfigure[\texttt{real-sim} -- Iteration.]{
  \includegraphics[width=3.6cm]{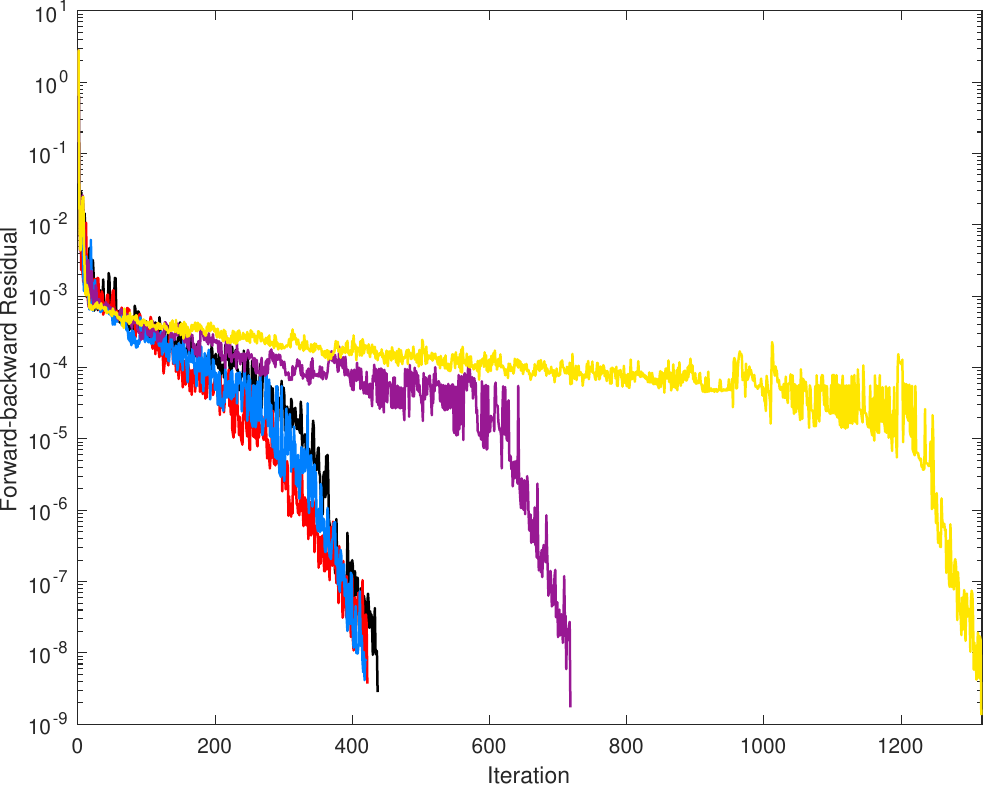}
  }
  \caption{Ablation study of LSSSN using different $\lambda \in \{20/L,10/L,5/L,1/L,0.2/L\}$ on the group sparse least-squares problem \eqref{eq:prob-group} and on the datasets \texttt{gisette}, \texttt{rcv1}, \texttt{real-sim}.}
  \label{nls-lambda}
  \end{figure} 

 Figure~\ref{nls-time} shows that LSSSN outperforms all other algorithms on the datasets \texttt{gisette}, \texttt{rcv1}, and \texttt{real-sim}. During the first iterations, LSSSN, TRSSN, and ForBES usually have a comparable convergence speed, but LSSSN converges faster to a higher precision region ($\leq 10^{-5}$). Notably, the linesearch-based methods, LSSSN and ForBES, demonstrate a better performance than TRSSN on \texttt{news20}. The first-order algorithms typically perform well initially but are eventually outperformed by the tested higher-order methods. In Figure~\ref{nls-lambda}, we also run LSSSN with different choices of $\lambda$ on three datasets. It turns out that larger choices of $\lambda$ tend to yield better and more robust results, while too small values of $\lambda$ can affect the overall performance of the algorithm. Figure~\ref{nls-lambda} illustrates that the performance of LSSSN can be potentially further improved if $\lambda$ is tuned more carefully.
 

\section{Conclusion and Future Work}\label{sec:conclusion}
This work proposes a linesearch-type normal map-based semismooth Newton method (LSSSN) for solving a class of nonsmooth nonconvex optimization problems.  Our approach builds upon the trust region-type semismooth Newton method (TRSSN) \cite{ouyang2025trust} and on the general similarity and difference between linesearch and trust region mechanisms. In addition, we design an estimation strategy to avoid the explicit computation of the Lipschitz constant. The proposed algorithm retains the same convergence properties as its trust region counterpart, accompanied by a refined analysis. Finally, our numerical experiments on sparse logistic regression, an image compression problem, and a nonlinear least square problem with group Lasso penalty demonstrate that LSSSN is comparable with TRSSN and performs well on both convex and nonconvex problems. Looking ahead, one possible future work is to show that every accumulation point is a stationary point for normal map-based approaches assuming only local Lipschitz continuity of $\nabla f$. However, this causes more intricate dynamics of the regularization parameters $\{\tau_k\}$ and will likely necessitate further modifications of our algorithm framework. Furthermore, it might be interesting to study whether the proposed approach can effectively avoid (active) saddle points, see, e.g., \cite{davis2022proximal,jiang2024saddle}. \vspace{1ex}

\section*{Statements and Declarations} \vspace{-1ex}
\noindent\textbf{Funding and/or Conflicts of interests/Competing interests.}  Andre Milzarek was partly supported by the National Natural Science Foundation of China (Foreign Young Scholar Research Fund Project) under Grant No$.$ 12150410304, by the Shenzhen Science and Technology Program (No$.$ RCYX20221008093033010), and by Shenzhen Stability Science Program 2023, Shenzhen Key Lab of Multi-Modal Cognitive Computing. \vspace{.5ex}

\noindent\textbf{Acknowledgments.} We thank the Coordinating Editor and the four anonymous reviewers for their detailed and constructive comments, which have greatly helped to improve the quality and presentation of the manuscript.

\renewcommand*{\bibfont}{\fontsize{8.97pt}{10.27pt}\selectfont}
\setlength{\bibsep}{4pt plus 1ex}
\bibliography{sn-references}


\appendix
\section{The CG Method and Proof of \texorpdfstring{Lemma~\ref{lemma3-8}}{Lemma 2.1}} \label{appen:prooflemma}

\begin{algorithm}
        \caption{The CG Method}
          \label{algo:CG}
        \begin{algorithmic}[1]  
            \Require Input $S=DM$, $g =D\Fnor{z}$, $\epsilon \geq 0$. Set $q_0=0$, $r_0=g$, $p_0 = -g$ and $i=0$.
            \If{$\|r_0\|\leq\epsilon$}
            \State Terminate with $\tilde q = q_0$;
            \EndIf
            \While{$i\leq n-1$}
                    \If{$\iprod{p_i}{Sp_i}\leq0$}
                    \State Terminate with $\tilde q = p_0$ if $i=0$ and $\tilde q = q_i$ if $i \neq 0$;
                    \EndIf
                    \State Set $\alpha_i=\frac{\langle r_i,r_i \rangle}{\langle p_i,Sp_i\rangle}$, $q_{i+1}=q_i+\alpha_ip_i$, and $r_{i+1}=r_i+\alpha_iSp_i$;
                    \If{$\|r_{i+1}\|\leq \epsilon$}
                    \State Terminate with $\tilde q = q_{i+1}$;
                    \EndIf
                    \State Set $\beta_{i+1}=\frac{\|r_{i+1}\|^2}{\|r_i\|^2}$ and $p_{i+1}=-r_{i+1}+\beta_{i+1}p_i$;
                    \State $i\gets i+1$
            \EndWhile
        \end{algorithmic}
\end{algorithm}

\begin{proof}
Denote $S=DM$, $g=D\Fnor{z}$. It is clear that $\mathrm{dim}~\mathrm{range}(S)\le m$, $\mathrm{range}(S) \subset \mathrm{range}(D)$ and $g \in \mathrm{range}(D)$. Let us consider the iteration $i = m-1$ and suppose that the CG method has not terminated by detecting non-positive curvature. Following \cite{HesSti52, NocWri06}, it can be shown that the CG method has generated a sequence of vectors $\{r_0,r_1,\dots,r_{m-1}\}$ and $\{p_0,p_1,\dots,p_{m-1}\}$ with the properties
\be \label{eq:cg-prop} \iprod{p_\ell}{Sp_j} = 0, \;\; \iprod{r_\ell}{p_j} = 0, \;\; \mathrm{span}\{r_0,\dots,r_{\ell}\} = \mathrm{span}\{p_0,\dots,p_{\ell}\} = \mathcal K^\ell(S,r_0), \ee
for all $j = 0,\dots,\ell - 1$ and $\ell = 1,\dots,m-1$. Here, $\mathcal K^\ell(S,r_0)$ denotes the Krylov space $\mathrm{span}\{r_0, Sr_0,\dots,S^\ell r_0\}$. By construction, we have $p_\ell, r_\ell \in \mathrm{range}(D)$ and $\mathcal K^\ell(S,r_0) \subseteq  \mathrm{range}(D)$ for all $\ell = 0,\dots,m-1$. Moreover, due to the $S$-orthogonality and $\iprod{p_\ell}{Sp_\ell} > 0$, $\ell = 0,\dots,m-1$, the vectors $\{p_0,p_1,\dots,p_{m-1}\}$ are linearly independent and hence, it follows $\mathrm{span}\{p_0,p_1,\dots,p_{m-1}\} = \mathrm{range}(D)$. Using the second equality in \eqref{eq:cg-prop}, we can infer $$\mathrm{range}(D) \ni r_m \, \bot \, \mathrm{span}\{p_0,p_1,\dots,p_{m-1}\} = \mathrm{range}(D),$$
which implies $r_m = r_0 + S q_m = 0$. Consequently, the CG method would stop at iteration $m$ with $q = q_m$ satisfying $Sq = - g$. This finishes the proof of the first part. 

For part~(ii), when $S$ is positive semidefinite and $M$ is invertible, we note that
$$\iprod{g}{DMg} \leq 0 \; \iff \; \|[DM]^\frac12 g\| = 0 \; \iff \;  D M g = 0 \; \iff \; M^\top D g = 0.$$
Thus, the condition $\iprod{p_i}{Sp_i} \leq 0$ implies  $Dp_i = 0$. Combining $-\iprod{r_0}{p_0}=\|r_0\|^2$, $p_i=-r_i+\beta_i p_{i-1}$,  and \eqref{eq:cg-prop}, we have $\|r_i\|^2 = -\iprod{r_i}{p_i}$. By $r_i=r_0 + \sum_{j=0}^{i-1}\alpha_j S p_j$ and \eqref{eq:cg-prop}, it then holds that $-\iprod{r_i}{p_i} = -\iprod{r_0}{p_i} = - \iprod{\Fnor{z}}{Dp_i} = 0$.
Hence, the CG method would exit at the $(i-1)$-th iteration when checking the norm of the residual $\|r_i\|$ instead of detecting non-positive curvature at the $i$-th iteration and therefore the method cannot terminate at step 6 of Algorithm~\ref{algo:CG}. Consequently, Algorithm~\ref{algo:CG} can only terminate in step 10 which completes the proof.

The proof for part~(iii) is identical to the derivation of \cite[Lemma 3.1]{ouyang2025trust}.
\end{proof}

\end{document}